\documentclass[review,3p,square,sort,compress]{elsarticle}

\usepackage{amssymb}
\usepackage{amsmath}
\usepackage{amsthm}
\usepackage{stmaryrd}
\usepackage{bbm}
\usepackage{enumitem}

\usepackage{fancyhdr}
\usepackage{hyperref}
\hypersetup{%
  pdftitle={Doubly RBSDEs},%
  pdfsubject={DRBSDEs},%
  pdfauthor={ShengJun FAN},%
  pdfkeywords={BSDEs},%
  pdfstartview=FitH,%
  CJKbookmarks=true,%
  bookmarksnumbered=true,%
  bookmarksopen=true,%
  colorlinks=true, linkcolor=blue, urlcolor=blue, citecolor=blue, %
}

\usepackage[nameinlink]{cleveref}
\crefname{thm}{Theorem}{Theorems}
\crefname{pro}{Proposition}{Propositions}
\crefname{lem}{Lemma}{Lemmas}
\crefname{rmk}{Remark}{Remarks}
\crefname{cor}{Corollary}{Corollaries}
\crefname{dfn}{Definition}{Definitions}
\crefname{ex}{Example}{Examples}
\crefname{section}{Section}{Sections}
\crefname{subsection}{Subsection}{Subsections}

\newcommand{\To}{\rightarrow}
\newcommand{\as}{{\rm d}\mathbb{P}\times{\rm d} t-a.e.}

\newcommand{\ps}{\mathbb{P}-a.s.}

\newcommand{\F}{\mathcal{F}}
\newcommand{\E}{\mathbb{E}}

\newcommand{\s}{\mathcal{S}}
\newcommand{\M}{{\rm M}}
\newcommand{\hcal}{\mathcal{H}}
\newcommand{\vcal}{\mathcal{V}}
\newcommand{\mcal}{\mathcal{M}}

\newcommand{\T}{[0,T]}
\newcommand{\Lp}{{\mathbb L}^p(\F_T)}

\newcommand{\R}{{\mathbb R}}

\newcommand{\RE}{\forall}

\newcommand {\Lim}{\lim\limits_{n\rightarrow\infty}}
\newcommand {\Dis}{\displaystyle}

\newtheorem{thm}{Theorem}[section]
\newtheorem{lem}[thm]{Lemma}
\newtheorem{pro}[thm]{Proposition}
\newtheorem{rmk}[thm]{Remark}
\newtheorem{cor}[thm]{Corollary}
\newtheorem{dfn}[thm]{Definition}
\newtheorem{ex}[thm]{Example}

\begin{document}
\begin{frontmatter}

\title{{\boldmath\bf
$L^p$ solutions of doubly reflected BSDEs \\ under general assumptions}\tnoteref{found}}
\tnotetext[found]{Supported by the Fundamental
Research Funds for the Central Universities (No.\,2017XKZD11).\vspace{0.1cm}}


\author{Shengjun FAN\corref{cor1}\vspace{0.1cm}}%
\author{\ Qianyun QIAN}

\cortext[cor1]{\ Corresponding author. E-mail: f\_s\_j@126.com.}

\address{School of Mathematics, China University of Mining and Technology, Xuzhou, Jiangsu, 221116, PR China.\vspace{-0.5cm}}

\begin{abstract}
Under a generalized Mokobodzki condition for reflected BSDEs with two continuous barriers which relates the growth of the generator $g$ and that of the barriers, we establish several existence and uniqueness results on $L^p\ (p>1)$ solutions of doubly reflected BSDEs with generators satisfying a one-sided Osgood condition together with a general growth in the state variable $y$, and a uniform continuity condition or a linear growth condition in the state variable $z$. This Mokobodzki condition is also proved to be necessary for existence of the $L^p$ solutions. And, we prove that the $L^p$ solutions can be approximated by the penalization method and by some sequences of the $L^p$ solutions of doubly reflected BSDEs. \vspace{0.2cm}
\end{abstract}

\begin{keyword}
Doubly reflected backward stochastic differential equation \sep Existence and uniqueness\sep \\
\hspace{1.8cm} Comparison theorem\sep Penalization method\sep
$L^p$ solution\vspace{0.2cm}

\MSC[2010] 60H10, 60H30
\end{keyword}

\end{frontmatter}

\section{Introduction}
\label{sec:1-Introduction}
\setcounter{equation}{0}

Backward stochastic differential equations (BSDEs for short) were first introduced in linear case by \citet{Bismut1973JMAA} in 1973, and extended to a fully nonlinear version at the first time by \citet{PardouxPeng1990SCL} in 1990. Later on, as a variation of the notion of nonlinear BSDEs, nonlinear reflected BSDEs (RBSDEs for short) with one and two continuous barriers were introduced by \citet{ElKarouiKapoudjianPardouxPengQuenez1997AoP} and \citet{CvitanicKaratzas1996AoP} respectively. At present it has been widely recognized that these equations have natural connections with many problems in different mathematical fields, such as partial differential equations, mathematical finance, stochastic control and game theory, optimal switching problem and other optimality problems and others (see, e.g. \cite{BayraktarYao2012SPA,BayraktarYao2015SPA,ElKarouiKapoudjianPardouxPengQuenez1997AoP,
ElKarouiPardouxQuenez1997NMIF,ElKarouiPengQuenez1997MF,Hamadene2002Stochastics,
HamadeneLepeltier2000SPA,HamadeneLepeltierWu1999PMS,
HamadeneZhang2010SPA,HuTang2010PTRF,Jia2010SPA,MaZhang2005SPA,
Pardoux1999NADEC,Peng1997BSDEP,Peng1999PTRF,Peng2004LNM,
PengXu2005AIHPPS,PengXu2010Bernoulli,RosazzaEmanuela2006IME}, etc.), and they provide a very useful and efficient tool for studying these problems.\vspace{0.2cm}

In \citet{PardouxPeng1990SCL}, \citet{ElKarouiKapoudjianPardouxPengQuenez1997AoP} and \citet{CvitanicKaratzas1996AoP}, the existence and uniqueness result of $L^2$ solutions of non-reflected BSDEs and RBSDEs with one and two continuous barriers are established under the standard assumption that the generator $g$ satisfies the linear growth condition and is Lipschitz continuous with respect to the state variables $y$ and $z$. Many attempts have been made to relax these assumptions, which are too strong for many interesting applications mentioned above. For example, some works were devoted to solving RBSDEs with less regular barriers, see \cite{BayraktarYao2015SPA,Hamadene2002Stochastics,HamadeneHassani2005PTRF,
HamadeneHassaniOuknine2010BSM,Klimsiak2013BSM,LepeltierXu2005SPL,PengXu2005AIHPPS}, etc; many scholars were interested in the existence and uniqueness of solutions for non-reflected BSDEs and RBSDEs with data that are in the spaces of $L^p\ (p>1)$ and $L^1$, see \cite{BriandDelyonHu2003SPA,BriandLepetierSanMrtin2007Bernoulli,
ElAsriHamadeneWang2011SAA,ElKarouiPengQuenez1997MF,Fan2015JMAA,Izumi2013SPL,
MaFanSong2013BSM} for non-reflected BSDEs,
and \cite{Aman2009ROSE,BayraktarYao2012SPA,HamadenePopier2012SD,Klimsiak2012EJP,
Klimsiak2013BSM,RozkoszSlominski2012SPA} for RBSDEs; and more papers focused their attention on weakening the linear growth condition and Lipschitz-continuity condition of the generator $g$ with respect to the state variables $y$ and $z$, see \cite{BriandCarmona2000JAMSA,BriandDelyonHu2003SPA,
BriandLepetierSanMrtin2007Bernoulli,ElKarouiPengQuenez1997MF,Fan2015JMAA,
FanJiang2010CRA,FanJiang2012JoTP,FanJiang2012SPL,FanJiangDavison2013FMC,
HuTang2016SPA,Izumi2013SPL,Jia2008CRA,Jia2010SPA,MaFanSong2013BSM,Mao1995SPA,
Pardoux1999NADEC} for non-reflected BSDEs,
and \cite{Aman2009ROSE,BayraktarYao2012SPA,BayraktarYao2015SPA,EssakyHassani2011BSM,
EssakyHassani2013JDE,Fan2017AMS,HamadeneHdhiri2006JAMSA,
HamadeneLepeltierMatoussi1997BSDEs,HamadeneZhang2010SPA,HuaJiangShi2013JKSS,
JiaXu2014arXiv,Klimsiak2012EJP,Klimsiak2013BSM,KobylanskiLepeltierQuenezTorres2002PMS,
LepeltierMatoussiXu2005AdvinAP,LepeltierSanMartin2004JAP,Matoussi1997SPL,
RozkoszSlominski2012SPA,Xu2008SPA,ZhengZhou2008SPL} for RBSDEs.\vspace{0.2cm}

The present paper is the continuation along the last two research directions, which is devoted to solving RBSDEs with two continuous barriers in the space of $L^p\ (p>1)$ under general assumptions. The generator $g$ of doubly reflected BSDEs only needs to satisfy a one-sided Osgood condition together with a general growth in the state variable $y$, and a uniform continuity condition or a linear growth condition in the state variable $z$ (see assumptions \ref{A:(H1)}, \ref{A:(H3)}, \ref{A:(H2)} and \ref{A:(H2')} in \cref{sec:2-NotationsAssumptions} respectively), which are weaker than those in many existing works. A generalized Mokobodzki condition (see assumption \ref{A:(H4)} in \cref{sec:2-NotationsAssumptions}) for doubly reflected BSDEs which relates the growth of the generator $g$ and that of the barriers is put forward and proved to be necessary and sufficient for existence of the $L^p$ solutions. Based on a combination between existing methods, their refinement and perfection, and some new ideas, we prove that the $L^p$ solutions of the doubly reflected BSDEs can be approximated by the penalization method and by some sequences of the $L^p$ solutions of doubly reflected BSDEs. In particular, we also consider the case that the generator $g$ may be discontinuous and have a general growth in the state variable  $y$.\vspace{0.2cm}

The rest of this paper is organized as follows. \cref{sec:2-NotationsAssumptions} contains some notations, definitions, assumptions together with some lemmas and propositions which will be frequently used later. In \cref{sec:3-NecessityUniqueness}, by \cref{thm:3-Necessity} we verify necessity of the generalized Mokobodzki condition to ensure existence of the $L^p$ solutions. Then, we prove a general comparison theorem of the $L^p$ solutions for RBSDEs with two continuous barriers, which naturally yields uniqueness of the $L^p$ solutions under assumptions \ref{A:(H1)} and \ref{A:(H2)} (ii), see \cref{pro:3-ComparisonTheoremofDRBSDE} and \cref{thm:3-UniquenessOfSolutions}. In \cref{sec:4-ExistencePenalization}, by \cref{pro:4-Penalization} we show the convergence of $L^p$ solutions of the penalization equations for RBSDEs with one continuous barrier under some elementary conditions. And, by \cref{pro:4-EstimateOfPenalizationEq} we establish an a priori uniform estimate on the $L^p$ solutions of the penalization equations for RBSDEs with one continuous barrier and non-reflected BSDEs under assumptions \ref{A:(H1)}-\ref{A:(H4)}. Then, based on Propositions \ref{pro:3-ComparisonTheoremofDRBSDE},  \ref{pro:4-Penalization} and \ref{pro:4-EstimateOfPenalizationEq} together with \cref{rmk:4-RemarkOfKeyEstimate}, by \cref{thm:4-ExistenceUnder(H2)} (resp. \cref{thm:4-ExistenceUnder(H2')}) we prove existence and uniqueness of the $L^p$ solutions (resp. existence of the maximal and minimal $L^p$ solutions) for the doubly reflected BSDEs under assumptions \ref{A:(H1)}, \ref{A:(H2)} (resp. \ref{A:(H2')}), \ref{A:(H3)} and \ref{A:(H4)}. In \cref{sec:5-Stability}, by \cref{pro:5-Approximation} we establish a general approximation result for the $L^p$ solutions of the doubly reflected BSDEs under some elementary assumptions, and based on it we prove an existence result of the minimal (maximal) $L^p$ solutions for the doubly reflected BSDEs under several weaker assumptions, see \cref{thm:5-ExistenceOfDRBSDEunder(H2')(A1a)}, where the generator $g$ has a general growth in $y$ and a linear growth in $z$, but it is interesting that $g$ may be discontinuous in $y$ as considered in \citet{FanJiang2012SPL} and \citet{ZhengZhou2008SPL}, see assumptions \ref{A:(A1a)} and \ref{A:(A1b)} in \cref{sec:5-Stability}. Finally, at the end of Sections \ref{sec:4-ExistencePenalization} and \ref{sec:5-Stability}, we
provide several examples which the results of this paper can (but any existing works can not) be applied to, and give some remarks to illustrate further our theoretical results, see Remarks \ref{rmk:4-1} and \ref{rmk:5-1} for more details. And, some details for the proof of Propositions \ref{pro:4-Penalization} and \ref{pro:5-Approximation} are provided in Appendix.\vspace{0.2cm}

At the end of the introduction, we would like to mention that the results obtained in this paper improve considerably some known works. In particular, these results extend the corresponding results in \citet{Fan2017AMS} for RBSDEs with one continuous barrier to the case of RBSDEs with two continuous barriers, and strengthen the corresponding results in \citet{Klimsiak2013BSM}, where the generator $g$ needs to satisfy some stronger assumptions (see assumptions \ref{A:(H1s)} and \ref{A:(H2s)} in \cref{sec:2-NotationsAssumptions}) than ours.

\section{Preliminaries}
\label{sec:2-NotationsAssumptions}
\setcounter{equation}{0}

\subsection{Notations\vspace{0.1cm}}

In the whole paper we fix a real number $T>0$ and a positive integer $d$, and let $(B_t)_{t\in\T}$ be a
standard $d$-dimensional Brownian motion defined on some complete filtered probability space $(\Omega,\F,\mathbb{P})$, where $(\F_t)_{t\in\T}$ is the completed $\sigma$-algebra filtration generated by $(B_t)_{t\in\T}$ and $\F=\F_T$. We assume that if there is not a special illustration, all processes of this paper are defined on $\Omega\times\T$, all notions whose definitions are related to some filtration are understood with respect to the filtration $(\F_t)_{t\in\T}$, and all equalities and inequalities between random elements are understood to hold $\ps$ To avoid ambiguity we stress that writing $X_t=Y_t,\ t\in\T$ we mean that
$$\mathbb{P}(\{\omega:X_t(\omega)=Y_t(\omega),\ t\in\T\})=1,$$
while writing $X_t=Y_t$ for each $t\in\T$
we mean that
$$\RE\ t\in\T,\ \ \mathbb{P}(\{\omega:X_t(\omega)=Y_t(\omega)\})=1.$$
It is clear that they are equivalent if $(X_t)_{t\in\T}$ and $(Y_t)_{t\in\T}$ are both continuous processes. In what follows, the variable $\omega$ in random elements is often omitted as usually done.\vspace{0.2cm}

Denote $\R_+:=[0,+\infty)$, $a^+:=\max \{a,0\}$ and $a^-:=(-a)^+$ for any real number $a$. For a set $A$, we denote by $A^c$ the complement of $A$ and by $\mathbbm{1}_{A}$ the indicator function of $A$. Let ${\rm sgn}(x)$ represent the sign of a real number $x$. For $n\geq 1$, the Euclidean norm of an element $y\in \R^{n}$ will be denoted by $|y|$.\vspace{0.2cm}

For $p>0$, we define the following spaces:
\begin{itemize}
\item $\Lp$ the set of all $\F_T$-measurable real-valued random variables $\xi$ satisfying
$$\|\xi\|_{\mathbb{L}^p}:=\left(\E[|\xi|^p]\right)^{1\wedge 1/p}<+\infty;\vspace{-0.2cm}$$
\item $\hcal$  the set of all progressively measurable real-valued processes $X_\cdot$ satisfying
    $$\mathbb{P}\left(\int_0^T|X_t|{\rm d}t<+\infty\right)=1;\vspace{-0.3cm}$$
\item $\hcal^p$ the set of all processes $X_\cdot\in \hcal$ satisfying
    $
    \|X\|_{\hcal^p}:=\left\{ \E\left[\left(\int_0^T |X_t|{\rm d}t\right)^p\right] \right\}^{1\wedge 1/p}<+\infty;
    $
\item $\s$  the set of all progressively measurable and continuous real-valued processes;
\item $\s^p$ the set of all processes $Y_\cdot\in \s$ satisfying
     $\|Y\|_{{\s}^p}:=\left(\E[\sup_{t\in\T} |Y_t|^p]\right)^{1\wedge 1/p}<+\infty;$
\item $\M$ the set of all progressively measurable $\R^d$-valued processes $Z_\cdot$ satisfying
    $$\mathbb{P}\left(\int_0^T|Z_t|^2{\rm d}t<+\infty\right)=1;\vspace{-0.3cm}$$
\item $\M^p$ the set of all processes $Z_\cdot\in \M$ satisfying
    $
    \|Z\|_{\M^p}:=\left\{ \E\left[\left(\int_0^T |Z_t|^2{\rm d}t\right)^{p/2}\right] \right\}^{1\wedge 1/p}<+\infty;
    $
\item $\mcal$ the set of all continuous real-valued local martingales;
\item $\mcal^p$ the set of all martingales $M_\cdot\in \mcal$ satisfying $\E\left[\left(\langle M\rangle_T\right)^{p/2}\right]<+\infty$;
\item $\vcal$ the set of all progressively measurable and continuous real-valued processes of finite variation;
\item $\vcal^+$ the set of all increasing processes $V_\cdot\in \vcal$ valued $0$ at $0$;
\item $\vcal^p$ the set of all processes $V_\cdot\in \vcal$ satisfying $\E\left[|V|^p_T\right]<+\infty$;
\item $\vcal^{+,p}$ the set of all processes $V_\cdot\in \vcal^+$ satisfying $\E\left[|V|^p_T\right]<+\infty$.
\end{itemize}
Here and hereafter, for each $V_\cdot\in\vcal$ and $t\in\T$, $|V|_{t,T}$ denotes the random finite variation of $V_\cdot$ on the interval $[t,T]$, and $|V|_{0,T}$ is denoted simply by $|V|_T$. Clearly, if $V_\cdot\in\vcal^{+}$, then $|V|_{t,T}=V_T-V_t$.

\subsection{Definitions\vspace{0.1cm}}

We now recall a definition used in \citet{EssakyHassani2013JDE}.

\begin{dfn}\label{dfn:2-Singular}
For any two processes $K^1_\cdot$ and $K^2_\cdot$ in $\vcal^1$, we say that
\begin{itemize}
\item ${\rm d}K^1\bot {\rm d}K^2$ if and only if there exists a progressively measurable set $D\subset \Omega\times\T$ such that
    $$\E\left[\int_0^T \mathbbm{1}_{D}(t,\omega)\ {\rm d}K^1_t(\omega)\right]=\E\left[\int_0^T \mathbbm{1}_{D^c}(t,\omega)\ {\rm d}K^2_t(\omega)\right]=0.$$
\item ${\rm d}K^1\leq {\rm d}K^2$ if and only if for each progressively measurable set $D\subset \Omega\times\T$,
    $$\E\left[\int_0^T \mathbbm{1}_{D}(t,\omega)\ {\rm d}K^1_t(\omega)\right]\leq \E\left[\int_0^T \mathbbm{1}_{D}(t,\omega)\ {\rm d}K^2_t(\omega)\right],\ \ i.e.,\  K^1_t-K^1_s\leq K^2_t-K^2_s,\ 0\leq s\leq t\leq T.$$
\end{itemize}
\end{dfn}

In the rest of this paper, we always assume that $\xi$ is an $\F_T$-measurable random variable, $V_\cdot\in\vcal$, $L_\cdot\in \s$ (or $L_\cdot\equiv -\infty$) and $U_\cdot\in \s$ (or $U_\cdot\equiv +\infty$) with $L_\cdot\leq U_\cdot$, and that a random function
$$
g(\omega,t,y,z):\Omega\times\T\times\R\times\R^d
\longmapsto \R
$$
is progressively measurable for each $(y,z)$, which is usually called a generator.

\begin{dfn}
\label{dfn:2-DefinitionOfBSDERBSDEDRBSDE}
By a solution to BSDE $(\xi,g+{\rm d}V)$ we understand a pair of progressively measurable processes $(Y_t,Z_t)_{t\in\T}$ belonging to $\s\times \M$ such that $\ps$,
\begin{equation}\label{eq:2-BSDE}
Y_t=\xi+\int_t^Tg(s,Y_s,Z_s){\rm d}s+\int_t^T{\rm d}V_s-\int_t^TZ_s\cdot {\rm d}B_s,\ \ t\in\T.
\end{equation}
By a solution to \underline{R}BSDE $(\xi,g+{\rm d}V,L)$ we understand a triple of progressively measurable processes $(Y_t,Z_t,K_t)_{t\in\T}$ belonging to $\s\times \M\times \vcal^{+,1}$ such that $\ps$,
\begin{equation}\label{eq:2-RBSDEwithLowerBarrier}
\left\{
\begin{array}{l}
\Dis Y_t=\xi+\int_t^Tg(s,Y_s,Z_s){\rm d}s+\int_t^T{\rm d}V_s+\int_t^T{\rm d}K_s-\int_t^TZ_s\cdot {\rm d}B_s,\ \   t\in\T,\vspace{0.2cm}\\
\Dis L_t\leq Y_t,\ \ t\in\T\ \ {\rm and}\ \ \int_0^T (Y_t-L_t){\rm d}K_t=0.
\end{array}
\right.
\end{equation}
By a solution to $\bar{R}$BSDE $(\xi,g+{\rm d}V,U)$ we understand a triple of progressively measurable processes $(Y_t,Z_t,A_t)_{t\in\T}$ belonging to $\s\times \M\times \vcal^{+,1}$ such that $\ps$,
\begin{equation}\label{eq:2-RBSDEwithSuperBarrier}
\left\{
\begin{array}{l}
\Dis Y_t=\xi+\int_t^Tg(s,Y_s,Z_s){\rm d}s+\int_t^T{\rm d}V_s-\int_t^T{\rm d}A_s-\int_t^TZ_s\cdot {\rm d}B_s,\ \ t\in\T,\vspace{0.2cm}\\
\Dis Y_t\leq U_t,\ \ t\in\T\ \ {\rm and}\ \ \int_0^T (U_t-Y_t){\rm d}A_t=0.
\end{array}
\right.
\end{equation}
By a solution to DRBSDE $(\xi,g+{\rm d}V,L,U)$ we understand a quadruple of progressively measurable processes $(Y_t,Z_t,K_t,A_t)_{t\in\T}$ belonging to $\s\times \M\times \vcal^{+,1}\times \vcal^{+,1}$ such that $\ps$,
\begin{equation}\label{eq:2-DRBSDE}
\left\{
\begin{array}{l}
\Dis Y_t=\xi+\int_t^Tg(s,Y_s,Z_s){\rm d}s+\int_t^T{\rm d}V_s+\int_t^T{\rm d}K_s-\int_t^T{\rm d}A_s-\int_t^TZ_s\cdot {\rm d}B_s,\ \   t\in\T,\vspace{0.2cm}\\
\Dis L_t\leq Y_t\leq U_t,\ \ t\in\T,\ \int_0^T (Y_t-L_t){\rm d}K_t=\int_0^T (U_t-Y_t){\rm d}A_t=0\ \ {\rm and}\ \  {\rm d}K\bot {\rm d}A.
\end{array}
\right.\vspace{0.3cm}
\end{equation}
\end{dfn}

\begin{rmk}
\label{rmk:2-ConnectionOfBSDEandRBSDEandDRBSDE}
It is not hard to verify the following assertions.
\begin{itemize}
\item [(i)] The claim that $(Y_t,Z_t,K_t)_{t\in\T}$ is a solution of \underline{R}BSDE $(\xi,g+{\rm d}V,L)$  is equivalent to the claim that
$(Y_t,Z_t,K_t,0)_{t\in\T}$ is a solution of DRBSDE $(\xi,g+{\rm d}V,L,+\infty)$.

\item [(ii)] The claim that $(Y_t,Z_t,A_t)_{t\in\T}$ is a solution of $\bar{R}$BSDE $(\xi,g+{\rm d}V,U)$ is equivalent to the claim that
$(Y_t,Z_t,0,A_t)_{t\in\T}$ is a solution of DRBSDE $(\xi,g+{\rm d}V,-\infty,U)$.

\item [(iii)] The claim that $(Y_t,Z_t)_{t\in\T}$ is a solution of BSDE $(\xi,g+{\rm d}V)$  is equivalent to anyone of the following three claims:
\begin{itemize}
\item $(Y_t,Z_t,0)_{t\in\T}$ is a solution of \underline{R}BSDE $(\xi,g+{\rm d}V,-\infty)$;
\item $(Y_t,Z_t,0)_{t\in\T}$ is a solution of $\bar{R}$BSDE $(\xi,g+{\rm d}V,+\infty)$;
\item $(Y_t,Z_t,0,0)_{t\in\T}$ is a solution of DRBSDE $(\xi,g+{\rm d}V,-\infty,+\infty)$.
\end{itemize}

\item [(iv)] The claim that $(Y_t,Z_t,K_t)_{t\in\T}$ is a solution of \underline{R}BSDE $(\xi,g+{\rm d}V,L)$ is equivalent to the claim that
$(\bar Y_t,\bar Z_t,\bar A_t)_{t\in\T}:=(-Y_t,-Z_t,K_t)_{t\in\T}$ is a solution of $\bar{R}$BSDE $(\bar\xi,\bar g+{\rm d}\bar V,\bar U)$, where
$$\bar\xi:=-\xi,\ \bar V_\cdot:=-V_\cdot,\ \bar U_\cdot:=-L_\cdot,\ \bar g(\omega,t,y,z):=-g(\omega,t,-y,-z).$$
\end{itemize}
\end{rmk}

\begin{dfn}
\label{dfn:2-MaximalandMinimalSolution}
A solution $(Y_t,Z_t)_{t\in\T}$ of BSDE $(\xi,g+{\rm d}V)$ is called the minimal (resp. maximal) one in some space if $(Y_\cdot,Z_\cdot)$ belongs to this space, and for any solution $(Y'_t,Z'_t)_{t\in\T}$ of BSDE $(\xi,g+{\rm d}V)$ in this space, we have
\begin{equation}\label{eqn:2-ComparisonOfY}
Y_t\leq Y'_t,\ \ t\in\T\ \ \ ({\rm resp.}\ \  Y_t\geq Y'_t,\ \ t\in\T).
\end{equation}
Similarly, we can define that
\begin{itemize}
\item A solution $(Y_t,Z_t,K_t)_{t\in\T}$ of \underline{R}BSDE $(\xi,g+{\rm d}V,L)$ is called the minimal (resp. maximal)  one in some space if  $(Y_\cdot,Z_\cdot,K_\cdot)$ belongs to this space, and \eqref{eqn:2-ComparisonOfY} holds for any solution $(Y'_t,Z'_t,K'_t)_{t\in\T}$ of \underline{R}BSDE $(\xi,g+{\rm d}V,L)$ in this space.
\item A solution $(Y_t,Z_t,A_t)_{t\in\T}$ of $\bar{R}$BSDE $(\xi,g+{\rm d}V,U)$ is called the minimal (resp. maximal) one in some space if  $(Y_\cdot,Z_\cdot,A_\cdot)$ belongs to this space, and \eqref{eqn:2-ComparisonOfY} holds for any solution $(Y'_t,Z'_t,A'_t)_{t\in\T}$ of $\bar{R}$BSDE $(\xi,g+{\rm d}V,U)$ in this space.
\item A solution $(Y_t,Z_t,K_t,A_t)_{t\in\T}$ of DRBSDE $(\xi,g+{\rm d}V,L,U)$ is called the minimal (resp. maximal)  one in some space if  $(Y_\cdot,Z_\cdot,K_\cdot,A_\cdot)$ belongs to this space, and \eqref{eqn:2-ComparisonOfY} holds for any solution $(Y'_t,Z'_t,K'_t,A'_t)_{t\in\T}$ of DRBSDE $(\xi,g+{\rm d}V,L,U)$ in this space.\vspace{0.2cm}
\end{itemize}
\end{dfn}

\subsection{Assumptions\vspace{0.1cm}}

In this paper, we will mainly use the following assumptions on the generator, the terminal condition and the barriers, where $p>1$.
\begin{enumerate}

\renewcommand{\theenumi}{(H\arabic{enumi})}
\renewcommand{\labelenumi}{\theenumi}

\item \label{A:(H1)}
$g$ satisfies the one-sided Osgood condition in $y$, i.e., there exists a nondecreasing concave function $\rho(\cdot):\R_+\mapsto \R_+$ with $\rho(0)=0$, $\rho(u)>0$ for $u>0$ and $\int_{0^+} {{\rm d}u\over \rho(u)}=+\infty$ such that $\as$, $\RE\ y_1,y_2\in \R,z\in\R^{d}$,
$$
(g(\omega,t,y_1,z)-g(\omega,t,y_2,z)){\rm sgn}(y_1-y_2)\leq \rho(|y_1-y_2|).
$$
\item \label{A:(H2)}
\begin{itemize}
\item [(i)] $g$ is continuous in $y$, i.e, $\as$, $\RE\ z\in {\R^{d}},\ \ g(\omega,t,\cdot,z)$ is continuous;
\item [(ii)] $g$ is uniformly continuous in $z$, i.e., there exists a nondecreasing  continuous function $\phi(\cdot):\R_+\mapsto \R_+$ with $\phi(0)=0$ such that $\as$, $\RE\ y\in\R, z_1,z_2\in\R^{d}$,
$$
|g(\omega,t,y,z_1)-g(\omega,t,y,z_2)|\leq \phi(|z_1-z_2|).
$$
\end{itemize}

\renewcommand{\theenumi}{(H2')}
\renewcommand{\labelenumi}{\theenumi}

\item \label{A:(H2')}
\begin{itemize}
\item [(i)] $g$ is stronger continuous in $(y,z)$, i.e., $\as$, $\RE\ y\in \R,\ g(\omega,t,y,\cdot)$ is continuous, and $g(\omega,t,\cdot,z)$ is continuous uniformly with respect to $z$;
\item [(ii)] $g$ has a stronger linear growth in $z$, i.e., there exist two constants $\mu, \lambda\geq 0$ and a nonnegative process $f_\cdot\in \hcal^p$ such that $\as$, $\RE\ y\in\R, z\in\R^{d}$,
$$
|g(\omega,t,y,z)-g(\omega,t,y,0)|\leq f_t(\omega)+\mu |y|+\lambda |z|.
$$
\end{itemize}

\renewcommand{\theenumi}{(H\arabic{enumi})}
\renewcommand{\labelenumi}{\theenumi}
\setcounter{enumi}{2}

\item \label{A:(H3)}
\begin{itemize}
\item [(i)] $g$ has a general growth in $y$, i.e, $\RE r>0, \ \psi_\cdot(r):=\sup\limits_{|y|\leq r}|g(\cdot,y,0)-g(\cdot,0,0)|\ \in\hcal$;
\item [(ii)] $g(\cdot,0,0)\in\hcal^p$.
\end{itemize}

\item \label{A:(H4)}
\begin{itemize}
\item [(i)]$L_\cdot\in \s\ ({\rm or}\  L_\cdot\equiv-\infty)$, $U_\cdot\in \s\ ({\rm or}\ U_\cdot\equiv +\infty)$, $\xi\in \Lp$ and $L_T\leq \xi\leq U_T$;

\item [(ii)] There exists an $X_\cdot\in \mcal^p+\vcal^p$ such that $g(\cdot,X_\cdot,0)\in \hcal^p$ and $L_t\leq X_t\leq U_t$ for each $t\in \T$.\vspace{0.2cm}
\end{itemize}
\end{enumerate}

\begin{rmk}
\label{rmk:2-LinearGrowthOfRhoandPhi}
Without loss of generality, we will always assume that the functions $\rho(\cdot)$ and $\phi(\cdot)$ defined respectively in (H1) and (H2) are of linear growth, i.e., there exists a constant $A>0$ such that
$$\RE\ x\in \R_+,\ \ \rho(x)\leq A(x+1)\ \ {\rm and}\ \ \phi(x)\leq A(x+1).$$
And, we note that assumption \ref{A:(H4)} is usually called the generalized Mokobodzki condition for doubly reflected BSDEs, which relate the growth of $g$ and that of $L_\cdot$ and $U_\cdot$\vspace{0.1cm}
\end{rmk}

In order to illustrate our results more clearly, the following several assumptions will also be used.\vspace{-0.1cm}

\begin{enumerate}

\renewcommand{\theenumi}{(H1s)}
\renewcommand{\labelenumi}{\theenumi}

\item \label{A:(H1s)} $g$ satisfies the monotonicity condition in $y$, i.e., there exists a constant $\mu\in \R$ such that $\as$, $\RE\ y_1,y_2\in \R,z\in\R^{d}$,
$$
(g(\omega,t,y_1,z)-g(\omega,t,y_2,z)){\rm sgn}(y_1-y_2)\leq \mu |y_1-y_2|.
$$

\renewcommand{\theenumi}{(H2s)}
\renewcommand{\labelenumi}{\theenumi}

\item \label{A:(H2s)}
\begin{itemize}
\item [(i)] $g$ is continuous in $y$, i.e, $\as$, $\RE\ z\in {\R^{d}},\ \ g(\omega,t,\cdot,z)$ is continuous;
\item [(ii)] $g$ satisfies the uniform Lipschitz condition in $z$, i.e., there exists a constant $\lambda\geq 0$ such that $\as$, $\RE\ y\in\R, z_1,z_2\in\R^{d}$,
$$
|g(\omega,t,y,z_1)-g(\omega,t,y,z_2)|\leq \lambda |z_1-z_2|.
$$
\end{itemize}

\renewcommand{\theenumi}{(H3s)}
\renewcommand{\labelenumi}{\theenumi}

\item \label{A:(H3s)}
$g$ has a linear growth in $y$, i.e., there exists a constant $\mu\geq 0$ and a nonnegative process $f_\cdot\in \hcal^p$ such that $\as$, $\RE\ y\in \R$, $|g(\omega,t,y,0)|\leq f_t(\omega)+ \mu |y|$.\vspace{0.2cm}
\end{enumerate}

\begin{rmk}\label{rmk:2-(H1s)Stronger(H1)}
It is clear that assumptions \ref{A:(H1s)}, \ref{A:(H2s)} and \ref{A:(H3s)} are respectively (strictly) stronger than \ref{A:(H1)}, \ref{A:(H2)} and \ref{A:(H3)}. And, (ii) of \ref{A:(H2)} implies (ii) of \ref{A:(H2')}.\vspace{0.2cm}
\end{rmk}

Moreover, the following several assumptions will be used in some technical results of this paper.

\begin{enumerate}

\renewcommand{\theenumi}{(AA)}
\renewcommand{\labelenumi}{\theenumi}

\item \label{A:(AA)} There exist two nonnegative constants $\bar\mu$ and $\bar\lambda$ such that $\as$, $\RE\ y\in\R,\ z\in\R^{d}$,
$$g(\omega,t,y,z){\rm sgn}(y)\leq \bar f_t(\omega)+\bar\mu |y|+\bar\lambda|z|,$$
where $\bar f_\cdot$ is a nonnegative process belonging to $\hcal^p$.

\renewcommand{\theenumi}{(HH)}
\renewcommand{\labelenumi}{\theenumi}

\item \label{A:(HH)}
\begin{itemize}
\item [(i)] $g$ is continuous in $(y,z)$, i.e, $\as$, \ $g(\omega,t,\cdot,\cdot)$ is continuous;
\item [(ii)] $g$ has a general growth in $(y,z)$, i.e., there exists a constant $\lambda\geq 0$, a nonnegative process $f_\cdot\in\hcal^p$ and a nonnegative function $\psi_\cdot(r)\in {\bf S}$ such that $\as$, $\RE\ y\in\R,\ z\in\R^{d}$,
$$
|g(\omega,t,y,z)|\leq f_t(\omega)+ \psi_t(\omega,|y|)+\lambda |z|,
$$
here and hereafter, ${\bf S}$ denotes the set of nonnegative functions $\psi_t(\omega,r):\Omega\times \T\times \R_+ \mapsto\R_+$
satisfying the following two conditions:
\begin{itemize}
\item $\as$, the function $r\mapsto \psi_t(\omega,r)$ is increasing and $\psi_t(\omega,0)=0$;
\item for each $r\geq 0$, $\psi_\cdot(r)\in \hcal$.
\end{itemize}
\end{itemize}

\renewcommand{\theenumi}{(H4L)}
\renewcommand{\labelenumi}{\theenumi}

\item \label{A:(H4L)}
\begin{itemize}
\item [(i)]$L_\cdot\in \s\ ({\rm or}\  L_\cdot\equiv-\infty)$, $\xi\in \Lp$ and $L_T\leq \xi$;

\item [(ii)] There exists an $X_\cdot\in \mcal^p+\vcal^p$ such that $g(\cdot,X_\cdot,0)\in \hcal^p$ and $L_t\leq X_t$ for each $t\in \T$.
\end{itemize}

\renewcommand{\theenumi}{(H4U)}
\renewcommand{\labelenumi}{\theenumi}

\item \label{A:(H4U)}
\begin{itemize}
\item [(i)] $U_\cdot\in \s\ ({\rm or}\ U_\cdot\equiv +\infty)$, $\xi\in \Lp$ and $\xi\leq U_T$;

\item [(ii)] There exists an $X_\cdot\in \mcal^p+\vcal^p$ such that $g(\cdot,X_\cdot,0)\in \hcal^p$ and $ X_t\leq U_t$ for each $t\in \T$.\vspace{0.2cm}
\end{itemize}

\end{enumerate}

\begin{rmk}\label{rmk:2-ConnectionOf(HH)(AA)(H4)(H4L)}
It is not hard to verify that, see also Remark 2.2 in \citet{Fan2017AMS} for details,
\begin{itemize}
\item [(i)]  \ref{A:(H2)}+\ref{A:(H3)} $\Rightarrow$ \ref{A:(HH)}; \ \ref{A:(H2')}+\ref{A:(H3)} $\Rightarrow$ \ref{A:(HH)}; \ \ref{A:(H1)}+\ref{A:(HH)}(ii) $\Rightarrow$ \ref{A:(AA)}; \ \ref{A:(HH)}(ii) $\Rightarrow$ \ref{A:(H3)};
\item [(ii)] \ref{A:(H4)}  $\Rightarrow$ \ref{A:(H4L)} + \ref{A:(H4U)}; \  \ref{A:(H4L)}(ii) $\Rightarrow$ $L_\cdot^+\in \s^p$; \ \ref{A:(H4U)}(ii) $\Rightarrow$ $U_\cdot^-\in \s^p$;
\item [(iii)] $L_\cdot^+\in\s^p$ together with
    $\left(g(t,\sup\limits_{s\in [0,t]}L_s^+,0)
    \right)_{t\in\T}\in\hcal^p$ implies
    \ref{A:(H4L)}(ii);
\item [(iv)] $U_\cdot^-\in\s^p$ together with
    $\left(g(t,-\inf\limits_{s\in [0,t]}
    U_s^-,0)\right)_{t\in\T}\in\hcal^p$ implies
   \ref{A:(H4U)}(ii);
\item [(v)] If \ref{A:(H3s)} holds, then \ref{A:(H4L)}(ii) $\Leftrightarrow$ $L_\cdot^+\in \s^p$ and \ref{A:(H4U)}(ii) $\Leftrightarrow$ $U_\cdot^-\in \s^p$.\vspace{0.3cm}
\end{itemize}
\end{rmk}

\subsection{Lemma and propositions\vspace{0.2cm}}

Let us first introduce the following lemma, which comes from Lemma 3.1 in \citet{Fan2017AMS}.

\begin{lem}
\label{lem:2-Lemma1}
Let $(\bar Y_\cdot,\bar Z_\cdot,\bar V_\cdot)\in \s\times\M\times\vcal$ satisfy the following equation:
\begin{equation}
\label{eq:4-BarY=BarV}
\bar Y_t=\bar Y_T+\int_t^T {\rm d}\bar V_s-\int_t^T \bar Z_s\cdot {\rm d}B_s,\ \ t\in \T.
\end{equation}
Then, the following two assertions hold.
\begin{itemize}
\item [(i)] There exists a constant $C_1>0$ depending only on $p>0$ such that for each $t\in\T$ and each stopping time $\tau$ valued in $\T$,
$$
\Dis\E\left[\left.\left(\int_{t\wedge\tau}
^{T\wedge\tau}|\bar Z_s|^2{\rm d}s\right)^{p\over 2}\right|\F_t\right]
\leq \Dis C_1\E\left[\left.\sup\limits_{s\in [t,T]}|\bar Y_{s\wedge\tau}|^p+\sup\limits_{s\in [t,T]}\left[\left(\int_{s\wedge\tau}^{T\wedge\tau} \bar Y_r{\rm d}\bar V_r\right)^+\right]^{p\over 2}\right|\F_t\right].
$$

\item [(ii)] If $\bar Y_\cdot\in \s^p$ for some $p>1$, then there exists a constant $C_2>0$ depending only on $p$ such that for each $t\in\T$ and each stopping time $\tau$ valued in $\T$,
$$
\begin{array}{ll}
&\Dis\E\left[\left.\sup\limits_{s\in [t,T]}|\bar Y_{s\wedge\tau}|^p+
\int_{t\wedge\tau}^{T\wedge\tau} |\bar Y_s|^{p-2}\mathbbm{1}_{\{|\bar Y_s|\neq 0\}}|\bar Z_s|^2{\rm d}s
\right|\F_t\right]\vspace{0.1cm}\\
\leq &\Dis C_2\E\left[\left.|\bar Y_\tau|^p+\sup\limits_{s\in [t,T]}\left(\int_{s\wedge\tau}^{T\wedge\tau} |\bar Y_r|^{p-1}{\rm sgn}(\bar Y_r){\rm d}\bar V_r\right)^+\right|\F_t\right].
\end{array}\vspace{0.2cm}
$$
\end{itemize}
\end{lem}

The following proposition is a direct corollary of Lemma 3.2 in \citet{Fan2017AMS}.

\begin{pro}\label{pro:2-Pro1}
Assume that $p>1$, $\xi\in\Lp$, $V_\cdot\in\vcal^p$ and the generator $g$ satisfies assumption \ref{A:(AA)}. Let $(Y_\cdot,Z_\cdot)\in \s^p\times\M^p$ be a solution of BSDE $(\xi,g+{\rm d}V)$. Then, there exists a constant $C>0$ depending only on $p,\bar\mu,\bar\lambda,T$ such that for each $t\in\T$,
$$
\begin{array}{l}
\Dis \E\left[\left.\sup\limits_{s\in [t,T]}|Y_s|^p+\left(\int_t^T|Z_s|^2{\rm d}s\right)^{p\over 2}+\left(\int_t^T |g(s,Y_s,Z_s)|{\rm d}s\right)^p\right|\F_t\right]\\
\ \ \ \ \leq \Dis C\E\left[\left.|\xi|^p+|V|^p_{t,T}+\left(\int_t^T \bar f_s\ {\rm d}s\right)^p
\right|\F_t\right].\vspace{0.3cm}
\end{array}
$$
\end{pro}

By Lemma 3.4 in \citet{Fan2017AMS}, we can verify the following a priori estimate.

\begin{pro}\label{pro:2-Pro2}
Assume that $p>1$, $\xi\in\Lp$, $V_\cdot\in\vcal^p$, $K_\cdot,A_\cdot\in\vcal^{+,p}$, and the generator $g$ satisfies assumptions \ref{A:(H1)} with $\rho(\cdot)$, (ii) of \ref{A:(H2')} with $f_\cdot$, $\mu$ and $\lambda$, and (ii) of \ref{A:(H3)}.
\begin{itemize}
\item [(i)] Let $(Y_\cdot,Z_\cdot)\in \s^p\times\M^p$ be a solution of BSDE $(\xi,g+{\rm d}\bar V)$ with $\bar V_\cdot=V_\cdot+K_\cdot$, and the following assumption \ref{A:(B1)} hold:
\begin{enumerate}
\renewcommand{\theenumi}{(B1)}
\renewcommand{\labelenumi}{\theenumi}
\item \label{A:(B1)} There exists an $\bar X_\cdot\in \s^p$ such that $g(\cdot,\bar X_\cdot,0)\in \hcal^p$ and $Y_t\leq \bar X_t$ for each $t\in \T$.
\end{enumerate}
Then, there exists a constant $C>0$ depending only on $p,\mu, \lambda,A,T$ such that for each $t\in\T$,
\begin{equation}
\label{eq:3-BoundOfZandKForBSDEwithBarV}
\begin{array}{ll}
&\Dis\E\left[\left.\left(\int_t^T|Z_s|^2{\rm d}s\right)^{p\over 2}+|K_T-K_t|^p+\left(\int_t^T|g(s,Y_s,Z_s)|{\rm d}s\right)^p\right|\F_t\right]\vspace{0.1cm}\\
\leq &\Dis C\E\left[\sup\limits_{s\in [t,T]}|Y_s|^p+|V|^p_{t,T}+\sup\limits_{s\in [t,T]}|\bar X_s|^{p}+\left(\int_t^T f_s{\rm d}s\right)^p+1\right.\vspace{0.1cm}\\
&\hspace{1cm}\Dis+\left.\left.\left(\int_t^T |g(s,\bar X_s,0)|\ {\rm d}s\right)^p+\left(\int_t^T |g(s,0,0)|\ {\rm d}s\right)^p\right|\F_t\right].
\end{array}
\end{equation}

\item [(ii)] Let $(Y_\cdot,Z_\cdot)\in \s^p\times\M^p$ be a solution of BSDE $(\xi,g+{\rm d}\underline{V})$ with $\underline{V}_\cdot=V_\cdot-A_\cdot$, and the following assumption \ref{A:(B2)} hold:
\begin{enumerate}
\renewcommand{\theenumi}{(B2)}
\renewcommand{\labelenumi}{\theenumi}
\item \label{A:(B2)} There exists an $\underline{X}_\cdot\in \s^p$ such that $g(\cdot,\underline{X}_\cdot,0)\in \hcal^p$ and $\underline{X}_t\leq Y_t$ for each $t\in \T$.
\end{enumerate}
Then, there exists a constant $C>0$ depending only on $p,\mu, \lambda,A,T$ such that for each $t\in\T$,
\begin{equation}
\label{eq:3-BoundOfZandAForBSDEwithUnderlineV}
\begin{array}{ll}
&\Dis\E\left[\left.\left(\int_t^T|Z_s|^2{\rm d}s\right)^{p\over 2}+|A_T-A_t|^p+\left(\int_t^T|g(s,Y_s,Z_s)|{\rm d}s\right)^p\right|\F_t\right]\vspace{0.1cm}\\
\leq &\Dis C\E\left[\sup\limits_{s\in [t,T]}|Y_s|^p+|V|^p_{t,T}+\sup\limits_{s\in [t,T]}|\underline{X}_s|^{p}+\left(\int_t^T f_s{\rm d}s\right)^p+1\right.\vspace{0.1cm}\\
&\hspace{1cm}\Dis+\left.\left.\left(\int_t^T |g(s,\underline{X}_s,0)|\ {\rm d}s\right)^p+\left(\int_t^T |g(s,0,0)|\ {\rm d}s\right)^p\right|\F_t\right].
\end{array}\vspace{0.2cm}
\end{equation}
\end{itemize}
\end{pro}

\begin{proof} Since $g$ satisfies \ref{A:(H1)} with $\rho(\cdot)$ and \ref{A:(H2')}(ii) with $f_\cdot$, $\mu$ and $\lambda$, it follows that $\as$,
$$
\begin{array}{lll}
\Dis g(t,y,z){\rm sgn}(y)&\leq &\Dis |(g(t,y,z)-g(t,y,0)){\rm sgn}(y)|+(g(t,y,0)-g(t,0,0)){\rm sgn}(y)+|g(t,0,0))|\\
&\leq &\Dis f_t+\mu |y|+\lambda |z|+\rho(|y|)+|g(t,0,0))|\\
&\leq &\Dis f_t+|g(t,0,0))|+A+(\mu+A)|y|+\lambda |z|,\ \ \ \RE\ y\in\R$, $z\in\R^d.
\end{array}
$$
Furthermore, if $(Y_\cdot,Z_\cdot)\in \s^p\times\M^p$ is a solution of BSDE $(\xi,g+{\rm d}\bar V)$ with $\bar V_\cdot=V_\cdot+K_\cdot$, and assumption \ref{A:(B1)} holds, then it follows from \ref{A:(H1)} and \ref{A:(H2')}(ii) that $\as$,
$$
\begin{array}{lll}
\Dis -g(t,Y_t,Z_t)&\leq &\Dis
\left(g(t,\bar X_t,Z_t)
-g(t,Y_t,Z_t)\right)+\left|g(t,\bar X_t,0)-g(t,\bar X_t,Z_t)
\right|+|g(t,\bar X_t,0)|\\
&\leq &\Dis \rho(|\bar X_t-Y_t|)+f_t+\mu |\bar X_t|+\lambda |Z_t|+|g(t,\bar X_t,0)|\\
&\leq &\Dis |g(t,\bar X_t,0)|+(\mu+A)|\bar X_t|+f_t+A+A|Y_t|+\lambda |Z_t|.
\end{array}
$$
With the above two inequalities in hand and in view of the fact of $ f_\cdot, g(\cdot,0,0), g(\cdot,\bar X_\cdot,0)\in \hcal^p$ and $\bar X_\cdot\in \s^p$, the desired estimate \eqref{eq:3-BoundOfZandKForBSDEwithBarV} follows immediately from Lemma 3.4 in \citet{Fan2017AMS}. Finally, in view of (iv) in \cref{rmk:2-ConnectionOfBSDEandRBSDEandDRBSDE}, we know that (ii) of \cref{pro:2-Pro2} is also true.
\end{proof}

Now, let us recall several important results on reflected BSDEs with one continuous barrier and non-reflected BSDEs obtained in \citet{Fan2017AMS}, by using (iv) in \cref{rmk:2-ConnectionOfBSDEandRBSDEandDRBSDE} for the case of $\bar{R}$BSDEs, see Theorem 4.4, Corollary 4.5, Remark 4.6, Theorem 5.2, Corollary 5.4, Theorem 5.8, Corollary 5.9, Remark 5.10, Theorem 5.11, Theorem 5.13, Remark 5.14 and Proposition 5.15 in \citet{Fan2017AMS} for more details.

\begin{pro}
\label{pro:2-ExitenceUniquenessOfBSDEsAndRBSDEs}
Assume that $p>1$, $\xi\in\Lp$, $V_\cdot\in\vcal^p$ and the generator $g$ satisfies assumptions \ref{A:(H1)}, \ref{A:(H2)} and \ref{A:(H3)}. We have the following assertions.

\begin{itemize}
\item [(i)] BSDE $(\xi,g+{\rm d}V)$ admits a unique solution $(Y_t,Z_t)_{t\in\T}$ in $\s^p\times\M^p$.
\item [(ii)] Assume further that \ref{A:(H4L)} holds for $\xi$, $L_\cdot$ and $X_\cdot$ Then, \underline{R}BSDE $(\xi,g+{\rm d}V,L)$ admits a unique solution $(Y_t,Z_t,K_t)_{t\in\T}$ in $\s^p\times\M^p\times\vcal^{+,p}$.
\item [(iii)] Assume further that \ref{A:(H4U)} holds for $\xi$, $U_\cdot$ and $X_\cdot$ Then, $\bar{R}$BSDE $(\xi,g+{\rm d}V,U)$ admits a unique solution $(Y_t,Z_t,A_t)_{t\in\T}$ in $\s^p\times\M^p\times\vcal^{+,p}$.
\end{itemize}
\end{pro}

\begin{pro}
\label{pro:2-ComparisonOfUniquenessSolution}
Let $p>1$ and assume that for $i=1,2$, $\xi^i\in\Lp$ with $\xi^1\leq \xi^2$, $V_\cdot^i\in \vcal^p$ with ${\rm d}V^1\leq {\rm d}V^2$, and the generator $g^i$ satisfies assumptions \ref{A:(H1)}-\ref{A:(H3)}. We have the following assertions.

\begin{itemize}
\item [(i)]  For $i=1,2$, let $(Y_\cdot^i,Z_\cdot^i)$ be the unique solution of BSDE $(\xi^i,g^i+{\rm d}V^i)$ in $\s^p\times \M^p$. If
\begin{equation}
\label{eq:2-(2.6)}
\as,\ \ g^1(t,Y_t^1,Z_t^1)\leq g^2(t,Y_t^1,Z_t^1)\ \ ({\rm resp.}\  g^1(t,Y_t^2,Z_t^2)\leq g^2(t,Y_t^2,Z_t^2)),
\end{equation}
then $Y_t^1\leq Y_t^2$ for each $t\in \T$.

\item [(ii)] For $i=1,2$, suppose that \ref{A:(H4L)}(i) holds for $\xi^i$ and $L^i$, and that $(Y_\cdot^i,Z_\cdot^i,K_\cdot^i)$ is the unique solution of \underline{R}BSDE $(\xi^i,g^i+{\rm d}V^i,L^i)$ in the space $\s^p\times \M^p\times \vcal^{+,p}$. If $L^1_\cdot\leq L^2_\cdot$ and \eqref{eq:2-(2.6)} is satisfied, then
    \begin{equation}\label{eq:2-Y1LeqY2}
     Y_t^1\leq Y_t^2, \ t\in \T.
    \end{equation}
    Moreover, if  $L^1_\cdot=L^2_\cdot$ and for each $(y,z)\in \R\times \R^d$,
\begin{equation}
\label{eq:2-(2.8)}
\as,\ \ g^1(t,y,z)\leq g^2(t,y,z),
\end{equation}
then
    ${\rm d}K^1\geq {\rm d}K^2$.

\item [(iii)] For $i=1,2$, suppose that \ref{A:(H4U)}(i) holds for $\xi^i$ and  $U^i$, and that $(Y_\cdot^i,Z_\cdot^i,A_\cdot^i)$ is the unique solution of $\bar{R}$BSDE $(\xi^i,g^i+{\rm d}V^i,U^i)$ in the space $\s^p\times \M^p\times \vcal^{+,p}$. If $U^1_\cdot\leq U^2_\cdot$ and \eqref{eq:2-(2.6)} is satisfied, then \eqref{eq:2-Y1LeqY2} holds. Moreover, if  $U^1_\cdot=U^2_\cdot$ and \eqref{eq:2-(2.8)} holds, then ${\rm d}A^1\leq {\rm d}A^2$.\vspace{0.2cm}
\end{itemize}
\end{pro}

\begin{rmk}
\label{rmk:2-ComparisonWithoutIntegrableCondition}
From the related proof in \citet{Fan2017AMS}, it can be observed that in order to obtain \eqref{eq:2-Y1LeqY2} in (ii) of \cref{pro:2-ComparisonOfUniquenessSolution} we do not need the condition that
$$\int_0^T (Y^2_t-L^2_t)\ {\rm d}K^2_t=0,$$
and in order to obtain \eqref{eq:2-Y1LeqY2} in (iii) of \cref{pro:2-ComparisonOfUniquenessSolution} we do not need the condition that
$$\int_0^T (U^1_t-Y^1_t)\ {\rm d}A^1_t=0.$$
This important observation will be used several times later.\vspace{0.2cm}
\end{rmk}

\begin{rmk}
\label{rmk:2-ExistenceAndComparisonOfMaximalSolution}
\cref{pro:2-ExitenceUniquenessOfBSDEsAndRBSDEs},
\cref{pro:2-ComparisonOfUniquenessSolution} and \cref{rmk:2-ComparisonWithoutIntegrableCondition} still hold if we replace \ref{A:(H2)} with \ref{A:(H2')}, \eqref{eq:2-(2.6)} with \eqref{eq:2-(2.8)}, and the word ``unique" with ``maximal (minimal)" in their statements.
\end{rmk}

\section{Necessity of assumption \ref{A:(H4)} (ii) and uniqueness of the $L^p$ solution}
\label{sec:3-NecessityUniqueness}
\setcounter{equation}{0}

\subsection{Necessity of the assumption\vspace{0.1cm}}

By the following theorem we show that under conditions of \ref{A:(H1)}, \ref{A:(H2)} (ii) (resp. \ref{A:(H2')}(ii)), \ref{A:(H3)}(ii) and \ref{A:(H4)}(i), \ref{A:(H4)}(ii) is necessary to ensure the existence of $L^p$ solutions for DRBSDEs.

\begin{thm} [Necessity of \ref{A:(H4)}(ii)] \label{thm:3-Necessity}
Assume that $p>1$, $V_\cdot\in \vcal^p$, the generator $g$ satisfies assumptions \ref{A:(H1)}, \ref{A:(H2)}(ii) (resp. \ref{A:(H2')}(ii)) and \ref{A:(H3)}(ii), and that assumption \ref{A:(H4)}(i) holds for $L_\cdot$, $U_\cdot$ and $\xi$. If DRBSDE $(\xi,g+{\rm d}V,L,U)$ admits a solution $(Y_\cdot,Z_\cdot,K_\cdot,A_\cdot)$ in the space $\s^p\times \M^p\times \vcal^{+,p}\times \vcal^{+,p}$, then $g(\cdot,Y_\cdot,0)\in \hcal^p$. That is to say, \ref{A:(H4)}(ii) holds true.
\end{thm}

\begin{proof}
We only prove the case of \ref{A:(H2')}(ii). The case of \ref{A:(H2)}(ii) can be proved in a same way. In fact, by assumptions of \cref{thm:3-Necessity}, it is easy to verify that $g$ satisfies \ref{A:(AA)} with $\bar f_\cdot:=|g(\cdot,0,0)|+f_\cdot+A$, $\bar\mu:=\mu+A$
and $\bar\lambda:=\lambda$ (see also the proof of \cref{pro:2-Pro2}), and $(Y_\cdot,Z_\cdot)\in \s^p\times \M^p$ is a solution of BSDE$(\xi,g+{\rm d}\bar V)$ with $\bar V_\cdot=V_\cdot+K_\cdot-A_\cdot\in \vcal^{p}$. Then, it follows from \cref{pro:2-Pro1} that
$$
\E\left[\left(\int_0^T |g(t,Y_t,Z_t)|{\rm d}t\right)^p\right]<+\infty.
$$
Furthermore, by \ref{A:(H2')}(ii) together with H\"{o}lder's inequality we deduce that
$$
\begin{array}{lll}
\Dis\E\left[\left(\int_0^T |g(t,Y_t,0)|{\rm d}t\right)^p\right]
&\leq &\Dis
4^p\E\left[\left(\int_0^T |g(t,Y_t,Z_t)|{\rm d}t\right)^p\right]+4^p\E\left[\left(\int_0^T f_t{\rm d}t\right)^p\right]\vspace{0.1cm}\\
&&\Dis +(4\mu T)^p\E\left[\sup\limits_{t\in\T}|Y_t|^p\right]+(4\lambda)^p \E\left[\left(\int_0^T |Z_t|^2{\rm d}t\right)^{p\over 2}\right]<+\infty.
\end{array}
$$
Then, \ref{A:(H4)}(ii) holds for $X_\cdot=Y_\cdot$. \cref{thm:3-Necessity} is then proved.
\end{proof}

\subsection{Comparison theorem\vspace{0.1cm}}

In this subsection, we first establish a general comparison theorem for the $L^p\ (p>1)$ solutions of doubly RBSDEs, and then verify the uniqueness of the $L^p$ solution under assumptions \ref{A:(H1)} and \ref{A:(H2)}(ii).

\begin{pro}[Comparison Theorem for $L^p$ solutions of DRBSDEs]
\label{pro:3-ComparisonTheoremofDRBSDE}
Let $p>1$, $V_\cdot^i\in \vcal^p$, \ref{A:(H4)}(i) hold for $\xi^i$, $L_\cdot^i$ and $U_\cdot^i$, and $(Y_\cdot^i,Z_\cdot^i,K_\cdot^i,A_\cdot^i)$ be a solution of DRBSDE $(\xi^i,g^i+{\rm d}V^i,L^i,U^i)$ in $\s^p\times \M^p\times \vcal^{+,p}\times \vcal^{+,p}$ for $i=1,2$. If $\xi^1\leq \xi^2$, ${\rm d}V^1\leq {\rm d}V^2$, $L^1_\cdot\leq L^2_\cdot$, $U^1_\cdot\leq U^2_\cdot$, and either
$$
\left\{
\begin{array}{l}
g^1\ satisfies\ \ref{A:(H1)}\ and \ \ref{A:(H2)}(ii);\\
\as,\ \ \mathbbm{1}_{\{Y^1_t>Y^2_t\}}
\left(g^1(t,Y_t^2,Z_t^2)
-g^2(t,Y_t^2,Z_t^2)\right)\leq 0
\end{array}
\right.
$$
or
$$
\left\{
\begin{array}{l}
g^2\ satisfies\ \ref{A:(H1)}\ and \ \ref{A:(H2)}(ii);\\
\as,\ \ \mathbbm{1}_{\{Y^1_t>Y^2_t\}}
\left(g^1(t,Y_t^1,Z_t^1)
-g^2(t,Y_t^1,Z_t^1)\right)\leq 0
\end{array}
\right.\vspace{0.2cm}
$$
is satisfied, then $Y_t^1\leq Y_t^2$ for each $t\in \T$.
\end{pro}

\begin{proof}
It follows from It\^{o}-Tanaka's formula that
$$
\begin{array}{lll}
\Dis  (Y^1_t-Y^2_t)^+ &\leq &\Dis (\xi^1-\xi^2)^++\int_t^T {\rm sgn}((Y^1_s-Y^2_s)^+)({\rm d}V_s^1-{\rm d}V_s^2)\vspace{0.1cm}\\
&&\Dis +\int_t^T{\rm sgn}((Y^1_s-Y^2_s)^+)\left(g^1(s,Y_s^1,Z_s^1)
-g^2(s,Y_s^2,Z_s^2)\right){\rm d}s\vspace{0.1cm}\\
&&\Dis+\int_t^T {\rm sgn}((Y^1_s-Y^2_s)^+)\left({\rm d}K_s^1-{\rm d}K_s^2\right)+\int_t^T {\rm sgn}((Y^1_s-Y^2_s)^+)\left({\rm d}A_s^2-{\rm d}A_s^1\right)\vspace{0.1cm}\\
&&\Dis +\int_t^T {\rm sgn}((Y^1_s-Y^2_s)^+)(Z_s^1-Z_s^2){\rm d}B_s,\ \ t\in \T.
\end{array}
$$
Since $L_t^1\leq L_t^2\leq Y_t^2$, $L_t^1\leq Y_t^1$, $t\in \T$ and $\int_0^T(Y^1_s-L_s^1){\rm d}K_s^1=0$, we know that for $t\in \T$,
$$
\begin{array}{lll}
\Dis \int_t^T {\rm sgn}((Y^1_s-Y^2_s)^+)\left({\rm d}K_s^1-{\rm d}K_s^2\right)
&\leq & \Dis \int_t^T {\rm sgn}((Y^1_s-Y^2_s)^+){\rm d}K_s^1\leq \int_t^T {\rm sgn}((Y^1_s-L_s^1)^+){\rm d}K_s^1\vspace{0.1cm}\\
&=&\Dis\int_t^T \mathbbm{1}_{\{Y^1_s>L_s^1\}}|Y^1_s-L_s^1|^{-1}(Y^1_s-L_s^1){\rm d}K_s^1=0.
\end{array}
$$
Similarly, since $Y^1_t\leq U_t^1\leq U_t^2$, $Y^2_t\leq U_t^2$, $t\in \T$ and $\int_0^T(U^2_s-Y_s^2){\rm d}A_s^2=0$, we know that for $t\in \T$,
$$
\begin{array}{lll}
\Dis \int_t^T {\rm sgn}((Y^1_s-Y^2_s)^+)\left({\rm d}A_s^2-{\rm d}A_s^1\right)
&\leq & \Dis \int_t^T {\rm sgn}((Y^1_s-Y^2_s)^+){\rm d}A_s^2\leq \int_t^T {\rm sgn}((U^2_s-Y_s^2)^+){\rm d}A_s^2\vspace{0.1cm}\\
&=&\Dis\int_t^T \mathbbm{1}_{\{U^2_s>Y_s^2\}}|U^2_s-Y_s^2|^{-1}
(U^2_s-Y_s^2){\rm d}A_s^2=0.
\end{array}
$$
Thus, in view of $\xi^1\leq \xi^2$ and ${\rm d}V^1\leq {\rm d}V^2$, by the previous three inequalities we obtain
$$
\begin{array}{lll}
\Dis  (Y^1_t-Y^2_t)^+ &\leq &\Dis \int_t^T{\rm sgn}((Y^1_s-Y^2_s)^+)\left(g^1(s,Y_s^1,Z_s^1)
-g^2(s,Y_s^2,Z_s^2)\right){\rm d}s\vspace{0.1cm}\\
&&\Dis +\int_t^T {\rm sgn}((Y^1_s-Y^2_s)^+)(Z_s^1-Z_s^2){\rm d}B_s,\ \ t\in\T.
\end{array}
$$
Now, in view of the assumptions of $g^1$ and $g^2$, the rest proof runs as that in the proof of Theorem 4.4 in \citet{Fan2017AMS}. The proof is complete. \vspace{0.2cm}
\end{proof}

By \cref{pro:3-ComparisonTheoremofDRBSDE}, the following corollary follows immediately.

\begin{cor}
\label{cor:3-CorollaryOfComparisonTheormen}
Let $p>1$, $V_\cdot^i\in \vcal^p$, \ref{A:(H4)}(i) hold for $\xi^i$, $L_\cdot^i$ and $U_\cdot^i$, and $(Y_\cdot^i,Z_\cdot^i,K_\cdot^i,A_\cdot^i)$ be a solution of DRBSDE $(\xi^i,g^i+{\rm d}V^i,L^i,U^i)$ in $\s^p\times \M^p\times \vcal^{+,p}\times \vcal^{+,p}$ for $i=1,2$. If $\xi^1\leq \xi^2$, ${\rm d}V^1\leq {\rm d}V^2$, $L^1_\cdot\leq L^2_\cdot$, $U^1_\cdot\leq U^2_\cdot$, $g^1$ or $g^2$ satisfies \ref{A:(H1)} and \ref{A:(H2)}(ii), and for each $(y,z)\in \R\times \R^d$,
$$\as,\ \ g^1(t,y,z)\leq g^2(t,y,z),$$
then $Y_t^1\leq Y_t^2$ for each $t\in \T$.\vspace{0.2cm}
\end{cor}

\begin{rmk}\label{rmk:3-RemarkOfComparisionTheorem}
We note that in the proof of \cref{pro:3-ComparisonTheoremofDRBSDE} the following two assumptions are not utilized:
$$
\int_0^T(Y^2_s-L_s^2){\rm d}K_s^2=0\ \ {\rm and}\ \
\int_0^T(U^1_s-Y_s^1){\rm d}A_s^1=0.
$$
In addition, it follows from Remarks \ref{rmk:2-ConnectionOfBSDEandRBSDEandDRBSDE} and \ref{rmk:2-ConnectionOf(HH)(AA)(H4)(H4L)} that if the comparison of ${\rm d}K^i$ and ${\rm d}A^i$ is not taken into account, then \cref{pro:3-ComparisonTheoremofDRBSDE} and \cref{cor:3-CorollaryOfComparisonTheormen} improve \cref{pro:2-ComparisonOfUniquenessSolution}.\vspace{0.2cm}
\end{rmk}

\begin{thm}[Uniqueness]
\label{thm:3-UniquenessOfSolutions}
Let $p>1$, $V_\cdot\in\vcal^p$, \ref{A:(H4)}(i) hold for $\xi$, $L_\cdot$ and $U_\cdot$, and the generator $g$ satisfy assumptions \ref{A:(H1)} and \ref{A:(H2)}(ii). Then DRBSDE $(\xi,g+{\rm d}V,L,U)$ admits at most one solution in $\s^p\times\M^p\times\vcal^{+,p}\times\vcal^{+,p}$, i.e, if both $(Y_\cdot,Z_\cdot,K_\cdot,A_\cdot)$ and $(Y'_\cdot,Z'_\cdot,K'_\cdot,A'_\cdot)$ are solutions of DRBSDE $(\xi,g+{\rm d}V,L,U)$ in $\s^p\times\M^p\times\vcal^{+,p}\times\vcal^{+,p}$, then $\as$, $Y_\cdot=Y'_\cdot,\ Z_\cdot=Z'_\cdot\ K_\cdot=K'_\cdot$ and $A_\cdot=A'_\cdot$.
\end{thm}

\begin{proof}
Firstly, it follows from \cref{cor:3-CorollaryOfComparisonTheormen} that $Y_t=Y'_t$ for each $t\in\T$. Furthermore, by It\^{o}'s formula we know that $\as$, $Z_\cdot=Z'_\cdot$, and then $K_t-A_t=K'_t-A'_t$ for each $t\in\T$. Finally, the conclusion of $K_\cdot=K'_\cdot$ and $A_\cdot=A'_\cdot$ follows from the Ham-Bananch Composition of Sign Measure.
\end{proof}

\section{Existence of the $L^p$ solution: penalization method}
\label{sec:4-ExistencePenalization}
\setcounter{equation}{0}

In this section, we will prove the existence of $L^p$ solutions for DRBSDEs under assumptions \ref{A:(H1)}, \ref{A:(H2)} (resp. \ref{A:(H2')}), \ref{A:(H3)} and \ref{A:(H4)} by showing the convergence of the sequence of $L^p$ solutions for the penalized RBSDEs with one continuous barrier and the penalized BSDEs.

\subsection{Convergence of the sequence of $L^p$ solutions for the penalized RBSDEs\vspace{0.1cm}}

The following proposition is a general convergence result on the sequence of $L^p\ (p>1)$ solutions of penalized RBSDEs with one continuous barrier under some elementary conditions.

\begin{pro} [Penalization method] \label{pro:4-Penalization}
Assume that $p>1$, $V_\cdot\in \vcal^p$, \ref{A:(H4)}(i) holds for $L_\cdot, U_\cdot$ and $\xi$, and the generator $g$ satisfies \ref{A:(HH)} with $f_\cdot, \psi_\cdot(r)$ and $\lambda$.

\begin{itemize}
\item [(i)] For each $n\geq 1$, let $( Y^n_\cdot, Z^n_\cdot, A^n_\cdot)$ be a solution of $\bar{R}$BSDE $(\xi,\bar g_n+{\rm d}V,U)$ in $\s^p\times\M^p\times \vcal^{+,p}$ with $\bar g_n(t,y,z):=g(t,y,z)+n(y-L_t)^-$, i.e.,
\begin{equation}
\label{eq:4-PenalizationForRBSDEwithSuperBarrier}
\hspace*{-0.6cm}\left\{
\begin{array}{l}
\Dis{Y}^n_t=\xi+\int_t^Tg(s,{Y}^n_s,
{Z}^n_s)
{\rm d}s+\int_t^T{\rm d}V_s+\int_t^T{\rm d} K^n_s -\int_t^T{\rm d} A^n_s-\int_t^T{Z}^n_s\cdot {\rm d}B_s,\ \ t\in\T,\vspace{0.1cm} \\
\Dis  K^n_t:=n\int_0^t(
{Y}^n_s-L_s)^-\ {\rm d}s,\ \ t\in\T,\\
\Dis {Y}^n_t\leq U_t,\ t\in\T\ \ {\rm and} \ \int_0^T (U_t- Y^n_t){\rm d} A^n_t=0.
\end{array}
\right.
\end{equation}
If for each $n\geq 1$, $Y^n_\cdot\leq Y^{n+1}_\cdot\leq \bar X_\cdot$ for some $\bar X_\cdot\in \s^p$ and ${\rm d} A^n\leq {\rm d} A^{n+1}$, and
\begin{equation}\label{eq:4-ZnKnAnLeqEta}
\sup\limits_{n\geq 1}\E\left[\left(\int_0^T|Z_t^n|^2{\rm d}t\right)^{p\over 2}+|K_T^n|^p+|A_T^n|^p+\left(\int_0^T|g(t,Y_t^n,Z_t^n)|
{\rm d}t\right)^p\right]<+\infty,
\end{equation}
then there exists a quadruple $( Y_\cdot, Z_\cdot, K_\cdot, A_\cdot)\in \s^p\times\M^p\times\vcal^{+,p}\times\vcal^{+,p}$ which solves DRBSDE $(\xi,g+{\rm d}V,L,U)$,
$$
\lim\limits_{n\To \infty}\left(\| Y_\cdot^n- Y_\cdot\|_{\s^p}+ \| Z_\cdot^n- Z_\cdot\|_{\M^p}+\| A_\cdot^n- A_\cdot \|_{\s^p}\right)=0,
$$
and there exists a subsequence $\{ K_\cdot^{n_j}\}$ of $\{ K_\cdot^n\}$ such that
$$
\lim\limits_{j\To\infty}\sup\limits_{t\in\T}
| K_t^{n_j}- K_t|=0.
$$

\item [(ii)] For each $n\geq 1$, let $( Y^n_\cdot, Z^n_\cdot, K^n_\cdot)$ be a solution of $\underline{R}$BSDE $(\xi,\underline g_n+{\rm d}V,L)$ in $\s^p\times\M^p\times \vcal^{+,p}$ with $\underline g_n(t,y,z):=g(t,y,z)-n(y-U_t)^+$, i.e.,
$$
\left\{
\begin{array}{l}
\Dis{Y}^n_t=\xi+\int_t^Tg(s,{Y}^n_s,
{Z}^n_s)
{\rm d}s+\int_t^T{\rm d}V_s+\int_t^T{\rm d} K^n_s -\int_t^T{\rm d} A^n_s-\int_t^T{Z}^n_s\cdot {\rm d}B_s,\ \ t\in\T,\vspace{0.1cm}\\
\Dis  A^n_t:=n\int_0^t(
{Y}^n_s-U_s)^+\ {\rm d}s,\ \ t\in\T,\\
\Dis L_t\leq {Y}^n_t,\ t\in\T\ \ {\rm and} \ \int_0^T ( Y^n_t-L_t){\rm d} K^n_t=0.
\end{array}
\right.
$$
If for each $n\geq 1$, $\underline X_\cdot \leq Y^{n+1}_\cdot\leq Y^n_\cdot$ for some $\underline X_\cdot\in \s^p$ and ${\rm d} K^n\leq {\rm d} K^{n+1}$, and \eqref {eq:4-ZnKnAnLeqEta} holds, then there exists a quadruple $( Y_\cdot, Z_\cdot, K_\cdot, A_\cdot)\in \s^p\times\M^p\times\vcal^{+,p}\times\vcal^{+,p}$ which solves DRBSDE $(\xi,g+{\rm d}V,L,U)$,
$$\lim\limits_{n\To \infty}\left(\| Y_\cdot^n- Y_\cdot\|_{\s^p}+
\| Z_\cdot^n- Z_\cdot\|_{\M^p}+\| K_\cdot^n- K_\cdot \|_{\s^p}\right)=0,$$
and there exists a subsequence $\{ A_\cdot^{n_j}\}$ of $\{ A_\cdot^n\}$ such that
$$\lim\limits_{j\To\infty}\sup\limits_{t\in\T}
| A_t^{n_j}- A_t|=0.$$
\end{itemize}
\end{pro}

\begin{proof} We only prove the claim (i). The claim (ii) can be proved in the same way. Now, we assume that $p>1$, $V_\cdot\in \vcal^p$, (i) of \ref{A:(H4)} holds for $L_\cdot, U_\cdot$ and $\xi$, the generator $g$ satisfies \ref{A:(HH)} with $f_\cdot, \psi_\cdot(r)$ and $\lambda$, and all the assumptions in (i) are satisfied.\vspace{0.2cm}

Since $Y_\cdot^n$ increases in $n$ and is bounded above by $\bar X_\cdot$, we know the existence of a progressively measurable real-valued process $Y_\cdot$ such that for each $t\in\T$,
\begin{equation}\label{eq:4-YnIncreaseBounded}
 Y_t^n\uparrow Y_t \ \ \ {\rm and} \ \ \ |Y_t|\vee \left(\sup_{n\geq 1}|Y_t^n|\right)\leq |Y_t^1|+|\bar X_t|.
\end{equation}
Since ${\rm d}A^n\leq {\rm d}A^{n+1}$ for each $n\geq 1$, we get the existence of a progressively measurable and increasing real-valued process $(A_t)_{t\in\T}$ with $A_0=0$ such that $A_t^n\uparrow A_t$ for each $t\in\T$, and for each $j\geq n\geq 1$,
$$A^j_t-A^n_t\leq A^j_T-A^n_T,\ \ t\in\T.\vspace{-0.1cm}$$
In the above inequality, first letting $j\To\infty$, and then taking the superume with respect to $t$ in $\T$, finally letting $n\To\infty$, we obtain that
\begin{equation}\label{eq:4-UniformConvergenceOfAn}
\sup\limits_{t\in\T}|A^n_t-A_t|\To 0,\ \ {\rm as}\ \ n\To\infty,
\end{equation}
which means that $A_\cdot\in\vcal^+$. On the other hand, by Fatou's lemma and \eqref{eq:4-ZnKnAnLeqEta} we deduce that
$$
\Dis\E\left[|A_T|^p\right]
=\Dis\E\left[\liminf
\limits_{n\To\infty}|A_T^n|^p\right]
\leq \liminf\limits_{n\To\infty}\E\left[|A_T^n|^p\right]
\leq \sup_{n\geq 1}\E\left[|A_T^n|^p\right] <+\infty.
$$
So, $A_\cdot\in\vcal^{+,p}$. Then, in view of \eqref{eq:4-UniformConvergenceOfAn} and the fact that $|A^n_\cdot|\leq |A_\cdot|$ for each $n\geq 1$, Lebesgue's dominated convergence theorem yields that
\begin{equation}\label{eq:4-ConvergenceOfAnInSp}
\Lim\|A^n_\cdot-A_\cdot\|_{\s^p}=0.\vspace{0.2cm}
\end{equation}

For each integer $l,q\geq 1$, we introduce the following two stopping times:
$$
\begin{array}{rll}
\tau_l&:=&\Dis \inf\left\{t\geq 0:\  |Y_t^1|+|\bar X_t|+\int_0^t f_s{\rm d}s+L_t^+\geq l\right\}\wedge T,\vspace{0.2cm}\\
\sigma_{l,q}&:=&\Dis \inf\left\{t\geq 0:\  \int_0^t \psi_s(l) {\rm d}s\geq q\right\}\wedge \tau_l\vspace{0.1cm}
\end{array}
$$
with the convention $\inf\emptyset=+\infty$. Then we have, $\tau_l\To T$ as $l\To \infty$, $\sigma_{l,q}\To \tau_l$ as $q\To \infty$ for each $l\geq 1$,
$$
\mathbb{P}\left(\left\{\omega:\ \exists l_0(\omega)\geq 1, \ \RE l\geq l_0(\omega), \ \tau_l(\omega)=T\right\}\right)=1
$$
and
\begin{equation}\label{eq:4-StabilityOfTauSigma}
\mathbb{P}\left(\left\{\omega:\ \exists l_0(\omega), q_0(\omega)\geq 1,\ \RE l\geq l_0(\omega), \RE q\geq q_0(\omega), \ \sigma_{l,q}(\omega)=T\right\}\right)=1.
\end{equation}
Furthermore, since $g$ satisfies assumption (HH) with $f_\cdot, \psi_\cdot(r)$ and $\lambda$, and \eqref{eq:4-YnIncreaseBounded} holds, it follows from the definitions of $\tau_l$ and $\sigma_{l,q}$ that $\as$, for each $n\geq 1$,
\begin{equation}\label{eq:4-BoundOfgYnZn}
|h^{n;l,q}_t|\leq \mathbbm{1}_{t\leq \tau_{l}}f_t+\mathbbm{1}_{t\leq \sigma_{l,q}}\psi_t(l)+\lambda |Z_t^n|\vspace{0.1cm}
\end{equation}
with $h^{n;l,q}_t:=\mathbbm{1}_{t\leq \sigma_{l,q}}g(t,Y_t^n,Z_t^n)$,
\begin{equation}
\label{eq:4-IntegralOfftAndPsitleqlq}
\E\left[\int_0^T \mathbbm{1}_{t\leq \tau_{l}}f_t{\rm d}t\right]\leq l\ \ {\rm and}\ \ \E\left[\int_0^T \mathbbm{1}_{t\leq \sigma_{l,q}}\psi_t(l){\rm d}t\right]\leq q.\vspace{0.2cm}
\end{equation}

The rest proof is divided into 6 steps, which will be detailed in Appendix.\vspace{0.2cm}

{\bf Step 1.}\ Based on \eqref{eq:4-PenalizationForRBSDEwithSuperBarrier}-\eqref{eq:4-IntegralOfftAndPsitleqlq}, by using a weak convergence argument together with Lemma 4.4 of \citet{Klimsiak2012EJP} and Lemma A.3 in \citet{BayraktarYao2015SPA}, we show that $Y_\cdot$ is a c\`{a}dl\`{a}g process.\vspace{0.2cm}

{\bf Step 2.}\ Making use of the conclusion of step 1 together with the definition of $K^n_\cdot$, \eqref{eq:4-ZnKnAnLeqEta} and Dini's theorem, we show that $Y_t\geq L_t$ for each $t\in \T$ and $\Lim\sup\limits_{t\in\T}(Y^n_t-L_t)^-= 0$.\vspace{0.2cm}

{\bf Step 3.}\ By virtue of (ii) of \cref{lem:2-Lemma1}, the definition of $K_\cdot^n$ and $A_\cdot^n$ with \eqref{eq:4-StabilityOfTauSigma}-\eqref{eq:4-IntegralOfftAndPsitleqlq}, H\"{o}lder's inequality, the conclusion of step 2, \eqref{eq:4-ZnKnAnLeqEta}, \eqref{eq:4-YnIncreaseBounded} and Lebesgue's dominated convergence theorem, we show that the sequence $\{Y_\cdot^n\}$ converges to the process $Y_\cdot$ in $\s^p$ as $n\to \infty$.\vspace{0.2cm}

{\bf Step 4.}\ By virtue of (i) of \cref{lem:2-Lemma1}, H\"{o}lder's inequality, \eqref{eq:4-ZnKnAnLeqEta} and the conclusion of step 3, we show that the sequence $\{Z_\cdot^n\}$ converges to a process $Z_\cdot$ in $\M^p$ as $n\to \infty$.\vspace{0.2cm}

{\bf Step 5.}\ Making use of the continuity of $g$, \eqref{eq:4-ZnKnAnLeqEta}, \eqref{eq:4-ConvergenceOfAnInSp}-\eqref{eq:4-IntegralOfftAndPsitleqlq}, and the conclusions of steps 3 and 4, we show that there exists a subsequence $\{K_\cdot^{n_j}\}$ of the sequence $\{K_\cdot^n\}$ which converges almost surely to a process $K_\cdot\in \vcal^{+,p}$ uniformly in $t$ as $j\to \infty$.\vspace{0.2cm}

{\bf Step 6.}\ Based on all the conclusions of steps 1-6, we finally show that  $(Y_\cdot,Z_\cdot, K_\cdot, A_\cdot)$ is a solution of DRBSDE $(\xi,g+{\rm d}V,L,U)$ in the space $\s^p\times\M^p\times\vcal^{+,p}\times\vcal^{+,p}$. The proof is then completed.\vspace{0.2cm}
\end{proof}

\subsection{Uniform estimate on the $L^p$ solutions of the penalized equations\vspace{0.2cm}}

In this subsection, we establish the following uniform estimate on the $L^p\ (p>1)$ solutions of the penalization equations for reflected BSDEs with one continuous barrier and non-reflected BSDEs.

\begin{pro}\label{pro:4-EstimateOfPenalizationEq}
Assume that $p>1$, $V_\cdot\in\vcal^p$, the generator $g$ satisfies assumptions \ref{A:(H1)}, \ref{A:(H2)} and \ref{A:(H3)}, and assumption \ref{A:(H4)} holds for $L_\cdot,U_\cdot,\xi$ and $X_\cdot$.
\begin{itemize}
\item [(i)] For each $n\geq 1$, let $(\underline Y^n_\cdot,\underline Z^n_\cdot,\underline A^n_\cdot)$ be the unique solution of $\bar{R}$BSDE $(\xi,\bar g_n+{\rm d}V,U)$ in $\s^p\times\M^p\times \vcal^{+,p}$ with $\bar g_n(t,y,z):=g(t,y,z)+n(y-L_t)^-$, i.e.,
\begin{equation}
\label{eq:3-PenalizationForRBSDEwithSuperBarrier}
\hspace*{-0.6cm}\left\{
\begin{array}{l}
\Dis\underline{Y}^n_t=\xi+\int_t^Tg(s,\underline{Y}^n_s,
\underline{Z}^n_s)
{\rm d}s+\int_t^T{\rm d}V_s+\int_t^T{\rm d}\underline K^n_s -\int_t^T{\rm d}\underline A^n_s-\int_t^T\underline{Z}^n_s\cdot {\rm d}B_s,\ \   t\in\T,\vspace{0.1cm}\\
\Dis \underline K^n_t:=n\int_0^t(
\underline{Y}^n_s-L_s)^-\ {\rm d}s,\ \ t\in\T,\\
\Dis \underline{Y}^n_t\leq U_t,\ t\in\T\ \ {\rm and} \ \int_0^T (U_t-\underline Y^n_t){\rm d}\underline A^n_t=0
\end{array}
\right.
\end{equation}
(Recall (ii) of \cref{rmk:2-ConnectionOf(HH)(AA)(H4)(H4L)} and (iii) of \cref{pro:2-ExitenceUniquenessOfBSDEsAndRBSDEs}). Then, $\underline Y^n_\cdot$ increases in $n$ and is bounded above by a process $\bar X_\cdot\in \s^p$, ${\rm d}\underline A^n\leq {\rm d}\underline A^{n+1}$, and  there exists a random variable $\underline \eta\in L^1(\F_T)$ such that for each $n\geq 1$ and $t\in\T$,
\begin{equation}
\label{eq:3-BoundPenalizationForRBSDESuperBarrier}
\begin{array}{l}
\Dis\E\left[\left.\sup\limits_{s\in [t,T]}|\underline Y_s^n|^p+\left(\int_t^T
|\underline Z_s^n|^2{\rm d}s\right)^{p\over 2}+|\underline K_T^n-\underline K_t^n|^p+|\underline A_T^n-\underline A_t^n|^p\right|\F_t\right]\vspace{0.1cm}\\
\ \ \ \ \Dis +\E\left[\left.\left(\int_t^T|g(s,\underline Y_s^n,\underline Z_s^n)|{\rm d}s\right)^p\right|\F_t\right]\leq \E\left[\left.\underline \eta\right|\F_t\right].
\end{array}
\end{equation}

\item [(ii)] For each $n\geq 1$, let $(\bar Y^n_\cdot,\bar Z^n_\cdot,\bar K^n_\cdot)$ be the unique solution of $\underline{R}$BSDE $(\xi,\underline g_n+{\rm d}V,L)$ in $\s^p\times\M^p\times \vcal^{+,p}$ with $\underline g_n(t,y,z):=g(t,y,z)-n(y-U_t)^+$, i.e.,
\begin{equation}
\label{eq:3-PenalizationForRBSDEwithLowerBarrier}
\hspace*{-0.8cm}\left\{
\begin{array}{l}
\Dis\bar{Y}^n_t=\xi+\int_t^Tg(s,\bar{Y}^n_s,
\bar{Z}^n_s)
{\rm d}s+\int_t^T{\rm d}V_s+\int_t^T{\rm d}\bar K^n_s -\int_t^T{\rm d}\bar A^n_s-\int_t^T\bar{Z}^n_s\cdot {\rm d}B_s,\ \   t\in\T,\vspace{0.1cm}\\
\Dis \bar A^n_t:=n\int_0^t(
\bar{Y}^n_s-U_s)^+\ {\rm d}s,\ \ t\in\T,\\
\Dis L_t\leq \bar{Y}^n_t,\ t\in\T\ \ {\rm and} \ \int_0^T (\bar Y^n_t-L_t){\rm d}\bar K^n_t=0
\end{array}
\right.
\end{equation}
(Recall (ii) of \cref{rmk:2-ConnectionOf(HH)(AA)(H4)(H4L)} and (ii) of \cref{pro:2-ExitenceUniquenessOfBSDEsAndRBSDEs}). Then, $\bar Y^{n}_\cdot$ decreases in $n$ and is bounded below by a process $\underline X_\cdot\in \s^p$, ${\rm d}\bar K^n\leq {\rm d}\bar K^{n+1}$, and there exists a random variable $\bar\eta\in L^1(\F_T)$ such that for each $n\geq 1$ and $t\in\T$,
\begin{equation}
\label{eq:3-BoundPenalizationForRBSDElowerBarrier}
\begin{array}{l}
\Dis\E\left[\left.\sup\limits_{s\in [t,T]}|\bar Y_s^n|^p+\left(\int_t^T
|\bar Z_s^n|^2{\rm d}s\right)^{p\over 2}+|\bar K_T^n-\bar K_t^n|^p+|\bar A_T^n-\bar A_t^n|^p\right|\F_t\right]\vspace{0.1cm}\\
\ \ \ \ \Dis +\E\left[\left.\left(\int_t^T|g(s,\bar Y_s^n,\bar Z_s^n)|{\rm d}s\right)^p\right|\F_t\right] \leq \E\left[\left.\bar \eta\right|\F_t\right].
\end{array}
\end{equation}

\item [(iii)] For each $n\geq 1$, let $( Y^n_\cdot, Z^n_\cdot)$ be the unique solution of BSDE $(\xi,g_n+{\rm d}V)$ in the space $\s^p\times\M^p$ with $g_n(t,y,z):=g(t,y,z)+n(y-L_t)^- -n(y-U_t)^+$ (Recall (i) of \cref{pro:2-ExitenceUniquenessOfBSDEsAndRBSDEs}), i.e.,
\begin{equation}
\label{eq:3-PenalizationForBSDE}
\hspace*{-0.8cm}\left\{
\begin{array}{l}
\Dis{Y}^n_t=\xi+\int_t^Tg(s,{Y}^n_s,
{Z}^n_s)
{\rm d}s+\int_t^T{\rm d}V_s+\int_t^T{\rm d} K^n_s -\int_t^T{\rm d} A^n_s-\int_t^T{Z}^n_s\cdot {\rm d}B_s,\ \   t\in\T,\vspace{0.1cm}\\
\Dis  K^n_t:=n\int_0^t(
{Y}^n_s-L_s)^-\ {\rm d}s\ \ {\rm and}\ \  A^n_t:=n\int_0^t(
{Y}^n_s-U_s)^+\ {\rm d}s,\ \ t\in\T.
\end{array}
\right.
\end{equation}
Then, for each $n\geq 1$ and $t\in\T$, $\underline Y^n_t\leq Y^n_t\leq \bar Y^n_t$, ${\rm d}K^n\leq {\rm d}\underline K^n$ and ${\rm d}A^n\leq {\rm d}\bar A^n$. And, there exists a random variable $\eta\in L^1(\F_T)$ such that for each $n\geq 1$ and $t\in\T$,
\begin{equation}
\label{eq:3-BoundPenalizationForBSDE}
\begin{array}{l}
\Dis\E\left[\left.\sup\limits_{s\in [t,T]}| Y_s^n|^p+\left(\int_t^T
| Z_s^n|^2{\rm d}s\right)^{p\over 2}+| K_T^n- K_t^n|^p+| A_T^n- A_t^n|^p\right|\F_t\right]\vspace{0.1cm}\\
\ \ \ \ \Dis+\E\left[\left.\left(\int_t^T|g(s, Y_s^n, Z_s^n)|{\rm d}s\right)^p\right|\F_t\right] \leq \E\left[\left. \eta\right|\F_t\right].
\end{array}\vspace{0.2cm}
\end{equation}
\end{itemize}

\begin{proof} Let $p>1$, $V_\cdot\in\vcal^p$, the generator $g$ satisfy assumptions \ref{A:(H1)} with $\rho(\cdot)$, \ref{A:(H2)} with $\phi(\cdot)$, and \ref{A:(H3)} with $\varphi_\cdot(r)$, and assumption \ref{A:(H4)} hold for $L_\cdot,U_\cdot,\xi$ and $X_\cdot$.\vspace{0.2cm}

By representation property of Brownian filtration, we can let $(C_\cdot,H_\cdot)$ be the unique pair of processes in the space $\vcal^p\times\M^p$ such that
\begin{equation}\label{eq:3-CompositionOfX}
X_t=X_T+\int_t^T {\rm d}C_s-\int_t^T H_s\cdot {\rm d}B_s,\ \ t\in\T.
\end{equation}
It follows from (ii) of \ref{A:(H2)} that $\as$,
$$
|g(\cdot,X_\cdot,H_\cdot)|\leq |g(\cdot,X_\cdot,0)|+\phi(|H_\cdot|)\leq |g(\cdot,X_\cdot,0)|+A |H_\cdot|+A,
$$
from which together with \ref{A:(H4)} we know that $g(\cdot,X_\cdot,H_\cdot)\in \hcal^p$, and then
$$
\check{K}_t:=\int_0^t g^-(s,X_s,H_s){\rm d}s+\int_0^t {\rm d}C^{0,-}_s+\int_0^t {\rm d}V^{0,-}_s \in \vcal^{+,p}\ \
$$
and
$$
\check{A}_t:=\int_0^t g^+(s,X_s,H_s){\rm d}s+\int_0^t {\rm d}C^{0,+}_s+\int_0^t {\rm d}V^{0,+}_s\in \vcal^{+,p}.\vspace{0.1cm}
$$
where $C_\cdot-C_0=C^{0,+}_\cdot-C^{0,-}_\cdot$ and $V_\cdot-V_0=V^{0,+}_\cdot-V^{0,-}_\cdot$ with $V_\cdot^{0,+},V_\cdot^{0,-},C_\cdot^{0,+},C_\cdot^{0,-}\in \vcal^{+,p}$. Thus, the equation \eqref{eq:3-CompositionOfX} can be rewritten as
$$
X_t=\Dis X_T+\int_t^Tg(s,X_s,H_s){\rm ds}+\int_t^T{\rm d}V_s+
\int_t^T{\rm d}\check{K}_s- \int_t^T{\rm d}\check{A}_s-\int_t^TH_s\cdot {\rm d}B_s,\ \ t\in\T.
$$
On the other hand, by (i) of \cref{pro:2-ExitenceUniquenessOfBSDEsAndRBSDEs}, let $(\bar X_\cdot,\bar Z_\cdot)$ be the unique solution in $\s^p\times\M^p$ of the BSDE
$$
\bar X_t=\Dis X_T\vee \xi+\int_t^Tg(s,\bar X_s,\bar Z_s){\rm ds}+\int_t^T{\rm d}V_s\Dis +\int_t^T{\rm d}\check{K}_s-\int_t^T\bar Z_s\cdot {\rm d}B_s,\ \ t\in\T,
$$
and $(\underline{X}_\cdot,\underline{Z}_\cdot)$ be the unique solution in $\s^p\times\M^p$ of the following BSDE
$$
\underline{X}_t=\Dis X_T\wedge \xi +\int_t^Tg(s,\underline{X}_s,\underline{Z}_s){\rm ds}+\int_t^T{\rm d}V_s-\int_t^T{\rm d}\check{A}_s\Dis -\int_t^T\underline{Z}_s\cdot {\rm d}B_s,\ \ t\in\T.
$$
It follows from (i) of \cref{rmk:2-ConnectionOf(HH)(AA)(H4)(H4L)} that $g$ satisfies assumption \ref{A:(AA)}. Then, \cref{pro:2-Pro1} yields that $g(\cdot,\bar X_\cdot,\bar Z_\cdot)\in\hcal^p$ and $g(\cdot,\underline X_\cdot,\underline Z_\cdot)\in\hcal^p$, which together with \ref{A:(H2)} leads to
\begin{equation}\label{eq:3-gBarXbelongsHp}
|g(\cdot,\bar X_\cdot,0)|\leq |g(\cdot,\bar X_\cdot,\bar Z_\cdot)|+\phi(|\bar Z_\cdot|)\leq |g(\cdot,\bar X_\cdot,\bar Z_\cdot)|+A|\bar Z_\cdot|+A\in\hcal^p
\end{equation}
and
\begin{equation}\label{eq:3-gUnderlineXblongsHp}
|g(\cdot,\underline X_\cdot,0)|\leq |g(\cdot,\underline X_\cdot,\underline Z_\cdot)|+\phi(|\underline Z_\cdot|)\leq |g(\cdot,\underline X_\cdot,\underline Z_\cdot)|+A|\underline Z_\cdot|+A\in\hcal^p.\vspace{0.2cm}
\end{equation}

In what follows, for each $n\geq 1$, by (i) of \cref{pro:2-ExitenceUniquenessOfBSDEsAndRBSDEs}, let $(\dot Y^n_\cdot, \dot Z^n_\cdot)$ and $(\ddot Y^n_\cdot, \ddot Z^n_\cdot)$ be respectively the unique solution in the space $\s^p\times\M^p$ of the following BSDEs:
$$
\dot{Y}^n_t=\Dis X_T\wedge \xi +\int_t^Tg(s,\dot{Y}^n_s,\dot{Z}^n_s){\rm ds}+\int_t^T{\rm d}V_s+n\int_t^T(\dot Y^n_s-L_s)^-{\rm d}s-\int_t^T{\rm d}\check{A}_s\Dis -\int_t^T\dot{Z}^n_s\cdot {\rm d}B_s,
$$
and
$$
\ddot Y^n_t=\Dis X_T\vee \xi+\int_t^Tg(s,\ddot Y^n_s,\ddot Z^n_s){\rm ds}+\int_t^T{\rm d}V_s+\int_t^T{\rm d}\check{K}_s-n\int_t^T(\ddot Y^n_s-U_s)^+{\rm d}s\Dis -\int_t^T\ddot Z^n_s\cdot {\rm d}B_s
$$
with
$$\dot K^n_t:= n\int_0^t(\dot Y^n_s-L_s)^-{\rm d}s\ \ {\rm and}\ \ \ddot A^n_t:= n\int_0^t(\ddot Y^n_s-U_s)^+{\rm d}s,\ \ \ \ t\in\T.\vspace{0.1cm}
$$
Note that $L_\cdot\leq X_\cdot\leq U_\cdot$. We have, with $t\in\T$ and $n\geq 1$,
$$
X_t=\Dis X_T+\int_t^Tg(s,X_s,H_s){\rm ds}+\int_t^T{\rm d}V_s+n\int_t^T(X_s-L_s)^-{\rm d}s+
\int_t^T{\rm d}\check{K}_s-\int_t^T{\rm d}\check{A}_s-\int_t^TH_s\cdot {\rm d}B_s,
$$
and
$$
X_t=\Dis X_T+\int_t^Tg(s,X_s,H_s){\rm ds}+\int_t^T{\rm d}V_s+\int_t^T{\rm d} \check{K}_s-n\int_t^T(X_s-U_s)^+{\rm d}s-\int_t^T{\rm d}\check{A}_s-\int_t^TH_s\cdot {\rm d}B_s.
$$
It then follows from (ii) of \ref{A:(H4)} and (i) of \cref{pro:2-ComparisonOfUniquenessSolution} that for each $n\geq 1$,
\begin{equation}\label{eq:3-DotYnLeqXLeqU}
\underline X_t\leq \dot Y^1_t\leq \dot Y^n_t\leq X_t\leq U_t,\ \ t\in\T
\end{equation}
and
\begin{equation}\label{eq:3-DDotYnGeqXGeqL}
L_t\leq X_t\leq \ddot Y^n_t\leq \ddot Y^1_t\leq \bar X_t,\ \ t\in\T,
\end{equation}
which means that for each $n\geq 1$, \ref{A:(B1)} in \cref{pro:2-Pro2} holds for $X_\cdot$, the generator $g$ and $\dot Y^n_\cdot$, and \ref{A:(B2)} in \cref{pro:2-Pro2} holds for $X_\cdot$, the generator $g$ and $\ddot Y^n_\cdot$. Thus, in view of the fact that $g$ satisfies \ref{A:(H1)}, \ref{A:(H2')} with $f_\cdot=A$, $\mu=0$ and $\lambda=A$, and \ref{A:(H3)}, by \cref{pro:2-Pro2} together with \eqref{eq:3-DotYnLeqXLeqU} and \eqref{eq:3-DDotYnGeqXGeqL} we obtain that there exists a constant $C>0$ depending only on $p,A,T$ such that for each $t\in\T$,
\begin{equation}
\label{eq:3-DotKnLeqC}
\begin{array}{lll}
\Dis\sup\limits_{n\geq 1}\E\left[\left.|\dot K^n_T-\dot K^n_t|^p\right|\F_t\right]
&\leq &\Dis C\E\left[\sup\limits_{s\in [t,T]}|\underline X_s|^p+\sup\limits_{s\in [t,T]}|\bar X_s|^{p}+|V|^p_{t,T}+
|\check{A}_T-\check{A}_t|^p+1\right.\\
&&\hspace{1cm}\Dis+\left.\left.\left(\int_t^T |g(s,X_s,0)|\ {\rm d}s\right)^p+\left(\int_t^T |g(s,0,0)|\ {\rm d}s\right)^p\right|\F_t\right]
\end{array}
\end{equation}
and
\begin{equation}
\label{eq:3-DDotAnLeqC}
\begin{array}{lll}
\Dis\sup\limits_{n\geq 1}\E\left[\left.|\ddot A^n_T-\ddot A^n_t|^p\right|\F_t\right]
&\leq &\Dis C\E\left[\sup\limits_{s\in [t,T]}|\underline X_s|^p+\sup\limits_{s\in [t,T]}|\bar X_s|^{p} +|V|^p_{t,T}+|\check{K}_T-\check{K}_t|^p
+1\right.\\
&&\hspace{1cm}\Dis+\left.\left.\left(\int_t^T |g(s,X_s,0)|\ {\rm d}s\right)^p+\left(\int_t^T |g(s,0,0)|\ {\rm d}s\right)^p\right|\F_t\right].
\end{array}
\end{equation}
In the sequel, we will prove (i)-(iii) respectively.\vspace{0.2cm}

(i) For each $n\geq 1$, let $(\underline Y^n_\cdot,\underline Z^n_\cdot,\underline A^n_\cdot)$ be the unique solution of $\bar{R}$BSDE $(\xi,\bar g_n+{\rm d}V,U)$ in the space $\s^p\times\M^p\times \vcal^{+,p}$ with $\bar g_n(t,y,z):=g(t,y,z)+n(y-L_t)^-$, i.e., \eqref{eq:3-PenalizationForRBSDEwithSuperBarrier}.\vspace{0.2cm}

Firstly, since $\bar X_\cdot\geq L_\cdot$ by \eqref{eq:3-DDotYnGeqXGeqL}, it follows from (iii) and (i) of \cref{pro:2-ComparisonOfUniquenessSolution} that for each $n\geq 1$,
\begin{equation}\label{eq:3-UnderlineYnIncrease}
\underline Y^1_t\leq \underline Y^n_t\leq \underline Y^{n+1}_t\leq \bar X_t,\ \ t\in\T\ \ {\rm and}\ \ {\rm d}\underline A^n\leq {\rm d}\underline A^{n+1}.
\end{equation}
Then, in view of \eqref{eq:3-DotYnLeqXLeqU}, by (iii) of \cref{pro:2-ComparisonOfUniquenessSolution}
with \cref{rmk:2-ComparisonWithoutIntegrableCondition}
we deduce that for each $n\geq 1$,
\begin{equation}
\label{eq:3-UnderlineXleqDotYnLeqUnderlineYn}
\underline X_t\leq \dot Y^n_t\leq \underline Y^n_t,\ \ t\in\T,
\end{equation}
which means that for each $t\in\T$,
\begin{equation}\label{eq:3-3.18}
|\underline K^n_T-\underline K^n_t|=n\int_t^T (\underline Y^n_s-L_s)^-{\rm d}s\leq n\int_t^T (\dot Y^n_s-L_s)^-{\rm d}s=|\dot K^n_T-\dot K^n_t|.
\end{equation}
Furthermore, it follows from \eqref{eq:3-gUnderlineXblongsHp} and \eqref{eq:3-UnderlineXleqDotYnLeqUnderlineYn} that for each $n\geq 1$, \ref{A:(B2)} in \cref{pro:2-Pro2} holds for $\underline X_\cdot$, the generator $g$ and $\underline Y^n_\cdot$. Thus, in view of the fact that $g$ satisfies \ref{A:(H1)}, \ref{A:(H2')} with $f_\cdot=A$, $\mu=0$ and $\lambda=A$, and \ref{A:(H3)}, by (ii) of \cref{pro:2-Pro2} we know that for each $t\in\T$ and $n\geq 1$,
\begin{equation}
\label{eq:3-BoundOfUderlineZnAndAn}
\begin{array}{lll}
&&\Dis\E\left[\left.\left(\int_t^T |\underline Z^n_s|^2\ {\rm d}s\right)^{p\over 2}+|\underline A^n_T-\underline A^n_t|^p+\left(\int_t^T |g(s,\underline Y^n_s,\underline Z^n_s)|\ {\rm d}s\right)^p\right|\F_t\right]\vspace{0.1cm}\\
&\leq &\Dis C\E\left[\sup\limits_{s\in [t,T]}|\underline Y^n_s|^p+|V|^p_{t,T}+|\underline K^n_T-\underline K^n_t|^p+\sup\limits_{s\in [t,T]}|\underline X_s|^{p}+1\right.\vspace{0.1cm}\\
&&\hspace{1cm}\Dis+\left.\left.\left(\int_t^T |g(s,\underline X_s,0)|\ {\rm d}s\right)^p+\left(\int_t^T |g(s,0,0)|\ {\rm d}s\right)^p\right|\F_t\right].
\end{array}
\end{equation}
It then follows from \eqref{eq:3-UnderlineYnIncrease}-- \eqref{eq:3-BoundOfUderlineZnAndAn} and \eqref{eq:3-DotKnLeqC} that \eqref{eq:3-BoundPenalizationForRBSDESuperBarrier} holds with
$$
\begin{array}{l}
\Dis \underline\eta:=\bar C\left[\sup\limits_{s\in [0,T]}|\underline X_s|^{p}+\sup\limits_{s\in [0,T]}|\bar X_s|^{p}+|V|^p_T+
|\check{A}_T|^p+1\right.\vspace{0.1cm}\\
\hspace*{1.5cm}\Dis \left.+\left(\int_0^T |g(s,X_s,0)|\ {\rm d}s\right)^p+\left(\int_0^T |g(s,\underline X_s,0)|\ {\rm d}s\right)^p+\left(\int_0^T |g(s,0,0)|\ {\rm d}s\right)^p\right]
\end{array}
$$
for some constant $\bar C>0$ depending only on $p,A,T$.\vspace{0.2cm}

(ii) For each $n\geq 1$, let $(\bar Y^n_\cdot,\bar Z^n_\cdot,\bar K^n_\cdot)$ be the unique solution of $\underline{R}$BSDE $(\xi,\underline g_n+{\rm d}V,L)$ in the space $\s^p\times\M^p\times \vcal^{+,p}$ with $\underline g_n(t,y,z):=g(t,y,z)-n(y-U_t)^+$, i.e., \eqref{eq:3-PenalizationForRBSDEwithLowerBarrier}.\vspace{0.2cm}

Firstly, since $\underline X_\cdot\leq U_\cdot$ by \eqref{eq:3-DotYnLeqXLeqU}, it follows from (i) and (ii) of \cref{pro:2-ComparisonOfUniquenessSolution} that for each $n\geq 1$,
\begin{equation}\label{eq:3-BarYnDecrease}
\underline X_t\leq \bar Y^{n+1}_t\leq \bar Y^{n}_t\leq \bar Y^1_t,\ \ t\in\T\ \ {\rm and}\ \ {\rm d}\bar K^n\leq {\rm d}\bar K^{n+1}.
\end{equation}
Then, in view of \eqref{eq:3-DDotYnGeqXGeqL}, by (ii) of \cref{pro:2-ComparisonOfUniquenessSolution}
with \cref{rmk:2-ComparisonWithoutIntegrableCondition}
we deduce that for each $n\geq 1$,
\begin{equation}
\label{eq:3-BarYnleqDDotYnLeqBarX}
\bar Y^n_t\leq \ddot Y^n_t \leq \bar X_t,\ \ t\in\T,
\end{equation}
which means that for each $t\in\T$,
\begin{equation}\label{eq:3-3.22}
|\bar A^n_T-\bar A^n_t|=n\int_t^T (\bar Y^n_s-U_s)^+{\rm d}s\leq n\int_t^T (\ddot Y^n_s-U_s)^+{\rm d}s=|\ddot A^n_T-\ddot A^n_t|.
\end{equation}
Furthermore, it follows from \eqref{eq:3-gBarXbelongsHp} and \eqref{eq:3-BarYnleqDDotYnLeqBarX} that for each $n\geq 1$, \ref{A:(B1)} in \cref{pro:2-Pro2} holds for $\bar X_\cdot$, the generator $g$ and $\bar Y^n_\cdot$. Thus, in view of the fact that $g$ satisfies \ref{A:(H1)}, \ref{A:(H2')} with $f_\cdot=A$, $\mu=0$ and $\lambda=A$, and \ref{A:(H3)}, by (i) of \cref{pro:2-Pro2} we know that for each $t\in\T$ and $n\geq 1$,
\begin{equation}
\label{eq:3-BoundOfBarZnAndKn}
\begin{array}{lll}
&&\Dis\E\left[\left.\left(\int_t^T |\bar Z^n_s|^2\ {\rm d}s\right)^{p\over 2}+|\bar K^n_T-\bar K^n_t|^p+\left(\int_t^T |g(s,\bar Y^n_s,\bar Z^n_s)|\ {\rm d}s\right)^p\right|\F_t\right]\vspace{0.1cm}\\
&\leq &\Dis C\E\left[\sup\limits_{s\in [t,T]}|\bar Y^n_s|^p+|V|^p_{t,T}+|\bar A^n_T-\bar A^n_t|^p+\sup\limits_{s\in [t,T]}|\bar X_s|^{p}+1\right.\vspace{0.1cm}\\
&&\hspace{1cm}\Dis+\left.\left.\left(\int_t^T |g(s,\bar X_s,0)|\ {\rm d}s\right)^p+\left(\int_t^T |g(s,0,0)|\ {\rm d}s\right)^p\right|\F_t\right].
\end{array}
\end{equation}
It then follows from \eqref{eq:3-BarYnDecrease}-- \eqref{eq:3-BoundOfBarZnAndKn} and  \eqref{eq:3-DDotAnLeqC} that \eqref{eq:3-BoundPenalizationForRBSDElowerBarrier} holds with
$$
\begin{array}{l}
\Dis \bar\eta:=\bar C\left[\sup\limits_{s\in [0,T]}|\underline X_s|^{p}+\sup\limits_{s\in [0,T]}|\bar X_s|^{p}+|V|^p_T+
|\check{K}_T|^p+1\right.\vspace{0.1cm}\\
\hspace*{1.5cm}\Dis \left.+\left(\int_0^T |g(s,X_s,0)|\ {\rm d}s\right)^p+\left(\int_0^T |g(s,\bar X_s,0)|\ {\rm d}s\right)^p+\left(\int_0^T |g(s,0,0)|\ {\rm d}s\right)^p\right]
\end{array}
$$
for some constant $\bar C>0$ depending only on $p,A,T$.\vspace{0.2cm}

(iii) For each $n\geq 1$, let $( Y^n_\cdot, Z^n_\cdot)$ be the unique solution of BSDE $(\xi,g_n+{\rm d}V)$ in the space $\s^p\times\M^p$ with $g_n(t,y,z):=g(t,y,z)+n(y-L_t)^- -n(y-U_t)^+$, i.e.,
\eqref{eq:3-PenalizationForBSDE}.\vspace{0.2cm}

It follows from \eqref{eq:3-PenalizationForRBSDEwithLowerBarrier} and \eqref{eq:3-PenalizationForRBSDEwithSuperBarrier}
that, with $t\in\T$ and $n\geq 1$,
$$
\begin{array}{l}
\underline{Y}^n_t=\Dis \xi +\int_t^Tg(s,\underline{Y}^n_s,\underline{Z}^n_s){\rm ds}+\int_t^T{\rm d}V_s+n\int_t^T(\underline Y^n_s-L_s)^-{\rm d}s-n\int_t^T(\underline Y^n_s-U_s)^+{\rm d}s\vspace{0.1cm}\\
\Dis \hspace*{1.2cm}-\int_t^T{\rm d}\underline{A}^n_s -\int_t^T\underline{Z}^n_s\cdot {\rm d}B_s,
\end{array}
$$
and
$$
\begin{array}{l}
\bar Y^n_t=\Dis \xi+\int_t^Tg(s,\bar Y^n_s,\bar Z^n_s){\rm ds}+\int_t^T{\rm d}V_s+n\int_t^T(\bar Y^n_s-L_s)^-{\rm d}s-n\int_t^T(\bar Y^n_s-U_s)^+{\rm d}s\vspace{0.1cm}\\
\Dis \hspace*{1.2cm}+\int_t^T{\rm d}\bar{K}^n_s\Dis -\int_t^T\bar Z^n_s\cdot {\rm d}B_s.\vspace{0.1cm}
\end{array}
$$
Then, (i) of \cref{pro:2-ComparisonOfUniquenessSolution} together with \eqref{eq:3-UnderlineXleqDotYnLeqUnderlineYn} and \eqref{eq:3-BarYnleqDDotYnLeqBarX} yields that for each $n\geq 1$,
\begin{equation}
\label{eq:3-UnderXLeqDotYnLeqYnLeqDDotYnLeqBarX}
\underline X_t\leq \dot Y^n_t\leq \underline Y^n_t\leq Y^n_t\leq \bar Y^n_t\leq \ddot Y^n_t\leq  \bar X_t,\ \ t\in\T,
\end{equation}
which means that for each $t\in\T$, in view of \eqref{eq:3-3.18} and \eqref{eq:3-3.22},
\begin{equation}\label{eq:3-KnLeqDotKn}
|K^n_T-K^n_t|\leq |\underline K^n_T-\underline K^n_t| \leq |\dot K^n_T-\dot K^n_t|
\end{equation}
and
\begin{equation}\label{eq:3-AnLeqDDotAn}
|A^n_T-A^n_t|\leq |\bar A^n_T-\bar A^n_t| \leq |\ddot A^n_T-\ddot A^n_t|.
\end{equation}
Furthermore, note that $g$ satisfies assumption \ref{A:(AA)} with $\bar f_\cdot=|g(\cdot,0,0)|+2A$, $\bar\mu=A$ and $\bar\lambda=A$ by (i) of \cref{rmk:2-ConnectionOf(HH)(AA)(H4)(H4L)}. It follows from \cref{pro:2-Pro1} that there exists a constant $C'>0$ depending only on $p,A,T$ such that
for each $t\in\T$ and $n\geq 1$,
\begin{equation}
\label{eq:3-BoundOfZn}
\begin{array}{lll}
&&\Dis\E\left[\left.\left(\int_t^T |Z^n_s|^2\ {\rm d}s\right)^{p\over 2}+\left(\int_t^T |g(s,Y^n_s, Z^n_s)|\ {\rm d}s\right)^p\right|\F_t\right]\\
&\leq &\Dis C'\E\left[\left.|\xi|^p +|V|^p_{t,T} +|K^n_T-K^n_t|^p+|A^n_T-A^n_t|^p+\left(\int_t^T |g(s,0,0)|\ {\rm d}s\right)^p+1\right|\F_t\right].
\end{array}
\end{equation}
Finally, in view of \eqref{eq:3-UnderXLeqDotYnLeqYnLeqDDotYnLeqBarX}--
\eqref{eq:3-BoundOfZn} and \eqref{eq:3-DotKnLeqC}--\eqref{eq:3-DDotAnLeqC}, we can deduce that \eqref{eq:3-BoundPenalizationForBSDE} holds with
$$
\begin{array}{l}
\Dis \eta:=\bar C\left[\sup\limits_{s\in [0,T]}|\underline X_s|^{p}+\sup\limits_{s\in [0,T]}|\bar X_s|^{p}+|V|^p_T+|\check{K}_T|^p+
|\check{A}_T|^p\right.\vspace{0.1cm}\\
\hspace*{1.5cm}\Dis \left.+\left(\int_0^T |g(s,X_s,0)|\ {\rm d}s\right)^p+\left(\int_0^T |g(s,0,0)|\ {\rm d}s\right)^p+1\right]
\end{array}
$$
for some constant $\bar C>0$ depending only on $p,A,T$. The proof of \cref{pro:4-EstimateOfPenalizationEq} is then complete.
\end{proof}
\end{pro}

\begin{rmk}\label{rmk:4-RemarkOfKeyEstimate}
In view of \cref{rmk:2-ExistenceAndComparisonOfMaximalSolution}, all conclusions of \cref{pro:4-EstimateOfPenalizationEq} still hold if we replace \ref{A:(H2)} with \ref{A:(H2')}, and the expression ``the unique solution" with ``the maximal (minimal) solution" in its statement. Furthermore, \eqref{eq:3-BoundPenalizationForRBSDESuperBarrier}, \eqref{eq:3-BoundPenalizationForRBSDElowerBarrier} and \eqref{eq:3-BoundPenalizationForBSDE} still hold if we replace \ref{A:(H2)} with \ref{A:(H2')}, and the expression ``the unique solution" with ``any solution" in its statement. In fact, in this case, it is enough to let $(\underline X_\cdot,\underline Z_\cdot)$ and $(\dot Y^n_\cdot,\dot Z^n_\cdot)$ be respectively the minimal solution of the corresponding equation instead of the unique solution, and let $(\ddot Y^n_\cdot,\ddot Z^n_\cdot)$ and $(\bar X_\cdot,\bar Z_\cdot)$ be respectively the maximal solution of the corresponding equation instead of the unique solution in the procedure of proof of \cref{pro:4-EstimateOfPenalizationEq}, with omitting the comparisons between the processes indexed with $n$ and $n+1$.
\end{rmk}

\subsection{Existence and uniqueness\vspace{0.1cm}}

We are now at a position to state and prove a general existence and unique result on the $L^p$ solutions of doubly RBSDEs, which, in view of Remarks \ref{rmk:2-ConnectionOfBSDEandRBSDEandDRBSDE} and \ref{rmk:2-ConnectionOf(HH)(AA)(H4)(H4L)}, strengthens \cref{pro:2-ExitenceUniquenessOfBSDEsAndRBSDEs}.

\begin{thm}
\label{thm:4-ExistenceUnder(H2)}
Assume that $p>1$, $V_\cdot\in\vcal^p$, the generator $g$ satisfies assumptions \ref{A:(H1)}, \ref{A:(H2)} and \ref{A:(H3)}, and assumption \ref{A:(H4)} holds for $L_\cdot,U_\cdot,\xi$ and $X_\cdot$ Then, DRBSDE $(\xi,g+{\rm d}V,L,U)$ admits a unique solution $(Y_\cdot,Z_\cdot,K_\cdot,A_\cdot)$ in $\s^p\times \M^p\times \vcal^{+,p}\times \vcal^{+,p}$. Moreover, we have the following assertions.
\begin{itemize}
\item [(i)] For each $n\geq 1$, let $(\underline Y^n_\cdot,\underline Z^n_\cdot,\underline A^n_\cdot)$ be the unique solution of $\bar{R}$BSDE $(\xi,\bar g_n+{\rm d}V,U)$ in $\s^p\times\M^p\times \vcal^{+,p}$ with $\bar g_n(t,y,z):=g(t,y,z)+n(y-L_t)^-$, i.e., \eqref{eq:3-PenalizationForRBSDEwithSuperBarrier}.
    Then,
\begin{equation}
\label{eq:4-ConvergenceOfUnderlineYnZnKnAn}
\lim\limits_{n\To \infty}\left(\|\underline Y_\cdot^n- Y_\cdot\|_{\s^p}+\|\underline Z_\cdot^n- Z_\cdot\|_{\M^p}+\|\underline K_\cdot^n- K_\cdot \|_{\s^p}+\|\underline A_\cdot^n- A_\cdot \|_{\s^p}\right)=0.
\end{equation}

\item [(ii)] For each $n\geq 1$, let $(\bar Y^n_\cdot,\bar Z^n_\cdot,\bar K^n_\cdot)$ be the unique solution of $\underline{R}$BSDE $(\xi,\underline g_n+{\rm d}V,L)$ in $\s^p\times\M^p\times \vcal^{+,p}$ with $\underline g_n(t,y,z):=g(t,y,z)-n(y-U_t)^+$, i.e., \eqref{eq:3-PenalizationForRBSDEwithLowerBarrier}.
Then,
\begin{equation}
\label{eq:4-ConvergenceOfBarYnZnKnAn}
\lim\limits_{n\To \infty}\left(\|\bar Y_\cdot^n- Y_\cdot\|_{\s^p}+\|\bar Z_\cdot^n- Z_\cdot\|_{\M^p}+\|\bar K_\cdot^n- K_\cdot \|_{\s^p}+\|\bar A_\cdot^n- A_\cdot \|_{\s^p}\right)=0.
\end{equation}

\item [(iii)] For each $n\geq 1$, let $( Y^n_\cdot, Z^n_\cdot)$ be the unique solution of BSDE $(\xi,g_n+{\rm d}V)$ in the space $\s^p\times\M^p$ with $g_n(t,y,z):=g(t,y,z)+n(y-L_t)^- -n(y-U_t)^+$, i.e., \eqref{eq:3-PenalizationForBSDE}. Then,
\begin{equation}
\label{eq:4-ConvergenceOfYnZnKnAn}
\lim\limits_{n\To \infty}\left(\|Y_\cdot^n- Y_\cdot\|_{\s^p}+\|Z_\cdot^n- Z_\cdot\|_{\M^p}+\|(K_\cdot^n-A_\cdot^n)-(K_\cdot-A_\cdot)\|_{\s^p}\right)=0.
\end{equation}
\end{itemize}
\end{thm}

\begin{proof}
The uniqueness part follows from \cref{thm:3-UniquenessOfSolutions}. With regard to (i), combining Propositions \ref{pro:4-EstimateOfPenalizationEq} and \ref{pro:4-Penalization}, in view of (i) in \cref{rmk:2-ConnectionOf(HH)(AA)(H4)(H4L)}, we can deduce that there exists a quadruple $( Y_\cdot, Z_\cdot, K_\cdot, A_\cdot)\in \s^p\times\M^p\times\vcal^{+,p}\times\vcal^{+,p}$ which solves DRBSDE $(\xi,g+{\rm d}V,L,U)$,
\begin{equation}
\label{eq:4-ConvergenceOfUnderlineYnZnAn}
\lim\limits_{n\To \infty}\left(\|\underline Y_\cdot^n- Y_\cdot\|_{\s^p}+ \|\underline Z_\cdot^n- Z_\cdot\|_{\M^p}+\|\underline A_\cdot^n- A_\cdot \|_{\s^p}\right)=0,
\end{equation}
and there exists a subsequence $\{\underline K_\cdot^{n_j}\}$ of $\{\underline K_\cdot^n\}$ such that
$$
\lim\limits_{j\To\infty}\sup\limits_{t\in\T}
|\underline K_t^{n_j}- K_t|=0.
$$
Furthermore, using a similar argument to that in the proof of Theorem 5.8 of \citet{Fan2017AMS}, we obtain that
\begin{equation}
\label{eq:4-ConvergenceOfgUnderlineYnZn}
\lim\limits_{n\To\infty}\left\| \int_0^\cdot g(s,\underline Y_s^n,\underline Z_s^n){\rm d}s-\int_0^\cdot g(s,Y_s,Z_s){\rm d}s\right\|_{\s^p}=0.
\end{equation}
Then, \eqref{eq:4-ConvergenceOfUnderlineYnZnKnAn} follows from \eqref{eq:4-ConvergenceOfUnderlineYnZnAn} and \eqref{eq:4-ConvergenceOfgUnderlineYnZn}. And, \eqref{eq:4-ConvergenceOfBarYnZnKnAn} can be proved in a same way.\vspace{0.2cm}

In the sequel, we prove (iii). Firstly, it follows from \cref{pro:4-EstimateOfPenalizationEq} that
for each $n\geq 1$ and $t\in\T$, $\underline Y^n_t\leq Y^n_t\leq \bar Y^n_t$, ${\rm d}K^n\leq {\rm d}\underline K^n$ and ${\rm d}A^n\leq {\rm d}\bar A^n$. Then, \eqref{eq:4-ConvergenceOfUnderlineYnZnKnAn}  and \eqref{eq:4-ConvergenceOfBarYnZnKnAn} yield that
\begin{equation}
\label{eq:4-4.30}
\lim\limits_{n\To \infty} \| Y_\cdot^n- Y_\cdot\|_{\s^p}=0.
\end{equation}
Now, we show the convergence of the sequence $\{Z_\cdot^n\}$ in $\M^p$. Indeed, for each $n\geq 1$, observe that
$$
\begin{array}{lll}
\Dis (\bar Y_\cdot,\bar Z_\cdot,\bar V_\cdot)&
:=&\Dis (Y_\cdot^n-Y_\cdot,Z_\cdot^n-Z_\cdot,\\
&&\Dis \ \ \int_0^\cdot \left(g(s,Y_s^n,Z_s^n)-g(s,Y_s,Z_s)\right){\rm d}s+\left(K_\cdot^n-K_\cdot\right)
-\left(A_\cdot^n-A_\cdot\right))
\end{array}
$$
solves equation \eqref{eq:4-BarY=BarV}. It follows from (i) of \cref{lem:2-Lemma1} with $t=0$ and $\tau=T$ that there exists a constant $C'>0$ such that for each $n\geq 1$,
$$
\begin{array}{lll}
\Dis \|Z_\cdot^n-Z_\cdot\|_{\M^p}^p
&\leq & \Dis C'\E\left[\sup\limits_{t\in [0,T]}|Y_t^n-Y_t|^p+\sup\limits_{t\in [0,T]}\left[\left(\int_t^T (Y_s^n-Y_s)\left({\rm d}K_s^n-{\rm d}K_s\right)\right)^+\right]^{p\over 2}\right] \\
&&\Dis +C'\E\left[\sup\limits_{t\in [0,T]}\left[\left(\int_t^T (Y_s^n-Y_s)\left({\rm d}A_s-{\rm d}A_s^n\right)\right)^+\right]^{p\over 2}\right]\vspace{0.1cm}\\
&&\Dis + C'\E\left[\left(\int_{0}^T|Y^n_s-Y_s|
\left|g(s,Y_s^n,Z_s^n)-g(s,Y_s,Z_s)\right| {\rm d}s\right)^{p\over 2}\right].
\end{array}
$$
It then follows from the definitions of $K^n_\cdot$ and $A^n_\cdot$ as well as the fact of $L_\cdot\leq Y_\cdot\leq U_\cdot$ together with H\"{o}lder's inequality that
\begin{equation}
\label{eq:4-ConvergenceOfZn}
\begin{array}{lll}
\Dis\|Z_\cdot^n-Z_\cdot\|_{\M^p}^p
&\leq & \Dis C'\|Y_\cdot^n-Y_\cdot\|_{\s^p}^p+ C'\|Y_\cdot^n-Y_\cdot\|_{\s^p}^{p\over 2}\left\{\left(\E\left[|K_T|^p\right]\right)^{1\over 2}+\left(\E\left[|A_T|^p\right]\right)^{1\over 2}\right\}\\
&&\Dis +C'\|Y_\cdot^n-Y_\cdot\|_{\s^p}
^{p\over 2} \left(\E\left[\left(\int_{0}^T
\left(|g(t,Y_t^n,Z_t^n)|+|g(t,Y_t,Z_t)|\right) {\rm d}t\right)^{p}\right]\right)^{1\over 2}.
\end{array}
\end{equation}
Thus, in view of \eqref{eq:4-ConvergenceOfZn}, \eqref{eq:4-4.30}, \eqref{eq:3-BoundPenalizationForBSDE} together with \cref{pro:2-Pro1}, it follows that
\begin{equation}
\label{eq:4-4.32}
\lim\limits_{n\To \infty} \| Z_\cdot^n- Z_\cdot\|_{\M^p}=0.
\end{equation}
Furthermore, in view of \eqref{eq:4-4.30} and \eqref{eq:4-4.32}, by a similar argument to \eqref{eq:4-ConvergenceOfgUnderlineYnZn} we deduce that
\begin{equation}
\label{eq:4-4.33}
\lim\limits_{n\To\infty}\left\| \int_0^\cdot g(s,Y_s^n,Z_s^n){\rm d}s-\int_0^\cdot g(s,Y_s,Z_s){\rm d}s\right\|_{\s^p}=0.
\end{equation}
Finally, \eqref{eq:4-ConvergenceOfYnZnKnAn} follows from \eqref{eq:4-4.30}, \eqref{eq:4-4.32} and \eqref{eq:4-4.33}. \cref{thm:4-ExistenceUnder(H2)} is then proved.\vspace{0.2cm}
\end{proof}

By \cref{thm:4-ExistenceUnder(H2)} we can prove the following corollary, which together with \cref{pro:3-ComparisonTheoremofDRBSDE}, in view of Remarks \ref{rmk:2-ConnectionOfBSDEandRBSDEandDRBSDE} and \ref{rmk:2-ConnectionOf(HH)(AA)(H4)(H4L)}, strengthens \cref{pro:2-ComparisonOfUniquenessSolution}.

\begin{cor}\label{cor:4-ComparisonOfdKanddA}
Let $p>1$, $V_\cdot^1, V_\cdot^2,\in\vcal^{p}$ and both $g^1$ and $g^2$ satisfy assumptions \ref{A:(H1)}, \ref{A:(H2)} and \ref{A:(H3)}. For $i=1,2$, assume that \ref{A:(H4)} holds for $\xi^i$, $L_\cdot^i$, $U_\cdot^i$ and $X_\cdot^i$ associated with $g^i$, and that $(Y_\cdot^i,Z_\cdot^i,K_\cdot^i,A_\cdot^i)$ is the unique solution of DRBSDE $(\xi^i,g^i+{\rm d}V^i,L^i,U^i)$ in $\s^p\times \M^p\times \vcal^{+,p}\times \vcal^{+,p}$. If $\xi^1\leq \xi^2$, ${\rm d}V^1\leq {\rm d}V^2$, $L^1_\cdot=L^2_\cdot$, $U^1_\cdot=U^2_\cdot$ and
$$\as ,\ \ g^1(t,y,z)\leq g^2(t,y,z)\vspace{0.1cm}$$
for each $(y,z)\in \R\times \R^d$, then ${\rm d}K^1\geq {\rm d}K^2$ and ${\rm d}A^1\leq {\rm d}A^2$.
\end{cor}

\begin{proof}
For each $n\geq 1$ and $i=1,2$, let $(Y^{i,n}_\cdot, Z^{i,n}_\cdot, A^{i,n}_\cdot)\in \s^p\times\M^p$ be the unique solution of $\bar R$BSDE $(\xi^i,\bar g_n^i+{\rm d}V^i,U^i)$ with $\bar g_n^i(t,y,z):=g^i(t,y,z)+n(y-L_t^i)^-$.
In view of assumptions of \cref{cor:4-ComparisonOfdKanddA}, it follows from (iii) of \cref{pro:2-ComparisonOfUniquenessSolution} that for each $n\geq 1$, $Y^{1,n}_\cdot\leq Y^{2,n}_\cdot$ and ${\rm d}A^{1,n}\leq {\rm d}A^{2,n}$, and then for each progressively measurable set $D\subset \Omega\times\T$ and each $n\geq 1$, we have
$$
\E\left[\int_0^T \mathbbm{1}_{D}{\rm d}K^{1,n}_t\right]:=n\E\left[\int_0^T \mathbbm{1}_{D}(Y^{1,n}_t-L^1_t)^-{\rm d}t\right]\geq n\E\left[\int_0^T \mathbbm{1}_{D}(Y^{2,n}_t-L^2_t)^-{\rm d}t\right]=:\E\left[\int_0^T \mathbbm{1}_{D}{\rm d}K^{2,n}_t\right]
$$
and
$$
\E\left[\int_0^T \mathbbm{1}_{D}{\rm d}A^{1,n}_t\right]\leq \E\left[\int_0^T \mathbbm{1}_{D}{\rm d}A^{2,n}_t\right].
$$
Since
$$
\|K_\cdot^{1,n}-K_\cdot^1\|_{\s^p}
+\|K_\cdot^{2,n}-K_\cdot^2\|_{\s^p}\To 0
$$
and
$$
\|A_\cdot^{1,n}-A_\cdot^1\|_{\s^p}
+\|A_\cdot^{2,n}-A_\cdot^2\|_{\s^p}\To 0
$$
as $n\To \infty$ by (i) of \cref{thm:4-ExistenceUnder(H2)}, it follows that
$$
\E\left[\int_0^T \mathbbm{1}_D{\rm d}K^1_t\right]\geq \E\left[\int_0^T \mathbbm{1}_D{\rm d}K^2_t\right]\ \ {\rm and}\ \ \E\left[\int_0^T \mathbbm{1}_D{\rm d}A^1_t\right]\leq \E\left[\int_0^T \mathbbm{1}_D{\rm d}A^2_t\right],
$$
which is the desired result.
\end{proof}

At the end of this subsection, we put forward and prove a general existence result of the $L^p$ solutions for DRBSDEs under assumptions \ref{A:(H1)}, \ref{A:(H2')}, \ref{A:(H3)} and \ref{A:(H4)}.

\begin{thm}
\label{thm:4-ExistenceUnder(H2')}
Assume that $p>1$, $V_\cdot\in\vcal^p$, the generator $g$ satisfies assumptions \ref{A:(H1)}, \ref{A:(H2')} and \ref{A:(H3)}, and assumption \ref{A:(H4)} holds for $L_\cdot,U_\cdot,\xi$ and $X_\cdot$

\begin{itemize}
\item [(i)] For each $n\geq 1$, let $(\underline Y^n_\cdot,\underline Z^n_\cdot,\underline A^n_\cdot)$ be the minimal (resp.maximal) solution of $\bar{R}$BSDE $(\xi,\bar g_n+{\rm d}V,U)$ in $\s^p\times\M^p\times \vcal^{+,p}$ with $\bar g_n(t,y,z):=g(t,y,z)+n(y-L_t)^-$, i.e., \eqref{eq:3-PenalizationForRBSDEwithSuperBarrier}, recalling \cref{pro:4-EstimateOfPenalizationEq} and \cref{rmk:4-RemarkOfKeyEstimate}.
Then, DRBSDE $(\xi,g+{\rm d}V,L,U)$ admits a minimal solution (resp. a solution) $(\underline Y_\cdot,\underline Z_\cdot,\underline K_\cdot,\underline A_\cdot)$ in the space $\s^p\times \M^p\times \vcal^{+,p}\times \vcal^{+,p}$ such that
$$
\lim\limits_{n\To \infty}\left(\|\underline Y_\cdot^n- \underline Y_\cdot\|_{\s^p}+\|\underline Z_\cdot^n-\underline  Z_\cdot\|_{\M^p}+\|\underline A_\cdot^n-\underline A_\cdot \|_{\s^p}\right)=0,
$$
and there exists a subsequence $\{\underline K_\cdot^{n_j}\}$ of $\{\underline K_\cdot^n\}$ such that
$$
\lim\limits_{j\To\infty}\sup\limits_{t\in\T}
|\underline K_t^{n_j}-\underline K_t|=0.
$$

\item [(ii)] For each $n\geq 1$, let $(\bar Y^n_\cdot,\bar Z^n_\cdot,\bar K^n_\cdot)$ be the maximal (resp. minimal) solution of $\underline{R}$BSDE $(\xi,\underline g_n+{\rm d}V,L)$ in $\s^p\times\M^p\times \vcal^{+,p}$ with $\underline g_n(t,y,z):=g(t,y,z)-n(y-U_t)^+$, i.e., \eqref{eq:3-PenalizationForRBSDEwithLowerBarrier}, recalling \cref{pro:4-EstimateOfPenalizationEq} and \cref{rmk:4-RemarkOfKeyEstimate}.
Then, DRBSDE $(\xi,g+{\rm d}V,L,U)$ admits a maximal solution (resp. a solution) $(\bar Y_\cdot,\bar Z_\cdot,\bar K_\cdot,\bar A_\cdot)$ in the space $\s^p\times \M^p\times \vcal^{+,p}\times \vcal^{+,p}$ such that
$$
\lim\limits_{n\To \infty}\left(\|\bar Y_\cdot^n- \bar Y_\cdot\|_{\s^p}+\|\bar Z_\cdot^n-\bar  Z_\cdot\|_{\M^p}+\|\bar K_\cdot^n-\bar K_\cdot \|_{\s^p}\right)=0,
$$
and there exists a subsequence $\{\bar A_\cdot^{n_j}\}$ of $\{\bar A_\cdot^n\}$ such that
$$
\lim\limits_{j\To\infty}\sup\limits_{t\in\T}
|\bar A_t^{n_j}-\bar A_t|=0.
$$
\end{itemize}
\end{thm}

\begin{proof}
We only prove (i), and (ii) can be proved in the same way.\vspace{0.2cm}

In view of Remarks \ref{rmk:2-ExistenceAndComparisonOfMaximalSolution} and \ref{rmk:2-ConnectionOf(HH)(AA)(H4)(H4L)}, using Propositions \ref{pro:4-Penalization}-\ref{pro:4-EstimateOfPenalizationEq} and \cref{rmk:4-RemarkOfKeyEstimate}, we can prove that all the conclusions in (i) of \cref{thm:4-ExistenceUnder(H2')} hold expect the minimal property of the solution $(\underline Y_\cdot,\underline Z_\cdot,\underline K_\cdot,\underline A_\cdot)$ of DRBSDE $(\xi,g+{\rm d}V,L,U)$ in $\s^p\times \M^p\times \vcal^{+,p}\times \vcal^{+,p}$ when $(\underline Y^n_\cdot,\underline Z^n_\cdot,\underline A^n_\cdot)$ is the minimal solution of $\bar{R}$BSDE $(\xi,\bar g_n+{\rm d}V,U)$ in $\s^p\times\M^p\times \vcal^{+,p}$ for each $n\geq 1$. Now, we show this property.\vspace{0.2cm}

Indeed, for any solution $(Y_\cdot, Z_\cdot, K_\cdot, A_\cdot)$ of DRBSDE $(\xi,g+{\rm d}V,L,U)$ in $\s^p\times \M^p\times \vcal^{+,p}\times \vcal^{+,p}$, it is not difficult to check that $(Y_\cdot, Z_\cdot, A_\cdot)$ is a solution of $\bar{R}$BSDE $(\xi,\bar g_n+{\rm d}\bar V,U)$ in $\s^p\times \M^p\times \vcal^{+,p}$ with $\bar V_\cdot:=V_\cdot+K_\cdot$ for each $n\geq 1$. Thus, in view of the assumption that $(\underline Y^n_\cdot,\underline Z^n_\cdot,\underline A^n_\cdot)$ is the minimal solution of $\bar{R}$BSDE $(\xi,\bar g_n+{\rm d}V,U)$ in $\s^p\times\M^p\times \vcal^{+,p}$ for each $n\geq 1$, using (iii) of \cref{pro:2-ComparisonOfUniquenessSolution} together with \cref{rmk:2-ExistenceAndComparisonOfMaximalSolution}
yields that for each $n\geq 1$,
$$\underline Y^n_t\leq Y_t,\ \ t\in\T.$$
Furthermore, since $\lim\limits_{n\To \infty}\|\underline Y_\cdot^n- \underline Y_\cdot\|_{\s^p}=0$, we have
$$\underline Y_t\leq Y_t,\ \ t\in\T,$$
which is the desired result.\vspace{0.2cm}
\end{proof}

\subsection{Examples and remarks\vspace{0.2cm}}

We first introduce several examples which Theorems \ref{thm:4-ExistenceUnder(H2)} and \ref{thm:4-ExistenceUnder(H2')} can be applied to. However, to the best of our knowledge, all of their conclusions can not be obtained by any existing results.
In these examples, we always assume that $p>1$, $V_\cdot\in\vcal^p$, and \ref{A:(H4)} holds for $L_\cdot,U_\cdot,\xi$ and $X_\cdot$

\begin{ex}\label{ex:4-1}
Let the generator $$g(\omega,t,y,z)=y-y^3+|z|+\sqrt{|z|}.$$ Clearly, this $g$ satisfies assumptions \ref{A:(H1s)} with $\mu=1$, \ref{A:(H2)} with $\phi(x)=x+\sqrt{x}$, and \ref{A:(H3)}, but does not satisfy assumption \ref{A:(H2s)}. Then, in view of \cref{rmk:2-(H1s)Stronger(H1)}, it then follows from \cref{thm:4-ExistenceUnder(H2)} that DRBSDE $(\xi,g+{\rm d}V,L,U)$ admits a unique solution in $\s^p\times \M^p\times \vcal^{+,p}\times \vcal^{+,p}$.
\end{ex}

\begin{ex}\label{ex:4-2}
Let the generator
$$
g(\omega,t,y,z)=h(|y|)+e^{-y|B_t(\omega)|^2}+\cos|z|-\sqrt[3]{|z|}+{1\over \sqrt{t}}\mathbbm{1}_{t>0},
$$
where, with $\delta>0$ small enough,
$$h(x)=\left\{
\begin{array}{lll}
x|\ln x|& ,&0<x\leq \delta;\\
h'(\delta-)(x-\delta)+h(\delta)& ,&x> \delta;\\
0& ,&{\rm other\ cases}.
\end{array}\right.\vspace{0.1cm}$$
It is not very hard to verify that this $g$ satisfies assumptions \ref{A:(H1)} with $\rho(\cdot)=h(\cdot)$, \ref{A:(H2)} with $\phi(x)=x+\sqrt[3]{x}$, and \ref{A:(H3)}, but does not satisfy assumptions \ref{A:(H1s)} and \ref{A:(H2s)}. It then follows from \cref{thm:4-ExistenceUnder(H2)} that DRBSDE $(\xi,g+{\rm d}V,L,U)$ admits a unique solution in $\s^p\times \M^p\times \vcal^{+,p}\times \vcal^{+,p}$.
\end{ex}

\begin{ex}\label{ex:4-3}
Let the generator $$g(\omega,t,y,z)=-e^{y}+|z|\sin|z|.$$ It is not hard to check that this $g$ satisfies assumptions \ref{A:(H1s)} with $\mu=0$, \ref{A:(H2')} with $f_\cdot\equiv 0$, $\mu=0$ and $\lambda=1$, and \ref{A:(H3)}, but does not satisfy \ref{A:(H2)} (ii). Then, in view of \cref{rmk:2-(H1s)Stronger(H1)}, it follows from \cref{thm:4-ExistenceUnder(H2')} that DRBSDE $(\xi,g+{\rm d}V,L,U)$ admits a minimal and a maximal solutions in $\s^p\times \M^p\times \vcal^{+,p}\times \vcal^{+,p}$.\vspace{0.2cm}
\end{ex}

\begin{rmk}
\label{rmk:4-1}
With respect to this section, we would like to mention the following things.

\begin{itemize}

\item [1)] Compared with that in Proposition 4.1 of \citet{Fan2017AMS}, the assumption \eqref{eq:4-ZnKnAnLeqEta} in our \cref{pro:4-Penalization} is weaker and more natural, although some ideas of the proof are borrowed from there. And, in view of \cref{rmk:2-ConnectionOfBSDEandRBSDEandDRBSDE}, \cref{pro:4-Penalization} includes Proposition 4.1 in \citet{Fan2017AMS} as its particular case.

\item [2)] \cref{thm:4-ExistenceUnder(H2)} strengthens (v) of Theorem 6.5 in \citet{Klimsiak2013BSM}, where the generator $g$ needs to satisfy stronger assumptions \ref{A:(H1s)} and \ref{A:(H2s)} than assumptions \ref{A:(H1)} and \ref{A:(H2)} in \cref{thm:4-ExistenceUnder(H2)}. And, in view of \cref{rmk:2-ConnectionOfBSDEandRBSDEandDRBSDE}, Theorems \ref{thm:4-ExistenceUnder(H2)} and \ref{thm:4-ExistenceUnder(H2')} include respectively Theorems 5.8 and 5.11 in \citet{Fan2017AMS} as its particular case.

\item [3)] \cref{pro:4-EstimateOfPenalizationEq} is a key and difficult step to verify Theorems \ref{thm:4-ExistenceUnder(H2)} and \ref{thm:4-ExistenceUnder(H2')}. In order to prove the conclusions of \cref{pro:4-EstimateOfPenalizationEq} under general assumptions \ref{A:(H1)}-\ref{A:(H4)}, several auxiliary BSDEs need to be introduced, Propositions \ref{pro:2-Pro1}-\ref{pro:2-ComparisonOfUniquenessSolution} need to be comprehensively and repeatedly applied, and Remarks \ref{rmk:2-ComparisonWithoutIntegrableCondition} and
    \ref{rmk:2-ConnectionOf(HH)(AA)(H4)(H4L)} need to be always kept in mind.

\item [4)] In (iii) of \cref{thm:4-ExistenceUnder(H2)}, it is not clear whether the sequences of processes $\{K^n_\cdot\}$ and $\{A^n_\cdot\}$ converge respectively to $K_\cdot$ and $A_\cdot$ as $n$ tends to infinity.

\item [5)] Under the assumptions of \cref{thm:4-ExistenceUnder(H2')}, we do not know whether the minimal (resp. maximal) solution of DRBSDE $(\xi,g+{\rm d}V,L,U)$ in $\s^p\times\M^p\times \vcal^{+,p}$ can be approximated by a sequence of solutions of $\underline R$BSDEs with lower barrier $L$ (resp. $\bar{R}$BSDEs with upper barrier $U$). In particular, under the same assumptions we also do not know whether a solution of DRBSDE $(\xi,g+{\rm d}V,L,U)$ in $\s^p\times\M^p\times \vcal^{+,p}$ can be approximated by a sequence of solutions of BSDEs.
\end{itemize}
\end{rmk}

\section{Existence of the minimal (maximal) $L^p$ solution: approximation method}
\label{sec:5-Stability}
\setcounter{equation}{0}

In this section, we will establish a general approximation result for the $L^p\ (p>1)$ solutions of DRBSDEs under some elementary conditions, and study existence of the $L^p$ solution of DRBSDEs where the generator $g$ may be discontinuous and have a general growth in $y$.

\subsection{Convergence of the sequence of $L^p$ solutions for the approximation DRBSDEs\vspace{0.2cm}}

The following proposition is a general approximation result for the $L^p$ solutions of DRBSDEs under some elementary conditions.

\begin{pro} [Approximation method] \label{pro:5-Approximation}
Let $p>1$, $V_\cdot\in \vcal^p$ and \ref{A:(H4)}(i) hold for $L_\cdot$, $U_\cdot$ and $\xi$. Assume that for each $n\geq 1$, the generator $g_n$ satisfies (ii) of \ref{A:(HH)} with the same $f_\cdot, \psi_\cdot(r)$ and $\lambda$, and $(Y_\cdot^n,Z_\cdot^n,K_\cdot^n,A_\cdot^n)$ is a solution of DRBSDE $(\xi,g_n+{\rm d}V,L,U)$ in $\s^p\times\M^p\times\vcal^{+,p}\times\vcal^{+,p}$. If $Y_\cdot^n\leq Y_\cdot^{n+1}\leq \bar Y_\cdot$, ${\rm d}A^n\leq {\rm d}A^{n+1}$ and ${\rm d}K^n\geq {\rm d}K^{n+1}$ for each $n\geq 1$ and a process $\bar Y_\cdot\in \s^p$, $g_n$ tends locally uniformly in $(y,z)$ to the generator $g$ as $n\To\infty$ in the following sense:
\begin{equation}
\label{eq:5-LocalUniformConvergenceofgnFromleft}
\begin{array}{l}
{\rm For\ any\ sequence}\ \ \{(y^n,z^n)\}_{n=1}^\infty\ \ {\rm in}\ \ \R\times\R^d\ \ {\rm such\ that}\ \ \Lim(|y^n-y|+|z^n-z|)=0,\\
{\rm if}\ \ y^n\leq y\ \ {\rm for\ each}\ n\geq 1,\ \ {\rm then}\ \
\Lim g_n(t,y^n,z^n)=g(t,y,z)\ \ \as,
\end{array}
\end{equation}
and
\begin{equation}\label{eq:5-BoundofZnKnAnGn}
\sup\limits_{n\geq 1}\E\left[\left(\int_0^T|Z_t^n|^2{\rm d}t\right)^{p\over 2}+|K_T^n|^p+|A_T^n|^p+\left(\int_0^T|g_n(t,Y_t^n,Z_t^n)|
{\rm d}t\right)^p\right]<+\infty,
\end{equation}
then there exists a quadruple $(Y_\cdot,Z_\cdot,K_\cdot,A_\cdot)\in \s^p\times\M^p\times\vcal^{+,p}\times\vcal^{+,p}$ which verifies DRBSDE $(\xi,g+{\rm d}V,L,U)$ such that
$$
\lim\limits_{n\To \infty}\left(\|Y_\cdot^n-Y_\cdot\|_{\s^p}+
\|Z_\cdot^n-Z_\cdot\|_{\M^p}
+\|K_\cdot^n-K_\cdot\|_{\s^p}
+\|A_\cdot^n-A_\cdot\|_{\s^p}\right)=0.\vspace{0.2cm}
$$
\end{pro}

\begin{proof}
Since $Y_\cdot^n$ increases in $n$ and is bounded above by a process $\bar Y_\cdot$, there exists a progressively measurable real-valued process $Y_\cdot$ such that for each $t\in\T$,
\begin{equation}\label{eq:5-YnIncreaseBounded}
 Y_t^n\uparrow Y_t \ \ \ {\rm and} \ \ \ |Y_t|\vee \left(\sup_{n\geq 1}|Y_t^n|\right)\leq |Y_t^1|+|\bar Y_t|.
\end{equation}
Furthermore, since ${\rm d}A^n\leq {\rm d}A^{n+1}$ and ${\rm d}K^n\geq {\rm d}K^{n+1}$ for each $n\geq 1$, a same analysis as that in proving \eqref{eq:4-ConvergenceOfAnInSp} yields that there exist two processes $K_\cdot$ and $A_\cdot$ in $\vcal^{+,p}$ such that
\begin{equation}
\label{eq:5-ConvergenceOfKnandAnInSp}
\lim\limits_{n\To \infty}\left(\|K_\cdot^n-K_\cdot\|_{\s^p}
+\|A_\cdot^n-A_\cdot\|_{\s^p}\right)=0.
\end{equation}

For each positive integer $l,q\geq 1$, as in the proof of \cref{pro:4-Penalization}, we introduce the following two stopping times:
$$
\begin{array}{rll}
\tau_l&:=&\Dis \inf\left\{t\geq 0:\  |Y_t^1|+|\bar Y_t|+\int_0^t f_s{\rm d}s\geq l\right\}\wedge T;\vspace{0.2cm}\\
\sigma_{l,q}&:=&\Dis \inf\left\{t\geq 0:\  \int_0^t \psi_s(l) {\rm d}s\geq q\right\}\wedge \tau_l.
\end{array}
$$
Then we have
\begin{equation}
\label{eq:5-StabilityOftauandSigma}
\mathbb{P}\left(\left\{\omega:\ \exists l_0(\omega), q_0(\omega)\geq 1,\ \RE l\geq l_0(\omega), \RE q\geq q_0(\omega), \ \sigma_{l,q}(\omega)=T\right\}\right)=1.
\end{equation}
Furthermore, since all $g_n$ satisfy \ref{A:(HH)} with the same $f_\cdot, \psi_\cdot(r)$ and $\lambda$, and \eqref{eq:5-YnIncreaseBounded} holds, in view of the definitions of $\tau_l$ and $\sigma_{l,q}$, we know that $\as$, for each $l,q,n\geq 1$,
\begin{equation}
\label{eq:5-BoundOfgnYnZnWithStoppingtime}
\mathbbm{1}_{t\leq \sigma_{l,q}}|g_n(t,Y_t^n,Z_t^n)|\leq \mathbbm{1}_{t\leq \tau_{l}}f_t+\mathbbm{1}_{t\leq \sigma_{l,q}}\psi_t(l)+\lambda |Z_t^n|
\vspace{0.1cm}
\end{equation}
with
\begin{equation}
\label{eq:5-FtandPhitLeqlq}
\E\left[\int_0^T \mathbbm{1}_{t\leq \tau_{l}}f_t{\rm d}t\right]\leq l\ \ {\rm and}\ \ \E\left[\int_0^T \mathbbm{1}_{t\leq \sigma_{l,q}}\psi_t(l){\rm d}t\right]\leq q.\vspace{0.3cm}
\end{equation}
The rest proof is divided into 3 steps, which will be detailed in Appendix.\vspace{0.2cm}

{\bf Step 1.}\ Based on \eqref{eq:5-BoundofZnKnAnGn}-\eqref{eq:5-FtandPhitLeqlq}, making use of (ii) of \cref{lem:2-Lemma1}, H\"{o}lder's inequality and Lebesgue's dominated convergence theorem, we show that $\{Y_\cdot^n\}$ converges to the process $Y_\cdot$ in $\s^p$ as $n\to \infty$.\vspace{0.2cm}

{\bf Step 2.}\ By virtue of (i) of \cref{lem:2-Lemma1}, \eqref{eq:5-BoundofZnKnAnGn} and the conclusion of step 1, we show that $\{Z_\cdot^n\}$ converges to a process $Z_\cdot$ in $\M^p$ as $n\to \infty$.\vspace{0.2cm}

{\bf Step 3.}\ Making use of \eqref{eq:5-LocalUniformConvergenceofgnFromleft}-
\eqref{eq:5-FtandPhitLeqlq} and conclusions of steps 1 and 2, we show that  $(Y_\cdot,Z_\cdot,K_\cdot,A_\cdot)$ is a solution of DRBSDE $(\xi,g+{\rm d}V,L,U)$ in $\s^p\times\M^p\times\vcal^{+,p}\times\vcal^{+,p}$. \cref{pro:5-Approximation} is then proved.\vspace{0.2cm}
\end{proof}

\begin{rmk}
\label{rmk:5-ApproximationWhenYnDecrease}
We remark that the conclusion of \cref{pro:5-Approximation} still holds if we replace the expression that $Y_\cdot^n\leq Y_\cdot^{n+1}\leq \bar Y_\cdot$, ${\rm d}A^n\leq {\rm d}A^{n+1}$ and ${\rm d}K^n\geq {\rm d}K^{n+1}$ for each $n\geq 1$ and a process $\bar Y_\cdot\in \s^p$ with $\underline Y_\cdot\leq Y_\cdot^{n+1}\leq Y_\cdot^n$, ${\rm d}A^n\geq {\rm d}A^{n+1}$ and ${\rm d}K^n\leq {\rm d}K^{n+1}$ for each $n\geq 1$ and a process $\underline Y_\cdot\in \s^p$, and replace the expression that $y^n\leq y$ in \eqref{eq:5-LocalUniformConvergenceofgnFromleft} with $y^n\geq y$.\vspace{0.2cm}
\end{rmk}

\subsection{Existence of the minimal (maximal) solution\vspace{0.2cm}}

We now consider the DRBSDEs where the generator $g$ may be discontinuous and have a general growth in $y$. Let us first introduce the following assumptions introduced by \citet{FanJiang2012SPL}:

\begin{enumerate}

\renewcommand{\theenumi}{(A1a)}
\renewcommand{\labelenumi}{\theenumi}

\item \label{A:(A1a)} $g$ is left-continuous and lower semi-continuous in $y$, and continuous in $z$, i.e., $\as$, for each $(y_0,z_0)\in \R^{1+d}$, we have
$$
\lim\limits_{(-\infty,y_0]\times\R^d\ni(y,z)\To (y_0, z_0)}g(\omega,t,y,z)=g(\omega,t,y_0,z_0)
$$
and
$$
\liminf\limits_{[y_0,+\infty)\times\R^d\ni(y,z)\To (y_0,z_0)}g(\omega,t,y,z)\geq g(\omega,t,y_0,z_0).
$$

\renewcommand{\theenumi}{(A1b)}
\renewcommand{\labelenumi}{\theenumi}

\item \label{A:(A1b)} $g$ is right-continuous and upper semi-continuous in $y$, and continuous in $z$, i.e., $\as$, for each $(y_0,z_0)\in \R^{1+d}$, we have
$$
\lim\limits_{[y_0,+\infty)\times\R^d\ni(y,z)\To (y_0, z_0)}g(\omega,t,y,z)=g(\omega,t,y_0,z_0)
$$
and
$$
\limsup\limits_{(-\infty,y_0]\times\R^d\ni(y,z)\To (y_0,z_0)}g(\omega,t,y,z)\leq g(\omega,t,y_0,z_0).
$$

\renewcommand{\theenumi}{(A2)}
\renewcommand{\labelenumi}{\theenumi}

\item \label{A:(A2)} $g$ has a linear growth in $(y,z)$, i.e., there exist two constants $\tilde\mu,\tilde\lambda\geq 0$ and a process $\tilde f_\cdot\in\hcal^p$ such that $\as$, for each $(y,z)\in \R^{1+d}$,
$$|g(\omega,t,y,z)|\leq \tilde f_t(\omega)+\tilde\mu |y|+\tilde\lambda |z|.$$
\end{enumerate}

\begin{rmk}
\label{rmk:5-(HH)(i)Eqivalent(A1a)(A1b)}
It is clear that \ref{A:(HH)}(i) $\Leftrightarrow$ \ref{A:(A1a)} $+$ \ref{A:(A1b)}, and that \ref{A:(A2)} $\Leftrightarrow$ \ref{A:(H2')}(ii) $+$ \ref{A:(H3s)}.\vspace{0.2cm}
\end{rmk}

\begin{thm}
\label{thm:5-ExistenceOfDRBSDEunder(H2')(A1a)}
Assume that $p>1$, $V_\cdot\in\vcal^{p}$, $g^1$ satisfies assumptions \ref{A:(H1)}, \ref{A:(H2')} and \ref{A:(H3)}, $g^2$ satisfies assumptions \ref{A:(A1a)} (resp. \ref{A:(A1b)}) and \ref{A:(A2)}, and that $g=g^1+g^2$. Assume further that assumption \ref{A:(H4)} holds for $L_\cdot$, $U_\cdot$, $\xi$, $X_\cdot$ and $g$ (or $g^1$). Then, DRBSDE $(\xi,g+{\rm d}V,L_\cdot,U_\cdot)$ admits a minimal (resp. maximal) solution $(Y_\cdot,Z_\cdot, K_\cdot,A_\cdot)$ in the space $\s^p\times\M^p\times \vcal^{+,p}\times \vcal^{+,p}$.
\end{thm}

\begin{proof}
We only prove the case of the minimal solution. Another case can be proved in a similar way in view of \cref{rmk:5-ApproximationWhenYnDecrease}. Assume now that $p>1$, $V_\cdot\in\vcal^{p}$, $g^1$ satisfies \ref{A:(H1)} with $\rho(\cdot)$, \ref{A:(H2')} with $f_\cdot$, $\mu$ and $\lambda$, and \ref{A:(H3)} with $\psi_\cdot(r)$, $g^2$ satisfies \ref{A:(A1a)} and \ref{A:(A2)} with $\tilde f_\cdot$, $\tilde \mu$ and $\tilde\lambda$, and that $g=g^1+g^2$. Assume further that \ref{A:(H4)} holds for $L_\cdot$, $U_\cdot$, $\xi$, $X_\cdot$ and $g^1$.\vspace{0.2cm}

In view of the assumptions of $g^1$ and $g^2$ together with the proof of Theorem 1 in \citet{FanJiang2012SPL}, it is not very difficult to verify that for each $n\geq 1$ and $(y,z)\in \R\times\R^d$, the following function
$$
g_n(\omega,t,y,z):=g^1_n(\omega,t,y,z)
+g^2_n(\omega,t,y,z)
$$
with
$$
g^1_n(\omega,t,y,z):=\inf\limits_{u\in\R^d}
\left[g^1(\omega,t,y,u)+(n+2\lambda)|u-z|\right]
$$
and
$$
g^2_n(\omega,t,y,z):=\inf\limits_{(u,v)\in\R^d}
\left[g^2(\omega,t,u,v)+(n+2\tilde\mu)|u-y|
+(n+2\tilde\lambda)|v-z|\right]\vspace{0.1cm}
$$
is well defined and progressively measurable, $\as$, $g_n$ increases in $n$ and  converges locally uniformly in $(y,z)$ to the generator $g=g^1+g^2$ as $n\To \infty$ in the sense of \eqref{eq:5-LocalUniformConvergenceofgnFromleft}, all $g^1_n$ satisfy \ref{A:(H1)} with the same $\rho(\cdot)$, \ref{A:(H2s)} with $n+2\lambda$, \ref{A:(H3)} with the same $\psi_\cdot(r)+\mu r+2f_\cdot$, and $\as$, $\RE\ n\geq 1$, $\RE\ (y,z)\in\R\times\R^{d}$,
\begin{equation}
\label{eq:5-g1nSatisfy(H2')(ii)}
|g^1_n(\omega,t,y,z)-g^1(\omega,t,y,0)|\leq f_t(\omega)+\mu |y|+\lambda |z|,
\end{equation}
and all $g^2_n$ satisfy \ref{A:(H1)} with $(n+2\tilde\mu)x$, \ref{A:(H2s)} with $n+2\tilde\lambda$, and $\as$, $\RE\ n\geq 1$, $\RE\ (y,z)\in\R\times\R^{d}$,
\begin{equation}
\label{eq:5-g2nSatisfy(A2)}
|g^2_n(\omega,t,y,z)|\leq \tilde f_t(\omega)+\tilde\mu |y|+\tilde\lambda |z|.
\end{equation}
Then, in view of \eqref{eq:5-g1nSatisfy(H2')(ii)}, \eqref{eq:5-g2nSatisfy(A2)} and \ref{A:(H3)} for $g^1$, we know that $\as$, $\RE\ n\geq 1$, $\RE\ (y,z)\in\R\times\R^{d}$,
\begin{equation}
\label{eq:5-gnyzLeq}
\begin{array}{lll}
\Dis|g_n(\cdot,y,z)|&\leq & |g^1_n(\cdot,y,z)|+|g^2_n(\cdot,y,z)|\leq  |g^1(\cdot,y,0)|+f_\cdot+\mu|y|+\lambda |z|+\tilde f_\cdot+\tilde\mu |y|+\tilde\lambda |z|\\
&\leq & |g^1(\cdot,0,0)|+f_\cdot+\tilde f_\cdot+\psi_\cdot(|y|)+(\mu+\tilde\mu)|y|
+(\lambda+\tilde\lambda) |z|.
\end{array}
\end{equation}
That is to say, all $g_n$ satisfy \ref{A:(HH)} with the same parameters.\vspace{0.2cm}

Note that for each $n\geq 1$, $g_n$ satisfies \ref{A:(H1)}, \ref{A:(H2)} and \ref{A:(H3)} by \cref{rmk:2-(H1s)Stronger(H1)} and \cref{rmk:5-(HH)(i)Eqivalent(A1a)(A1b)} together with \eqref{eq:5-g2nSatisfy(A2)}. Furthermore, by \eqref{eq:5-gnyzLeq} we know that for each $n\geq 1$, $g_n(\cdot,X_\cdot,0)\in\hcal^p$ and then \ref{A:(H4)} holds for $L_\cdot$, $U_\cdot$, $\xi$, $X_\cdot$ and $g_n$. It then follows from \cref{thm:4-ExistenceUnder(H2)} that there exists a unique solution $(Y_\cdot^n,Z_\cdot^n, K_\cdot^n, A_\cdot^n)$ of DRBSDE $(\xi,g_n+{\rm d}V,L,U)$ in the space $\s^p\times\M^p\times \vcal^{+,p}\times \vcal^{+,p}$ for each $n\geq 1$. And, noticing that $g_n$ increases in $n$, by \cref{cor:3-CorollaryOfComparisonTheormen} and \cref{cor:4-ComparisonOfdKanddA} we can deduce that $Y_\cdot^n\leq Y_\cdot^{n+1}$, ${\rm d}A^n\leq {\rm d}A^{n+1}$ and ${\rm d}K^n\geq {\rm d}K^{n+1}$ for each $n\geq 1$.\vspace{0.2cm}

In the sequel, we show that inequality \eqref{eq:5-BoundofZnKnAnGn} appearing in \cref{pro:5-Approximation} holds. In fact,
let
$$
\underline{g}(\cdot,y,z):=g^1(\cdot,y,0)-(f_\cdot+\tilde f_\cdot)-(\mu+\tilde\mu)|y|-(\lambda+\tilde\lambda)|z|
$$
and
$$
\bar g(\cdot,y,z):=g^1(\cdot,y,0)+(f_\cdot+\tilde f_\cdot)+(\mu+\tilde\mu)|y|+(\lambda+\tilde\lambda)|z|.
\vspace{0.2cm}
$$
It then follows from \eqref{eq:5-g1nSatisfy(H2')(ii)} and \eqref{eq:5-g2nSatisfy(A2)} that $\underline{g}\leq g_n\leq \bar g$ for each $n\geq 1$, and
both $\underline{g}$ and $\bar g$ satisfy assumptions \ref{A:(H1)}, \ref{A:(H2s)} and \ref{A:(H3)} with
$$
\underline{g}(\cdot, X_\cdot,0)=g^1(\cdot, X_\cdot,0)-(f_\cdot+\tilde f_\cdot)-(\mu+\tilde\mu)|X_\cdot|\in \hcal^p
$$
and
$$
\bar g(\cdot, X_\cdot,0)=g^1(\cdot, X_\cdot,0)+(f_\cdot+\tilde f_\cdot)+(\mu+\tilde\mu)|X_\cdot|\in \hcal^p.\vspace{0.1cm}
$$
Thus, it follows from \cref{thm:4-ExistenceUnder(H2)} that DRBSDE $(\xi,\underline{g}+{\rm d}V,L,U)$ and DRBSDE $(\xi,\bar g+{\rm d}V,L,U)$ admit respectively a unique solution $(\underline{Y}_\cdot,\underline{Z}_\cdot, \underline{K}_\cdot,\underline{A}_\cdot)$ and $(\bar Y_\cdot,\bar Z_\cdot, \bar K_\cdot, \bar A_\cdot)$ in the space $\s^p\times\M^p\times \vcal^{+,p}\times \vcal^{+,p}$, and by \cref{cor:3-CorollaryOfComparisonTheormen} and \cref{cor:4-ComparisonOfdKanddA}, we have that for each $t\in \T$ and $n\geq 1$,
\begin{equation}
\label{eq:5-BoundOfYnKnAn}
\underline{Y}_t\leq Y^n_t\leq \bar Y_t,\ \ |\bar K_T-\bar K_t|\leq |K^n_T-K^n_t|\leq |\underline K_T-\underline K_t|\ \ {\rm and}\ \  |\underline{A}_T-\underline{A}_t|\leq |A^n_T-A^n_t|\leq |\bar A_T-\bar A_t|.
\end{equation}
Furthermore, note that for each $n\geq 1$, $g_n$ satisfies assumption \ref{A:(AA)} with the same $\bar f_\cdot=|g^1(\cdot,0,0)|+f_\cdot+\tilde f_\cdot+A$, $\bar\mu=\tilde\mu+A$ and $\bar\lambda=\lambda+\tilde\lambda$ since $g^n_1$ satisfies assumption \ref{A:(H1)} with the same $\rho(\cdot)$, and inequalities \eqref{eq:5-g1nSatisfy(H2')(ii)} and \eqref{eq:5-g2nSatisfy(A2)} hold. It follows from \cref{pro:2-Pro1} that there exists a positive constant $C>0$ depending only on $p,A,\tilde\mu,\tilde\lambda,T$ such that
for each $t\in\T$ and $n\geq 1$,
\begin{equation}
\label{eq:5-BoundOfZngn}
\begin{array}{lll}
&&\Dis\E\left[\left.\left(\int_t^T |Z^n_s|^2\ {\rm d}s\right)^{p\over 2}+\left(\int_t^T |g_n(s,Y^n_s, Z^n_s)|\ {\rm d}s\right)^p\right|\F_t\right]\vspace{0.1cm}\\
&\leq &\Dis C\E\left[\left.|\xi|^p +|V|^p_{t,T} +|K^n_T-K^n_t|^p+|A^n_T-A^n_t|^p+\left(\int_t^T \bar f_s\ {\rm d}s\right)^p+1\right|\F_t\right].
\end{array}
\end{equation}
Finally, combining \eqref{eq:5-BoundOfYnKnAn} and \eqref{eq:5-BoundOfZngn} yields that \eqref{eq:5-BoundofZnKnAnGn} in \cref{pro:5-Approximation} holds.\vspace{0.2cm}

Up to now, we have checked all conditions in \cref{pro:5-Approximation}. It then follows from \cref{pro:5-Approximation} that DRBSDE $(\xi,g+{\rm d}V,L,U)$ admits a solution $(Y_\cdot,Z_\cdot,K_\cdot,A_\cdot)$ in $\s^p\times\M^p\times\vcal^{+,p}\times\vcal^{+,p}$ such that
\begin{equation}
\label{eq:5-ConvergenceOfYnZnKnAn}
\lim\limits_{n\To\infty}(\|Y_\cdot^n
-Y_\cdot\|_{\s^p}+\|Z_\cdot^n
-Z_\cdot\|_{\M^p}+\|K_\cdot^n
-K_\cdot\|_{\s^p}+\|A_\cdot^n
-A_\cdot\|_{\s^p})=0.\vspace{0.2cm}
\end{equation}

Finally, let us show that $(Y_\cdot,Z_\cdot,K_\cdot,A_\cdot)$ is just the minimal solution of RBSDE $(\xi,g+{\rm d}V,L,U)$ in the space $\s^p\times\M^p\times\vcal^{+,p}\times\vcal^{+,p}$. In fact, if $(Y'_\cdot,Z'_\cdot,K'_\cdot,A'_\cdot)$ is also a solution of DRBSDE $(\xi,g+{\rm d}V,L,U)$ in the space $\s^p\times\M^p\times\vcal^{+,p}\times\vcal^{+,p}$, then noticing that for each $n\geq 1$, $g_n\leq g$ and $g_n$ satisfies \ref{A:(H1)} and \ref{A:(H2)}, it follows from \cref{cor:3-CorollaryOfComparisonTheormen} that for each $n\geq 1$ and $t\in\T$, $Y_t^n\leq Y'_t$. Thus, by \eqref{eq:5-ConvergenceOfYnZnKnAn} we obtain that $Y_t\leq Y'_t$ for each $t\in \T$, which is the desired result. \cref{thm:5-ExistenceOfDRBSDEunder(H2')(A1a)} is then proved.\vspace{0.2cm}
\end{proof}

By \cref{cor:3-CorollaryOfComparisonTheormen}, \cref{cor:4-ComparisonOfdKanddA} and the proof of \cref{thm:5-ExistenceOfDRBSDEunder(H2')(A1a)}, it is not hard to verify the following comparison theorem for the minimal (resp. maximal) $L^p$ solutions of DRBSDEs.

\begin{pro}
\label{pro:5-ComparisonUnder(H2')(A1a)}
Let $p>1$ and for $i=1,2$, assume that $V_\cdot^i\in \vcal^p$, $g^{1,i}$ satisfies \ref{A:(H1)}, \ref{A:(H2')} and \ref{A:(H3)}, $g^{2,i}$ satisfies \ref{A:(A1a)} (resp. \ref{A:(A1b)}) and \ref{A:(A2)}, and that $g^i:=g^{1,i}+g^{2,i}$. Assume further that for $i=1,2$, \ref{A:(H4)} holds for $\xi^i$, $L_\cdot^i$, $U_\cdot^i$, and $X_\cdot^i$ associated with $g^i$ (or $g^{1,i}$), and $(Y_\cdot^i,Z_\cdot^i,K_\cdot^i,A_\cdot^i)$ is the minimal (resp. maximal) solution of DRBSDE $(\xi^i,g^i+{\rm d}V^i,L^i,U^i)$ in the space $\s^p\times \M^p\times \vcal^{+,p}\times \vcal^{+,p}$. If $\xi^1\leq \xi^2$, ${\rm d}V^1\leq {\rm d}V^2$, $L^1_\cdot\leq L^2_\cdot$, $U^1_\cdot\leq U^2_\cdot$, and for each $(y,z)\in \R\times \R^d$,
$$\as,\ \ g^{1,1}(t,y,z)\leq g^{1,2}(t,y,z)\ \ {\rm and}\ \ g^{2,1}(t,y,z)\leq g^{2,2}(t,y,z),$$
then $Y_t^1\leq Y_t^2$, $t\in \T$. Furthermore, if $L^1_\cdot=L^2_\cdot$ and $U^1_\cdot=U^2_\cdot$, then ${\rm d}K^1\geq {\rm d}K^2$ and ${\rm d}A^1\leq {\rm d}A^2$.\vspace{0.2cm}
\end{pro}

\subsection{Examples and remarks\vspace{0.2cm}}

We introduce two examples which \cref{thm:5-ExistenceOfDRBSDEunder(H2')(A1a)} can be applied to, but any existing results can not be applied to.

\begin{ex}
Let the generator $g:=g^1+g^2$ with
$$
g^1(\omega,t,y,z)=\bar h(|y|)-e^{y|B_t(\omega)|}+(e^{-y}\wedge 1) |z|\sin |z|+{1\over \sqrt[4]{t}}\mathbbm{1}_{t>0}
$$
and
$$
g^2(\omega, t,y,z)=\mathbbm{1}_{y\leq 0}\sqrt[3]{|y|} +\mathbbm{1}_{y>0}\cos y+\sqrt{|y| |z|}+|B_t(\omega)|,
$$
where, with $\delta>0$ small enough,
$$\bar h(x)=\left\{
\begin{array}{lll}
x|\ln x|\ln|\ln x|& ,&0<x\leq \delta;\\
\bar h'(\delta-)(x-\delta)+\bar h(\delta)& ,&x> \delta;\\
0& ,&{\rm other\ cases}.
\end{array}\right.\vspace{0.1cm}$$
It is not very hard to verify that $g^1$ satisfies \ref{A:(H1)} with $\rho(\cdot)=\bar h(\cdot)$, \ref{A:(H2')} with $f_\cdot\equiv 0$, $\mu=0$ and $\lambda=1$, and \ref{A:(H3)}, and that $g^2$ satisfies \ref{A:(A1a)} and \ref{A:(A2)} with $\tilde f_\cdot=|B_\cdot|+2$, $\tilde\mu=2$ and $\tilde\lambda=1$ for each $p>1$. Thus, if \ref{A:(H4)} holds for $g^1$ and some $\xi$, $L_\cdot$, $U_\cdot$, $X_\cdot$ and $p>1$, and $V_\cdot\in \vcal^p$, then by \cref{thm:5-ExistenceOfDRBSDEunder(H2')(A1a)} we know that DRBSDE $(\xi,g+{\rm d}V,L_\cdot,U_\cdot)$ admits a minimal solution in $\s^p\times\M^p\times\vcal^{+,p}\times\vcal^{+,p}$.
\end{ex}

\begin{ex}
Let the generator
$$g(\omega, t,y,z)=e^{-y}\mathbbm{1}_{y\geq 0}+\sqrt[3]{|z|}.$$
It is easy to check that this $g$ satisfies \ref{A:(A1b)} and \ref{A:(A2)} with $\tilde f_\cdot=2$, $\tilde\mu=0$ and $\tilde\lambda=1$ for each $p>1$. Thus, if \ref{A:(H4)} holds for $g$ and some $\xi$, $L_\cdot$, $U_\cdot$, $X_\cdot$ and $p>1$, and $V_\cdot\in \vcal^p$, it then follows from \cref{thm:5-ExistenceOfDRBSDEunder(H2')(A1a)} that DRBSDE $(\xi,g+{\rm d}V,L_\cdot,U_\cdot)$ admits a maximal solution in $\s^p\times\M^p\times\vcal^{+,p}\times\vcal^{+,p}$.\vspace{0.2cm}
\end{ex}

\begin{rmk}
\label{rmk:5-1}
With respect to this section, we would like to mention the following things.

\begin{itemize}

\item [1)] Compared with that in Proposition 4.2 of \citet{Fan2017AMS}, the assumptions \eqref{eq:5-LocalUniformConvergenceofgnFromleft} and \eqref{eq:5-BoundofZnKnAnGn} in \cref{pro:5-Approximation} is weaker and more natural, although some ideas of the proof are borrowed from there. And, in view of \cref{rmk:2-ConnectionOfBSDEandRBSDEandDRBSDE}, \cref{pro:5-Approximation} together with \cref{rmk:5-ApproximationWhenYnDecrease} includes Proposition 4.2 in \citet{Fan2017AMS} as its particular case.

\item [2)] In view of \cref{rmk:2-ConnectionOfBSDEandRBSDEandDRBSDE}, \cref{thm:5-ExistenceOfDRBSDEunder(H2')(A1a)} and \cref{pro:5-ComparisonUnder(H2')(A1a)} include respectively
Theorem 5.13 and Proposition 5.15 in \citet{Fan2017AMS} as its special case . In particular, \cref{thm:5-ExistenceOfDRBSDEunder(H2')(A1a)} and \cref{pro:5-ComparisonUnder(H2')(A1a)} also consider the case that the generator $g$ may be discontinuous in $y$.

\item [3)] It follows from Remarks \ref{rmk:2-ConnectionOfBSDEandRBSDEandDRBSDE},
\ref{rmk:2-(H1s)Stronger(H1)}, \ref{rmk:2-ConnectionOf(HH)(AA)(H4)(H4L)} and \ref{rmk:5-(HH)(i)Eqivalent(A1a)(A1b)} that since the associated assumptions are more general, \cref{thm:5-ExistenceOfDRBSDEunder(H2')(A1a)} and \cref{pro:5-ComparisonUnder(H2')(A1a)} strengthen and unify some existing corresponding results on DRBSDEs, RBSDEs with one continuous barrier, and non-reflected BSDEs obtained, for example, in \citet{BriandDelyonHu2003SPA}, \citet{ElAsriHamadeneWang2011SAA}, \citet{Fan2017AMS}, \citet{FanJiang2012SPL}, \citet{HamadeneLepeltierMatoussi1997BSDEs}, \citet{HamadenePopier2012SD}, \citet{Klimsiak2012EJP}, \citet{MaFanSong2013BSM} and \citet{RozkoszSlominski2012SPA}.
\end{itemize}
\end{rmk}

\appendix
\section{}
\renewcommand{\appendixname}{}


In this section, we will supply the details omitted in the proof procedures of \cref{pro:4-Penalization} and \cref{pro:5-Approximation}.\vspace{0.2cm}

{\bf Complementary of the details for the proof of \cref{pro:4-Penalization}}. Now, we will detail the proof of steps 1-7 after the inequality \eqref{eq:4-IntegralOfftAndPsitleqlq}.\vspace{0.2cm}

{\bf Step 1.}\ We show that $Y_\cdot$ is a c\`{a}dl\`{a}g process. Let us fix a pair of $l,q\geq 1$ arbitrarily. From \eqref{eq:4-BoundOfgYnZn} and \eqref{eq:4-IntegralOfftAndPsitleqlq} together with \eqref{eq:4-ZnKnAnLeqEta}, it follows that there exists a subsequence $\{h^{n_j;l,q}_\cdot\}_{j=1}^{\infty}$ of the sequence $\{h^{n; l,q}_\cdot\}_{n=1}^{\infty}$ which converges weakly to a process $h^{l,q}_\cdot$ in the space $\hcal^1$.  Now, take any bounded
linear functional $\Phi(\cdot)$ defined on $\mathbb{L}^1(\F_T)$. Then there exists a constant $b>0$ such that for each $\overline{h}_\cdot \in \hcal^1$ and each stopping time $\bar\tau$ taking values in $\T$, we have
$$
\left|\Phi(\int_0^{\bar\tau}\overline h_s {\rm d}s)\right|\leq b\left\|\int_0^{\bar\tau}\overline h_s {\rm d}s\right\|_{\mathbb{L}^1}\leq b\left\|\overline h\right\|_{\hcal^1}.
$$
Hence, for each stopping time $\bar\tau$ taking values in $\T$, $\Phi(\int_0^{\bar\tau}\cdot \ {\rm d}s )$ is a bounded linear functional defined on $\hcal^1$, which means that
$$\lim\limits_{j\To \infty} \Phi\left(\int_0^{\bar\tau} h_s^{n_j;l,q} {\rm d}s \right)=\Phi\left(\int_0^{\bar\tau} h^{l,q}_s {\rm d}s \right).\vspace{0.1cm}$$
As a result, for every stopping time $\tau$ with $0\leq\tau\leq \sigma_{l,q}$, as $j\To \infty$,
\begin{equation}
\label{eq:4-WeaklyConvergenceOfGYnZn}
\int_0^\tau g(s,Y_s^{n_j},Z_s^{n_j}) {\rm d}s=\int_0^\tau h_s^{n_j;l,q} {\rm d}s\ \To\  \int_0^\tau h^{l,q}_s {\rm d}s\ \ {\rm weakly\ in}\ \mathbb{L}^1(\F_T).
\end{equation}
Furthermore, it follows from \eqref{eq:4-ZnKnAnLeqEta} and Lemma 4.4 of \citet{Klimsiak2012EJP} that there exists a process $Z_\cdot\in \M^p$ together with a subsequence of the sequence $\{n_j\}_{j=1}^{\infty}$, still denoted by itself, such that for any stopping time $\bar\tau$ taking values in $\T$, as $j\To \infty$,
\begin{equation}
\label{eq:4-WeaklyConvergenceOfZn}
\int_0^{\bar\tau} Z_s^{n_j}{\rm d}B_s\To \int_0^{\bar\tau} Z_s\cdot {\rm d}B_s\ \ {\rm weakly\ in}\ \mathbb{L}^p(\F_T)\ {\rm and \ then\ in}\ \mathbb{L}^1(\F_T).
\end{equation}
In the sequel, we define
$$
K^{l,q}_t:=Y_0-Y_t-\int_0^t h^{l,q}_s{\rm d}s-\int_0^t {\rm d}V_s-\int_0^t {\rm d}A_s+\int_0^t Z_s\cdot {\rm d}B_s, \ \ t\in\T.
$$
Then, in view of \eqref{eq:4-PenalizationForRBSDEwithSuperBarrier}, \eqref{eq:4-ConvergenceOfAnInSp}, \eqref{eq:4-WeaklyConvergenceOfGYnZn}, \eqref{eq:4-WeaklyConvergenceOfZn} and the fact that for each stopping time $\bar\tau$ taking values in $\T$, $Y^n_{\bar\tau}\uparrow Y_{\bar\tau}$ in $\mathbb{L}^1(\F_T)$, we can deduce that for every stopping time $\tau$ with $0\leq\tau\leq \sigma_{l,q}$, the sequence
$$
K_\tau^{n_j}=Y_0^{n_j}-Y_\tau^{n_j}-\int_0^\tau
g(s,Y_s^{n_j},Z_s^{n_j}){\rm d}s-\int_0^\tau{\rm d}V_s-\int_0^\tau {\rm d}A_s^{n_j}+\int_0^\tau Z_s^{n_j}{\rm d}B_s
$$
converges weakly to $K^{l,q}_\tau$ in $\mathbb{L}^1(\F_T)$ as $j\To \infty$. Thus, since $K_\cdot^n\in \vcal^{+,p}$ for each $n\geq 1$, we have
$$
K^{l,q}_{\tau_1\wedge \sigma_{l,q}}\leq K^{l,q}_{\tau_2\wedge \sigma_{l,q}}
$$
for any stopping times $\tau_1\leq \tau_2$ valued in $\T$. Furthermore, in view of the definition of $K_\cdot^{l,q}$ as well as the facts that $V_\cdot\in\vcal^p$, $A_\cdot\in\vcal^{+,p}$, $Y^n_\cdot\uparrow Y_\cdot$ and $Y^n_\cdot\in \s^p$ for each $n\geq 1$, it is not very hard to verify that $K^{l,q}_{\cdot}$ is a optional process with $\ps$ upper semi-continuous paths. Thus, by Lemma A.3 in \citet{BayraktarYao2015SPA} we know that $K^{l,q}_{\cdot\wedge \sigma_{l,q}}$ is nondecreasing, and should have $\ps$ right lower semi-continuous paths. Hence, $K^{l,q}_{\cdot\wedge \sigma_{l,q}}$ is c\`{a}dl\`{a}g and so is $Y_{\cdot\wedge \sigma_{l,q}}$ from the definition of $K^{l,q}_\cdot$. Finally, it follows from \eqref{eq:4-StabilityOfTauSigma} that $Y_\cdot$ is also a c\`{a}dl\`{a}g process.\vspace{0.2cm}

{\bf Step 2.}\ We show that $Y_t\geq L_t$ for each $t\in \T$ and as $n\To \infty$,
\begin{equation}
\label{eq:4-UniformConvergenceOfYnMinusL}
\sup\limits_{t\in\T}(Y^n_t-L_t)^-\To 0.
\end{equation}
In fact, it follows from \eqref{eq:4-ZnKnAnLeqEta} and the definition of $K^n_\cdot$ that for each $n\geq 1$,
$$\E\left[\left(\int_0^T(Y^n_t-L_t)^-{\rm d}t\right)^p\right]\leq {\sup_{n\geq 1}\E[|K_T^n|^p]\over n^p}.$$
Hence, by Fatou's lemma and H\"{o}lder's inequality, we have
$$
\E\left[\int_0^T(Y_t-L_t)^-{\rm d}t\right]\leq \liminf\limits_{n\To\infty}\E\left[\int_0^T(Y^n_t-L_t)^-{\rm d}t\right]\leq \lim\limits_{n\To\infty}{(\sup_{n\geq 1}\E[|K_T^n|^p])^{1\over p}\over n}=0,
$$
which implies that $$\E\left[\int_0^T(Y_t-L_t)^-{\rm d}t\right]=0.$$ Since $Y_\cdot-L_\cdot$ is a c\`{a}dl\`{a}g process, we get that $(Y_t-L_t)^-=0$ and hence
$Y_t\geq L_t$ for each $t\in [0,T)$. Moreover, note that $Y_T=Y^n_T=\xi\geq L_T$. Hence $$(Y^n_t-L_t)^-\downarrow 0$$
for each $t\in [0,T]$, and \eqref{eq:4-UniformConvergenceOfYnMinusL} follows by Dini's theorem. \vspace{0.2cm}

{\bf Step 3.}\ We show the convergence of the sequence $\{Y_\cdot^n\}$ in $\s^p$. Let $\tau_l$ and $\sigma_{l,q}$ be the sequences of stopping times defined in Step 1. For each $n,m\geq 1$, observe from \eqref{eq:4-PenalizationForRBSDEwithSuperBarrier} that
\begin{equation}
\label{eq:4-DefinitionOfBarYBarZBarV}
\begin{array}{lll}
\Dis (\bar Y_\cdot,\bar Z_\cdot,\bar V_\cdot)&
:=&\Dis (Y_\cdot^n-Y_\cdot^m,Z_\cdot^n-Z_\cdot^m,\\
&&\Dis \ \ \int_0^\cdot \left(g(s,Y_s^n,Z_s^n)-g(s,Y_s^m,Z_s^m)\right){\rm d}s+\left(K_\cdot^n-K_\cdot^m\right)
-\left(A_\cdot^n-A_\cdot^m\right))
\end{array}
\end{equation}
solves equation \eqref{eq:4-BarY=BarV}. It then follows from (ii) of \cref{lem:2-Lemma1} with $p=2$, $t=0$ and $\tau=\sigma_{l,q}$ that there exists a constant $C>0$ such that for each $n,m,l,q\geq 1$,
\begin{equation}\label{eq:4-BoundOfYnMinusYm}
\begin{array}{ll}
&\Dis\E\left[\sup\limits_{t\in [0,T]} |Y_{t\wedge\sigma_{l,q}}^n
-Y_{t\wedge\sigma_{l,q}}^m|^2\right]\\
\leq &\Dis C\E\left[|Y_{\sigma_{l,q}}^n-Y_{\sigma_{l,q}}^m|^2 +\sup\limits_{t\in [0,T]}\left(\int_{t\wedge \sigma_{l,q}}^{\sigma_{l,q}}(Y^n_s-Y_s^m)
\left({\rm d}K_s^n-{\rm d}K_s^m\right) \right)^+\right.\\
&\hspace{1cm}\Dis +\sup\limits_{t\in [0,T]}\left(\int_{t\wedge \sigma_{l,q}}^{\sigma_{l,q}}(Y^n_s-Y_s^m)
\left({\rm d}A_s^m-{\rm d}A_s^n\right) \right)^+\vspace{0.2cm}\\
&\hspace{1cm}\Dis +\left. \int_{0}^{\sigma_{l,q}}|Y^n_t-Y_t^m|
\left|g(t,Y_t^n,Z_t^n)-g(t,Y_t^m,Z_t^m)\right| {\rm d}t\right].
\end{array}
\end{equation}
Furthermore, in view of the definition of $K_\cdot^n$ and $A_\cdot^n$ with \eqref{eq:4-PenalizationForRBSDEwithSuperBarrier}, we deduce that for each $t\in \T$,
\begin{equation}\label{eq:4-YnYmTimesKnKmLeq}
\begin{array}{ll}
&\Dis \int_{t\wedge \sigma_{l,q}}^{\sigma_{l,q}}(Y^n_s-Y_s^m)
\left({\rm d}K_s^n-{\rm d}K_s^m\right)\vspace{0.1cm}\\
= &\Dis \int_{t\wedge \sigma_{l,q}}^{\sigma_{l,q}}\left[(Y^n_s-L_s)
-(Y_s^m-L_s)\right]
{\rm d}K_s^n-\int_{t\wedge \sigma_{l,q}}^{\sigma_{l,q}}\left[(Y^n_s-L_s)
-(Y_s^m-L_s)\right]
{\rm d}K_s^m\vspace{0.1cm}\\
\leq &\Dis \int_{t\wedge \sigma_{l,q}}^{\sigma_{l,q}}(Y_s^m-L_s)^-
{\rm d}K_s^n+\int_{t\wedge \sigma_{l,q}}^{\sigma_{l,q}}(Y_s^n-L_s)^-{\rm d}K_s^m\vspace{0.1cm}\\
\leq &\Dis \sup\limits_{t\in\T}(Y_{t\wedge\sigma_{l,q} }^m-L_{t\wedge\sigma_{l,q}})^-|K_T^n|+
\sup\limits_{t\in\T}(Y_{t\wedge\sigma_{l,q} }^n-L_{t\wedge\sigma_{l,q}})^-|K_T^m|
\end{array}
\end{equation}
and
\begin{equation}\label{eq:4-YnYmTimesAmAnLeq}
\begin{array}{lll}
\Dis \int_{t\wedge \sigma_{l,q}}^{\sigma_{l,q}}(Y^n_s-Y_s^m)
\left({\rm d}A_s^m-{\rm d}A_s^n\right)&
= & \Dis \int_{t\wedge \sigma_{l,q}}^{\sigma_{l,q}}\left[(U_s-Y_s^m)
-(U_s-Y_s^n)\right]\left({\rm d}A_s^m-{\rm d}A_s^n\right)\vspace{0.2cm}\\
&= & \Dis -\int_{t\wedge \sigma_{l,q}}^{\sigma_{l,q}}(U_s-Y_s^m)
{\rm d}A_s^n-\int_{t\wedge \sigma_{l,q}}^{\sigma_{l,q}}(U_s-Y_s^n)
{\rm d}A_s^m\\
&\leq & 0.
\end{array}
\end{equation}
Combining \eqref{eq:4-BoundOfgYnZn}, \eqref{eq:4-BoundOfYnMinusYm}, \eqref{eq:4-YnYmTimesKnKmLeq}, \eqref{eq:4-YnYmTimesAmAnLeq} and H\"{o}lder's inequality leads to that for each $n,m,l,q\geq 1$, \vspace{0.1cm}
$$
\hspace*{-4cm}\begin{array}{ll}
&\Dis\E\left[\sup\limits_{t\in [0,T]} |Y_{t\wedge\sigma_{l,q}}^n
-Y_{t\wedge\sigma_{l,q}}^m|^2\right]\vspace{0.1cm}\\
\leq &\Dis C\E\left[|Y_{\sigma_{l,q}}^n-Y_{\sigma_{l,q}}^m|^2 +2\int_0^T |Y^n_t-Y_t^m|\left(\mathbbm{1}_{t\leq \tau_l}f_t+\mathbbm{1}_{t\leq\sigma_{l,q}}
\psi_t(l)\right){\rm d}t\right]
\end{array}
$$
\begin{equation}
\label{eq:4-FurtherBoundOfYnMinusYm}
\begin{array}{ll}
&\Dis+C\left(\E\left[ \sup\limits_{t\in\T}\left|(Y_{t\wedge\sigma_{l,q} }^m-L_{t\wedge\sigma_{l,q}})^-\right|^{p\over p-1}\right]\right)^{p-1\over p}\left(\E\left[|K_T^n|^p\right]\right)^{1\over p}\vspace{0.1cm}\\
&\Dis+C\left(\E\left[ \sup\limits_{t\in\T}\left|(Y_{t\wedge\sigma_{l,q} }^n-L_{t\wedge\sigma_{l,q}})^-\right|^{p\over p-1}\right]\right)^{p-1\over p}\left(\E\left[|K_T^m|^p\right]\right)^{1\over p}\\
&\Dis+2C\lambda\left(\E\left[\left(\int_{0}^{
\sigma_{l,q}}|Y^n_t-Y_t^m|^2{\rm d}t\right)^{p\over 2(p-1)}\right]\right)^{p-1\over p}\times \left(\E\left[\left(\int_{0}^T\left(|Z_t^n|+
|Z_t^m|\right)^2{\rm d}t\right)^{p\over 2}\right]\right)^{1\over p}.\vspace{0.1cm}
\end{array}
\end{equation}
Thus, in view of the definitions of $\tau_l$ and $\sigma_{l,q}$, \eqref{eq:4-YnIncreaseBounded}, \eqref{eq:4-ZnKnAnLeqEta}, \eqref{eq:4-IntegralOfftAndPsitleqlq} and \eqref{eq:4-UniformConvergenceOfYnMinusL}, by \eqref{eq:4-FurtherBoundOfYnMinusYm} and Lebesgue's dominated convergence theorem we can deduce that for each $l,q\geq 1$, as $n,m\To\infty$,
$$\E\left[\sup\limits_{t\in [0,T]} |Y_{t\wedge\sigma_{l,q}}^n
-Y_{t\wedge\sigma_{l,q}}^m|^2\right]\To 0,\vspace{0.1cm}$$
which implies that for each $l,q\geq 1$, as $n,m\To\infty$,
$$\sup\limits_{t\in [0,T]} |Y_{t\wedge\sigma_{l,q}}^n
-Y_{t\wedge\sigma_{l,q}}^m|\To 0\ {\rm in\ probability}\ \mathbb{P}.$$
And, by \eqref{eq:4-YnIncreaseBounded} and  \eqref{eq:4-StabilityOfTauSigma} we get that \begin{equation}
\label{eq:4-UniformConvergenceOfYnInProb}
\sup\limits_{t\in [0,T]} |Y_t^n
-Y_t|\To 0,\ \ {\rm as}\ n\To\infty.
\end{equation}
So, $Y_\cdot$ is a continuous process. Finally, in view of \eqref{eq:4-YnIncreaseBounded} and \eqref{eq:4-UniformConvergenceOfYnInProb}, Lebesgue's dominated convergence theorem yields that
\begin{equation}
\label{eq:4-ConvergenceOfYnInSp}
\lim\limits_{n\To\infty}\|Y_\cdot^n-Y_\cdot\|_{\s^p}^p =\lim\limits_{n\To\infty}\E\left[\sup\limits_{t\in [0,T]} |Y_t^n-Y_t|^p\right]=0.\vspace{0.3cm}
\end{equation}

{\bf Step 4.}\ We show the convergence of the sequence $\{Z_\cdot^n\}$ in $\M^p$. Note that \eqref{eq:4-DefinitionOfBarYBarZBarV} verifies \eqref{eq:4-BarY=BarV}. It follows from (i) of \cref{lem:2-Lemma1} with $t=0$ and $\tau=T$ that there exists a positive constant $C'>0$ such that for each $m,n\geq 1$,
$$
\begin{array}{lll}
&&\Dis\E\left[\left(\int_0^T|Z_t^n-Z_t^m|^2{\rm d}t\right)^{p\over 2}\right]\vspace{0.1cm}\\
&\leq & \Dis C'\E\left[\sup\limits_{t\in [0,T]}|Y_t^n-Y_t^m|^p+\sup\limits_{t\in [0,T]}\left[\left(\int_t^T (Y_s^n-Y_s^m)\left({\rm d}K_s^n-{\rm d}K_s^m\right)\right)^+\right]^{p\over 2}\right] \\
&&\Dis +C'\E\left[\sup\limits_{t\in [0,T]}\left[\left(\int_t^T (Y_s^n-Y_s^m)\left({\rm d}A_s^m-{\rm d}A_s^n\right)\right)^+\right]^{p\over 2}\right]\vspace{0.1cm}\\
&&\Dis + C'\E\left[\left(\int_{0}^T|Y^n_t-Y_t^m|
\left|g(t,Y_t^n,Z_t^n)-g(t,Y_t^m,Z_t^m)\right| {\rm d}t\right)^{p\over 2}\right].
\end{array}
$$
Then, it follows from H\"{o}lder's inequality and \eqref{eq:4-YnYmTimesAmAnLeq} that for each $m,n\geq 1$,
$$
\hspace*{-6cm}\E\left[\left(\int_0^T|Z_t^n-Z_t^m|^2{\rm d}t\right)^{p\over 2}\right]
$$
$$
\begin{array}{ll}
\leq & \Dis C'\E\left[\sup\limits_{t\in [0,T]}|Y_t^n-Y_t^m|^p\right]+\left.C'\left(\E\left[\sup\limits_{t\in [0,T]}|Y_t^n-Y_t^m|^p\right]\right)^{1\over 2}\right\{\left(\E\left[|K_T^n|^p\right]\right)^{1\over 2}\\
&\Dis+\left.\left(\E\left[|K_T^m|^p\right]\right)^{1\over 2}+\left(\E\left[\left(\int_{0}^T
\left(|g(t,Y_t^n,Z_t^n)|+|g(t,Y_t^m,Z_t^m)|\right) {\rm d}t\right)^{p}\right]\right)^{1\over 2}\right\},
\end{array}
$$
from which together with \eqref{eq:4-ZnKnAnLeqEta}, \eqref{eq:4-ConvergenceOfYnInSp} and \eqref{eq:4-WeaklyConvergenceOfZn} leads to the existence of a process $Z_\cdot\in \M^p$ such that
\begin{equation}
\label{eq:4-ConvergenceOfZnInMp}
\lim\limits_{n\To\infty}\|Z_\cdot^n-Z_\cdot\|_{\M^p}^p
=\lim\limits_{n\To\infty}\E\left[\left(\int_0^T|Z_t^n
-Z_t|^2{\rm d}t\right)^{p\over 2}\right]=0.\vspace{0.2cm}
\end{equation}

{\bf Step 5.}\ We show the desired convergence of the sequence $\{K_\cdot^n\}$. Let $\tau_l$ and $\sigma_{l,q}$ be the sequences of stopping times defined in Step 1. Since $g$ satisfies \ref{A:(HH)}, by \eqref{eq:4-BoundOfgYnZn}, \eqref{eq:4-IntegralOfftAndPsitleqlq}, \eqref{eq:4-ZnKnAnLeqEta}, \eqref{eq:4-UniformConvergenceOfYnInProb} and \eqref{eq:4-ConvergenceOfZnInMp} we can deduce the existence of a subsequence $\{n_j\}$ of $\{n\}$ such that for each $l,q\geq 1$,
$$\lim\limits_{j\To \infty}\int_0^{\sigma_{l,q}}|g(t,Y_t^{n_j},Z_t^{n_j})-
g(t,Y_t,Z_t)|{\rm d}t=0.$$
Then, in view of \eqref{eq:4-StabilityOfTauSigma}, we have
\begin{equation}
\label{eq:4-UniformConvergenceOfgYnZn}
\lim\limits_{j\To \infty}\sup\limits_{t\in\T}\left|\int_0^t
g(t,Y_t^{n_j},Z_t^{n_j}){\rm d}t-\int_0^t
g(t,Y_t,Z_t){\rm d}t\right|=0.
\end{equation}
Combining \eqref{eq:4-ConvergenceOfAnInSp}, \eqref{eq:4-UniformConvergenceOfYnInProb}, \eqref{eq:4-ConvergenceOfZnInMp} and \eqref{eq:4-UniformConvergenceOfgYnZn} leads to that $\ps$, for each $t\in\T$,
$$
K_t^{n_j}=Y_0^{n_j}-Y_t^{n_j}-\int_0^t
g(s,Y_s^{n_j},Z_s^{n_j}){\rm d}s-\int_0^t{\rm d}V_s-A_t^{n_j}+\int_0^tZ_s^{n_j}{\rm d}B_s
$$
tends to
$$
K_t:=Y_0-Y_t-\int_0^t
g(s,Y_s,Z_s){\rm d}s-\int_0^t{\rm d}V_s-A_t+\int_0^tZ_s\cdot {\rm d}B_s
$$
as $j\To \infty$ and that
\begin{equation}
\label{eq:4-UniformConvergenceOfKnPas}
\lim\limits_{j\To\infty}\sup\limits_{t\in\T}
|K_t^{n_j}-K_t|=0.\vspace{0.2cm}
\end{equation}
Hence, $K_\cdot$ is a continuous process. Furthermore, by Fatou's lemma with \eqref{eq:4-UniformConvergenceOfKnPas} and \eqref{eq:4-ZnKnAnLeqEta} we get that
$$
\Dis\E\left[\sup\limits_{t\in\T}|K_t|^p\right]
=\E\left[\lim\limits_{j\To\infty}
\sup\limits_{t\in\T}|K_t^{n_j}|^p\right]\leq \Dis \liminf\limits_{j\To\infty}
\E\left[\sup\limits_{t\in\T}|K_t^{n_j}|^p\right]\leq \Dis \sup\limits_{n\geq 1}\E\left[|K_T^{n}|^p\right]<+\infty.
$$
Thus, $K_\cdot\in\vcal^{+,p}$.

\vspace{0.2cm}

{\bf Step 6.}\ We show that $(Y_\cdot,Z_\cdot, K_\cdot, A_\cdot)$ is a desired solution of DRBSDE $(\xi,g+{\rm d}V,L,U)$ in the space $\s^p\times\M^p\times\vcal^{+,p}\times\vcal^{+,p}$. In fact, note that $(Y_\cdot,Z_\cdot,K_\cdot,A_\cdot)\in \s^p\times\M^p\times\vcal^{+,p}\times\vcal^{+,p}$ solves
$$
Y_t=\xi+\int_t^Tg(s,Y_s,Z_s){\rm d}s+\int_t^T{\rm d}V_s+\int_t^T{\rm d}K_s-\int_t^T{\rm d}A_s-\int_t^TZ_s\cdot {\rm d}B_s,\ \ t\in\T.
$$
It follows from Step 2 that $Y_t\geq L_t$ for each $t\in\T$, and then
$$
\int_0^T(Y_t-L_t){\rm d}K_t\geq 0.
$$
And, in view of \eqref{eq:4-UniformConvergenceOfYnInProb} and \eqref{eq:4-UniformConvergenceOfKnPas}, it follows from the definition of $K_\cdot^n$ that
$$
\int_0^T(Y_t-L_t){\rm d}K_t=\lim\limits_{j\To \infty}\int_0^T(Y_t^{n_j}-L_t){\rm d}K_t^{n_j}\leq 0.
$$
Consequently,
$$
\int_0^T(Y_t-L_t){\rm d}K_t=0.
$$
Furthermore, in view of the fact that $Y^n_\cdot\leq U_\cdot$ and $\int_0^T (U_t-Y^n_t)\ {\rm d}A^n_t=0$ for each $n\geq 1$, from \eqref{eq:4-ConvergenceOfYnInSp} and \eqref{eq:4-ConvergenceOfAnInSp} we derive that $Y_t\leq U_t$ for each $n\geq 1$, and
$$
\int_0^T (U_t-Y_t)\ {\rm d}A_t=\Lim \int_0^T (U_t-Y^n_t)\ {\rm d}A^n_t=0.
$$
Finally, let us verify that ${\rm d}K\bot{\rm d}A$. In fact, for each $n\geq 1$, we define the following progressively measurable set
$$D_n:=\{(\omega,t)\subset \Omega\times\T:\ Y^n_t(\omega)\geq L_t(\omega)\}.$$
It then follows from the definition of $K^n_\cdot$ that for each $n\geq 1$,
$$\E\left[\int_0^T \mathbbm{1}_{D_n} {\rm d}K^n_t\right]=0,\vspace{-0.1cm}$$
and, in view of $\int_0^T (U_t-Y^n_t){\rm d}A^n_t=0$,
$$
\E\left[\int_0^T \mathbbm{1}_{D_n^c}{\rm d}A^n_t\right]=\E\left[\int_0^T \mathbbm{1}_{\{Y^n_t<L_t\leq U_t\}}
|U_t-Y^n_t|^{-1}
(U_t-Y^n_t)\ {\rm d}A^n_t\right]=0.
$$
Thus, noticing that $D_n\subset D_{n+1}$ for each $n\geq1$ due to $Y^n_\cdot\leq Y^{n+1}_\cdot$, by \eqref{eq:4-UniformConvergenceOfKnPas} and \eqref{eq:4-ConvergenceOfAnInSp} we have
$$
\E\left[\int_0^T\mathbbm{1}_{\cup D_n} {\rm d}K_t\right]=\lim\limits_{j\To\infty} \E\left[\int_0^T
\mathbbm{1}_{D_{n_j}} {\rm d}K^{n_j}_t\right]=0
$$
and
$$
\E\left[\int_0^T\mathbbm{1}_{\cap D_n^c} {\rm d}A_t\right]=\Lim\E\left[\int_0^T
\mathbbm{1}_{D_n^c}{\rm d}A^n_t\right]=0.\vspace{0.2cm}
$$
Hence, ${\rm d}K\bot{\rm d}A$. The proof of \cref{pro:4-Penalization} is then complete.\vspace{0.3cm}

{\bf Complementary of the details for the proof of \cref{pro:5-Approximation}}. Now, we will detail the proof of steps 1-3 after the inequality \eqref{eq:5-FtandPhitLeqlq}. \vspace{0.2cm}

{\bf Step 1.}\ We show the convergence of the sequence $\{Y_\cdot^n\}$ in $\s^p$. For each $n,m\geq 1$, observe that
\begin{equation}
\label{eq:5-DefinitionOfBarYZV}
\begin{array}{lll}
\Dis (\bar Y_\cdot,\bar Z_\cdot,\bar V_\cdot)&
:=&\Dis (Y_\cdot^n-Y_\cdot^m,Z_\cdot^n-Z_\cdot^m,\\
&&\Dis \ \ \int_0^\cdot \left(g_n(s,Y_s^n,Z_s^n)-g_m(s,Y_s^m,Z_s^m)\right){\rm d}s+\left(K_\cdot^n-K_\cdot^m\right)
+\left(A_\cdot^m-A_\cdot^n\right) )
\end{array}
\end{equation}
solves equation \eqref{eq:4-BarY=BarV}. It then follows from (ii) of \cref{lem:2-Lemma1} with $p=2$, $t=0$ and $\tau=\sigma_{l,q}$ that there exists a constant $C>0$ such that for each $n,m,l,q\geq 1$,
$$
\hspace*{-4.5cm}\Dis\E\left[\sup\limits_{t\in [0,T]} |Y_{t\wedge\sigma_{l,q}}^n
-Y_{t\wedge\sigma_{l,q}}^m|^2\right]
$$
\begin{equation}
\label{eq:5-FirstBoundOfYnMinusYm}
\begin{array}{ll}
\leq &\Dis C\E\left[|Y_{\sigma_{l,q}}^n-Y_{\sigma_{l,q}}^m|^2 +\sup\limits_{t\in [0,T]}\left(\int_{t\wedge \sigma_{l,q}}^{\sigma_{l,q}}(Y^n_s-Y_s^m)
\left({\rm d}K_s^n-{\rm d}K_s^m\right) \right)^+\right.\vspace{0.1cm}\\
&\hspace{1cm}\Dis +\sup\limits_{t\in [0,T]}\left(\int_{t\wedge \sigma_{l,q}}^{\sigma_{l,q}}(Y^n_s-Y_s^m)
\left({\rm d}A_s^m-{\rm d}A_s^n\right) \right)^+\vspace{0.2cm}\\
&\hspace{1cm}\Dis +\left. \int_{0}^{\sigma_{l,q}}|Y^n_t-Y_t^m|
\left|g_n(t,Y_t^n,Z_t^n)-g_m(t,Y_t^m,Z_t^m)\right| {\rm d}t\right].
\end{array}
\end{equation}
Furthermore, in view of the facts that $L_t\leq Y_t^n\leq U_t$ for each $t\in\T$ and $n\geq 1$, and $\int_0^T(Y_t^n-L_t){\rm d}K_t^n=\int_0^T( U_t-Y_t^n){\rm d}A_t^n=0$ for each $n\geq 1$, we have that for each $t\in \T$ and $l,q,m,n\geq 1$,
\begin{equation}
\label{eq:5-YnYmKnKmLeq0}
\begin{array}{lll}
\Dis \int_{t\wedge \sigma_{l,q}}^{\sigma_{l,q}}(Y^n_s-Y_s^m)
\left({\rm d}K_s^n-{\rm d}K_s^m\right)&
= & \Dis \int_{t\wedge \sigma_{l,q}}^{\sigma_{l,q}}\left[(Y_s^n-L_s)
-(Y_s^m-L_s)\right]\left({\rm d}K_s^n-{\rm d}K_s^m\right)\vspace{0.2cm}\\
&= & \Dis -\int_{t\wedge \sigma_{l,q}}^{\sigma_{l,q}}(Y_s^n-L_s)
{\rm d}K_s^m-\int_{t\wedge \sigma_{l,q}}^{\sigma_{l,q}}(Y_s^m-L_s)
{\rm d}K_s^n\\
&\leq & 0
\end{array}
\end{equation}
and
\begin{equation}
\label{eq:5-YnYmAnAmLeq0}
\begin{array}{lll}
\Dis \int_{t\wedge \sigma_{l,q}}^{\sigma_{l,q}}(Y^n_s-Y_s^m)
\left({\rm d}A_s^m-{\rm d}A_s^n\right)&
= & \Dis \int_{t\wedge \sigma_{l,q}}^{\sigma_{l,q}}\left[(U_s-Y_s^m)
-(U_s-Y_s^n)\right]\left({\rm d}A_s^m-{\rm d}A_s^n\right)\vspace{0.2cm}\\
&= & \Dis -\int_{t\wedge \sigma_{l,q}}^{\sigma_{l,q}}(U_s-Y_s^m)
{\rm d}A_s^n-\int_{t\wedge \sigma_{l,q}}^{\sigma_{l,q}}(U_s-Y_s^n)
{\rm d}A_s^m\\
&\leq & 0.
\end{array}
\end{equation}
Combining \eqref{eq:5-BoundOfgnYnZnWithStoppingtime}, \eqref{eq:5-FirstBoundOfYnMinusYm},  \eqref{eq:5-YnYmKnKmLeq0} and \eqref{eq:5-YnYmAnAmLeq0} together with H\"{o}lder's inequality leads to that
\begin{equation}
\label{eq:5-SecondBoundOfYnMinusYm}
\begin{array}{ll}
&\Dis\E\left[\sup\limits_{t\in [0,T]} |Y_{t\wedge\sigma_{l,q}}^n
-Y_{t\wedge\sigma_{l,q}}^m|^2\right]\vspace{0.1cm}\\
\leq &\Dis C\E\left[|Y_{\sigma_{l,q}}^n-Y_{\sigma_{l,q}}^m|^2 +2\int_0^T |Y^n_t-Y_t^m|\left(\mathbbm{1}_{t\leq \tau_l}f_t+\mathbbm{1}_{t\leq\sigma_{l,q}}
\psi_t(l)\right){\rm d}t\right]\vspace{0.1cm}\\
&\Dis+2C\lambda\left(\E\left[\left(\int_{0}^{
\sigma_{l,q}}|Y^n_t-Y_t^m|^2{\rm d}t\right)^{p\over 2(p-1)}\right]\right)^{p-1\over p}\times \left(\E\left[\left(\int_{0}^T\left(|Z_t^n|+
|Z_t^m|\right)^2{\rm d}t\right)^{p\over 2}\right]\right)^{1\over p}.
\end{array}
\end{equation}
Thus, in view of the definitions of $\tau_l$ and $\sigma_{l,q}$, \eqref{eq:5-YnIncreaseBounded}, \eqref{eq:5-BoundofZnKnAnGn} and \eqref{eq:5-FtandPhitLeqlq}, it follows from \eqref{eq:5-SecondBoundOfYnMinusYm} and Lebesgue's dominated convergence theorem that for each $l,q\geq 1$, as $n,m\To\infty$,
$$\E\left[\sup\limits_{t\in [0,T]} |Y_{t\wedge\sigma_{l,q}}^n
-Y_{t\wedge\sigma_{l,q}}^m|^2\right]\To 0,$$
which implies that for each $l,q\geq 1$, as $n,m\To\infty$,
$$\sup\limits_{t\in [0,T]} |Y_{t\wedge\sigma_{l,q}}^n
-Y_{t\wedge\sigma_{l,q}}^m|\To 0\ {\rm in\ probability}\ \mathbb{P}.$$
And, in view of \eqref{eq:5-StabilityOftauandSigma} and \eqref{eq:5-YnIncreaseBounded}, we have
\begin{equation}
\label{eq:5-UniformConvergenceOfYnPas}
\sup\limits_{t\in [0,T]} |Y_t^n
-Y_t|\To 0,\ \ {\rm as}\ n\To\infty.
\end{equation}
So, $Y_\cdot$ is a continuous process. Finally, in view of \eqref{eq:5-UniformConvergenceOfYnPas} and \eqref{eq:5-YnIncreaseBounded}, Lebesgue's dominated convergence theorem yields that
\begin{equation}
\label{eq:5-ConvergenceOfYnInSp}
\lim\limits_{n\To\infty}\|Y_\cdot^n-Y_\cdot\|_{\s^p}^p =\lim\limits_{n\To\infty}\E\left[\sup\limits_{t\in [0,T]} |Y_t^n-Y_t|^p\right]=0.\vspace{0.2cm}
\end{equation}

{\bf Step 2.}\ We show the convergence of the sequence $\{Z_\cdot^n\}$ in $\M^p$. Note that \eqref{eq:5-DefinitionOfBarYZV} verifies \eqref{eq:4-BarY=BarV}. It follows from (i) of \cref{lem:2-Lemma1} with $t=0$ and $\tau=T$ that there exists a positive constant $C'>0$ such that for each $m,n\geq 1$,
$$
\begin{array}{ll}
&\Dis\E\left[\left(\int_0^T|Z_t^n-Z_t^m|^2{\rm d}t\right)^{p\over 2}\right]\vspace{0.1cm}\\
\leq & \Dis C'\E\left[\sup\limits_{t\in [0,T]}|Y_t^n-Y_t^m|^p+\sup\limits_{t\in [0,T]}\left[\left(\int_t^T (Y_s^n-Y_s^m)\left({\rm d}K_s^n-{\rm d}K_s^m\right)\right)^+\right]^{p\over 2}\right]\\
&\Dis +C'\E\left[\sup\limits_{t\in [0,T]}\left[\left(\int_t^T (Y_s^n-Y_s^m)\left({\rm d}A_s^m-{\rm d}A_s^n\right)\right)^+\right]^{p\over 2}\right]\vspace{0.1cm}\\
&\Dis + C'\E\left[\left(\int_{0}^T|Y^n_t-Y_t^m|
\left|g_n(t,Y_t^n,Z_t^n)-g_m(t,Y_t^m,Z_t^m)\right| {\rm d}t\right)^{p\over 2}\right].
\end{array}
$$
Then, in view of \eqref{eq:5-YnYmKnKmLeq0} and \eqref{eq:5-YnYmAnAmLeq0}, it follows from H\"{o}lder's inequality that for each $m,n\geq 1$,
$$
\begin{array}{lll}
\Dis\E\left[\left(\int_0^T|Z_t^n-Z_t^m|^2{\rm d}t\right)^{p\over 2}\right]
&\leq &\Dis C'\E\left[\sup\limits_{t\in [0,T]}|Y_t^n-Y_t^m|^p\right]+C'\left(\E\left[\sup\limits_{t\in [0,T]}|Y_t^n-Y_t^m|^p\right]\right)^{1\over 2}\\
&&\Dis\ \ \ \ \  \cdot\left(\E\left[\left(\int_{0}^T
\left(|g_n(t,Y_t^n,Z_t^n)|+|g_m(t,Y_t^m,Z_t^m)|\right) {\rm d}t\right)^{p}\right]\right)^{1\over 2},
\end{array}
$$
from which together with \eqref{eq:5-ConvergenceOfYnInSp} and \eqref{eq:5-BoundofZnKnAnGn} yields the existence of a process $Z_\cdot\in\M^p$ such that
\begin{equation}
\label{eq:5-ConvergenceOfZnInMp}
\lim\limits_{n\To\infty}\|Z_\cdot^n-Z_\cdot\|_{\M^p}^p
=\lim\limits_{n\To\infty}\E\left[\left(\int_0^T|Z_t^n
-Z_t|^2{\rm d}t\right)^{p\over 2}\right]=0.\vspace{0.2cm}
\end{equation}

{\bf Step 3.}\ We show that $(Y_\cdot,Z_\cdot,K_\cdot,A_\cdot)$ is a solution of DRBSDE $(\xi,g+{\rm d}V,L,U)$ in the space $\s^p\times\M^p\times\vcal^{+,p}\times\vcal^{+,p}$. By \eqref{eq:5-LocalUniformConvergenceofgnFromleft}-\eqref{eq:5-YnIncreaseBounded}, \eqref{eq:5-ConvergenceOfZnInMp} and \eqref{eq:5-BoundOfgnYnZnWithStoppingtime}-\eqref{eq:5-FtandPhitLeqlq} we derive the existence of a subsequence $\{n_j\}$ of $\{n\}$ such that for each $l,q\geq 1$,
$$\lim\limits_{j\To \infty}\int_0^{\sigma_{l,q}}|g_{n_j}(t,Y_t^{n_j},Z_t^{n_j})-
g(t,Y_t,Z_t)|{\rm d}t=0.$$
Then, in view of \eqref{eq:5-StabilityOftauandSigma}, we have
\begin{equation}
\label{eq:5-ConvergenceOfgn}
\lim\limits_{j\To \infty}\sup\limits_{t\in\T}\left|\int_0^t
g_{n_j}(t,Y_t^{n_j},Z_t^{n_j}){\rm d}t-\int_0^t
g(t,Y_t,Z_t){\rm d}t\right|=0.\vspace{0.1cm}
\end{equation}
Combining \eqref{eq:5-ConvergenceOfKnandAnInSp}, \eqref{eq:5-UniformConvergenceOfYnPas}, \eqref{eq:5-ConvergenceOfZnInMp} and \eqref{eq:5-ConvergenceOfgn} leads to that\vspace{0.1cm}
$$
Y_t=\xi+\int_t^Tg(s,Y_s,Z_s){\rm d}s+\int_t^T{\rm d}V_s+\int_t^T{\rm d}K_s-\int_t^T{\rm d}A_s-\int_t^TZ_s\cdot {\rm d}B_s,\ \ t\in\T.\vspace{0.1cm}
$$
Since $L_t\leq Y_t^n\leq U_t,\ n\geq 1$ and $Y_t^n\uparrow Y_t$ for each $t\in\T$, we know that $L_t \leq Y_t\leq U_t$ for each $t\in\T$. Furthermore, in view of \eqref{eq:5-ConvergenceOfYnInSp} and \eqref{eq:5-ConvergenceOfKnandAnInSp}, we have
$$
\int_0^T(Y_t-L_t){\rm d}K_t=\Lim\int_0^T(Y_t^{n}-L_t){\rm d}K_t^n=0
$$
and
$$
\int_0^T(U_t-Y_t){\rm d}A_t=\Lim\int_0^T(U_t-Y_t^{n}){\rm d}A_t^n=0.\vspace{0.2cm}
$$
Finally, let us show that ${\rm d}K\bot {\rm d}A$. In fact, for each $n\geq 1$, since ${\rm d}K^n\bot {\rm d}A^n$, we know that there exists a progressively measurable set
$D_n\subset \Omega\times\T$ such that
$$
\E\left[\int_0^T \mathbbm{1}_{D_n} {\rm d}K^n_t\right]=\E\left[\int_0^T \mathbbm{1}_{D_n^c} {\rm d}A^n_t\right]=0.
$$
Then, in view of \eqref{eq:5-ConvergenceOfKnandAnInSp} and the fact that ${\rm d}K\leq {\rm d}K^n$ for each $n\geq 1$, we have\vspace{0.1cm}
$$
0\leq \E\left[\int_0^T\mathbbm{1}_{\cup D_n} {\rm d}K_t\right]\leq \sum\limits_{n=1}^\infty \E\left[\int_0^T \mathbbm{1}_{D_n} {\rm d}K_t\right]\leq \sum\limits_{n=1}^\infty \E\left[\int_0^T\mathbbm{1}_{D_n} {\rm d}K^n_t\right]=0\vspace{0.1cm}
$$
and\vspace{0.1cm}
$$
0\leq \E\left[\int_0^T\mathbbm{1}_{\cap D_n^c} {\rm d}A_t\right]=\lim\limits_{m\To\infty}\E\left[\int_0^T
\mathbbm{1}_{\cap D_n^c}{\rm d}A^m_t\right]\leq \lim\limits_{m\To\infty}\E\left[\int_0^T
\mathbbm{1}_{D_m^c}{\rm d}A^m_t\right]=0.\vspace{0.2cm}
$$
Hence, ${\rm d}K\bot{\rm d}A$. The proof of \cref{pro:5-Approximation} is then complete.


\setlength{\bibsep}{2pt}

\end{document}